\documentclass[a4paper,10pt,leqno]{amsart}
        \title{Dualizability and index of subfactors}
       \author{Arthur Bartels}
      \address{Mathematisches Institut\\ 
               Westf\"alische Wilhelms-Universit\"at M\"unster\\
               D-48149 M\"unster, Deutschland}
        \email{bartelsa@math.uni-muenster.de}
      \urladdr{http://www.math.uni-muenster.de/u/bartelsa}
       \author{Christopher L. Douglas} 
      \address{Mathematical Institute\\ University of Oxford\\ Oxford OX1 3LB, United Kingdom}
        \email{cdouglas@maths.ox.ac.uk}
      \urladdr{http://people.maths.ox.ac.uk/cdouglas}
       \author{Andr{\'e} Henriques}
      \address{Mathematisch Instituut\\
               Universiteit Utrecht\\
               3508 TA Utrecht, The Netherlands}
        \email{a.g.henriques@uu.nl}
      \urladdr{http://www.staff.science.uu.nl/\!\raisebox{-1mm}{~}\!henri105}

\usepackage[pdftex, colorlinks=true, linkcolor=black, citecolor=black, urlcolor=black]{hyperref}
\usepackage{enumerate,amssymb}
\usepackage[arrow,curve,matrix,tips,2cell]{xy}
  \SelectTips{eu}{10} \UseTips
  \UseAllTwocells
\usepackage{tikz}                \usetikzlibrary{calc}               \usetikzlibrary{matrix}
\usepackage{pdfcolmk}
\usepackage{calc}
\usepackage{multirow}
\usepackage{stmaryrd}

  \newcommand{\IC}{\mathbb{C}}

  \newcommand{\IR}{\mathbb{R}}

  \newcommand{\bfB}{{\mathbf B}}

  \newcommand{\bfs}{{\mathbf s}}

  \newcounter{commentcounter}

    \definecolor{AHcolor}{rgb}{0.5,0.0,0.5}   
    \definecolor{CDcolor}{rgb}{0.7,0.0,0.3}   
    \definecolor{ABcolor}{rgb}{0.2,0.8,0.2}   

  \newcommand{\tikzmath}[2][]
     {\vcenter{\hbox{\begin{tikzpicture}[#1]#2
                     \end{tikzpicture}}}
     }

  \newcommand{\squarescale}{.07}

  \definecolor{spacecolor}{gray}{.7}
  \definecolor{antispacecolor}{gray}{.45}

  \theoremstyle{plain}
  \newtheorem{theorem}{Theorem}[section]
  
  \newtheorem{lemma}[theorem]{Lemma}
  \newtheorem{corollary}[theorem]{Corollary}
  \newtheorem{proposition}[theorem]{Proposition}
  \newtheorem{conjecture}[theorem]{Conjecture}
  
  \newtheorem*{theorem*}{Theorem}

  \theoremstyle{definition}
  \newtheorem{definition}[theorem]{Definition}

  \newtheorem{question}[theorem]{Question}

  \theoremstyle{remark}
  \newtheorem{remark}[theorem]{Remark}
  \newtheorem{example}[theorem]{Example}
  \newtheorem{warning}[theorem]{Warning}

  \makeatletter\let\c@equation=\c@theorem\makeatother

     {\end{list}}

  \DeclareMathOperator{\Ind}{Ind}

  \newcommand{\alg}{{\mathit{alg}}}
  
  \newcommand{\op}{{\mathit{op}}}

  \newcommand{\x}{{\times}}
  \newcommand{\ox}{{\otimes}}

  \newcommand{\nid}{\noindent}

\begin{document}

	\begin{abstract}
	In this paper, we develop the theory of bimodules over von 
        Neumann algebras, with an emphasis on categorical aspects.
	We  clarify the relationship between dualizability and finite index.
        We also show that, for von Neumann algebras with finite-dimensional 
        centers, the Haagerup $L^2$-space and Connes fusion are functorial 
        with respect to	homorphisms of finite index.
	Along the way, we describe a string  diagram notation for maps
        between bimodules that are not necessarily bilinear.       
	\end{abstract}

\maketitle


\setcounter{tocdepth}{1}
\tableofcontents

\newcommand{\comment}[1]{.}

\section{Introduction}

The operation $({}_AH_B,{}_BK_C)\mapsto {}_AH\boxtimes_BK_C$ of Connes fusion 
is an associative product on bimodules between von Neumann 
algebras \cite{Connes(Non-commutative-geometry), Popa(Correspondences), Sauvageot(Sur-le-produit-tensoriel-relatif), Wassermann(Operator-algebras-and-conformal-field-theory)}.
It behaves formally like a tensor product, 
but its construction is somewhat involved and relies heavily on the notion of 
non-commutative $L^2$-space \cite{Haagerup(1975standard-form), Kosaki(PhD-thesis), 
                  Yamagami(Algebraic-aspects-in-modular-theory)}.
Connes fusion is designed so as to have the $L^2$-space as its identity: 
${}_AL^2A\boxtimes_A H_B\cong{}_AH_B\cong{}_AH\boxtimes_BL^2B_B$.
Altogether, von Neumann algebras, their bimodules, and bimodule intertwiners 
form a symmetric monoidal 
bicategory \cite{Landsman(Bicategories-of-operator-algebras)}.
As in any bicategory, one can talk about a morphism being 
dualizable\footnote{As written, equation \eqref{eq: intro R S} corresponds to the 
 notion of \emph{left} dualizability,  but since our bicategory has a 
 $*$-structure, there is no difference between left and right dualizability.} 
\cite{MacLane(Categories-for-the-working-mathematician), 
  Street(Low-dimensional-topology-and-higher-order-categories-I)}:
a bimodule ${}_AH_B$ is called dualizable, with dual ${}_B\bar H_A$, 
if it comes equipped with maps
\begin{equation}\label{eq: intro R S}
R^*:{}_AH\boxtimes_B \bar H_A \rightarrow {}_AL^2(A)_A
\qquad\quad
S:{}_BL^2(B)_B \rightarrow  {}_B\bar H\boxtimes_A H_B
\end{equation}
subject to the duality equations $(R^*\otimes 1)(1\otimes S)=1$, 
$(1\otimes R^*)(S\otimes 1)=1$.
The dual bimodule ${}_B\bar H_A$ is well defined up to unique isomorphism.
In fact, under suitable normalization conditions on the duality maps $R^*$ and $S$, 
the dual bimodule is well defined up to  unique \emph{unitary} isomorphism.
If $A$ and $B$ are factors one can then define the \emph{statistical dimension} of 
${}_AH_B$ as $R^* R = S^*S$ \cite{Longo-Roberts(A-theory-of-dimension)}. 

A subfactor $N \subset M$ has an invariant called the index $[M:N]\in\IR_{\ge1}\cup\{\infty\}$ \cite{Kosaki(Extension-of-Jones-index-to-arbitrary-factors), Kosaki-Longo(A-remark-on-the-minimal-index)}, 
and this index is finite if and only if the bimodule ${}_NL^2M_M$ is dualizable.  When that bimodule is dualizable, the index may be defined as the square of the statistical dimension of ${}_NL^2M_M$.  We show that this definition agrees with the traditional notion of index, by comparing the squared statistical dimension with the optimal bound of a Pimsner--Popa inequality for the subfactor \cite{Hiai(Minimizing-indices), Longo(Index-of-subfactors-and-statistics-of-quantum-fields-I), Pimsner-Popa}.
                    
Given two von Neumann algebras $A$ and $B$ that have finite-dimensional centers (in other words are finite direct sums of factors), we call a homomorphism $f:A\to B$ \emph{finite} if the bimodule ${}_AL^2B_B$ is dualizable.  Restricting attention to these finite homomorphisms makes the $L^2$ construction functorial:

\begin{theorem*}
The assignment 
\[A\mapsto L^2(A)\] 
is a functor from the category 
\[
\tikzmath{
\node[anchor=north west] at (.6,1.02) {$\bigg\{$};
\node[anchor=north west] at (1,1) {objects:};
\node[anchor=north west] at (3,1) {\parbox[t]{8.2cm}{von Neumann
algebras with finite-dimensional center}};
\node[anchor=north west] at (1,0.5) {morphisms:};
\node[anchor=north west] at (3,0.5) {\parbox[t]{8.2cm}{finite homomorphisms}};
}
\]
to the category of Hilbert spaces and bounded linear maps.
\end{theorem*}

\noindent 
We conjecture that this functor in fact extends to the category of all von Neumann algebras and finite homomorphisms.

The Connes fusion $H\boxtimes_AK$ is certainly functorial in $H$ and $K$.
We show that it is moreover simultaneously functorial in the three variables $H$, $K$ and $A$:

\begin{theorem*}
The assignment 
\[
(H,\,A,\,K)\;\mapsto\; H\boxtimes_AK
\]
is a functor from the category
\[
\tikzmath{
\node[anchor=north west] at (.6,1.45)
{$\begin{cases}\\\\\\\phantom{\Bigg|}\\\end{cases}$};
\node[anchor=north west] at (1,1.5) {objects:};
\node[anchor=north west] at (3,1.5) {\parbox[t]{8.2cm}{triples $(H,A,K)$
where $A$ is a von Neumann algebra with finite-dimensional center, $H$
is a right $A$-module, and $K$ is a left $A$-module}};
\node[anchor=north west] at (1,0) {morphisms:};
\node[anchor=north west] at (3,0) {\parbox[t]{8.2cm}{triples $(h, \alpha, k)$ where $\alpha$ is a finite
homomorphism $A_1 \to A_2$, $h$ is a module map $H_1 \to H_2$, and $k$ is a module
map $K_1 \to K_2$}};
}
\]
to the category of Hilbert spaces and bounded linear maps.
\end{theorem*}

Note that our techniques and results all apply equally well to type {\it I}, {\it II}, 
        and {\it III} von Neumann algebras.

\subsection*{Outlook}

Our motivation for studying von Neumann algebras and Connes fusion comes from their relationship to quantum field theory
and to the Stolz--Teichner program on elliptic cohomology. 
Their relevance to the former is evident in Wassermann's work~\cite{Wassermann(Operator-algebras-and-conformal-field-theory)}
where Connes fusion is used to model the fusion rules of superselection sectors of the chiral Wess--Zumino--Witten conformal field theory with gauge group $\mathit{SU}(N)$.
Moreover, for those theories, Wassermann computes the Connes fusion explicitly, and recovers the Verlinde formulas.

The ongoing program of Stolz and Teichner aims to construct
elliptic cohomology using \emph{local} quantum field 
theories~\cite{Stolz-Teichner(2004what-is), 
  Stolz-Teichner(SUSY-field-theories-and-cohomology)}.
Motivated by~\cite{Wassermann(Operator-algebras-and-conformal-field-theory)},
Stolz and Teichner proposed the use of Connes fusion in their
description local quantum field theories.  
Moreover, they asked  asked whether there 
exits an interesting $3$-category that deloops the bicategory of 
von Neumann algebras, their bimodules, and bimodule intertwiners.
Here, interesting can be taken to mean that the $3$-category should 
have many dualizable objects:
as a consequence of the cobordism 
hypothesis~\cite{Lurie(On-classification-TFT)} every
dualizable object corresponds to a $3$-dimensional 
local quantum field theory.
The present paper arose as a byproduct of our ongoing construction of
such a $3$-category using conformal 
nets~\cite{BDH(nets),
             BDH(1*1)}.

The construction of our $3$-category of conformal nets depends 
very much on the theory of von Neumann algebras, 
in particular non-commutative $L^2$-spaces and the index for subfactors. 
We hope that the present treatment of these topics will help make our 
future papers more accessible for readers who do not have
a strong background in von Neumann algebras. 
This paper is \emph{not} 
a complete survey of the index for subfactors;
we mostly only discuss what we will need later on.
Many of the results in this paper are surely well-known to experts;
for example the identification of the index, 
defined using statistical dimension, 
with what we later call the minimal index, is no doubt known, 
but we are not aware of a reference.

\subsection*{Outline}
Our new graphical notation is described in section~\ref{sec: Preliminaries}, along with preliminaries concerning von Neumann algebras and Haagerup's $L^2$-space. We emphasize the fact 
that it is not necessary to chose a state $\phi:A\to \mathbb C$ 
in order to define $L^2(A)$ \cite{Kosaki(PhD-thesis)}.
In section~\ref{sec: Connes fusion}, we discuss Connes fusion and 
some of its elementary properties.
In section~\ref{sec: Dualizability}, we investigate the 
concept of dualizable bimodules.
We prove that the endomorphism algebra $\mathrm{End}({}_AH_B)$ 
of a dualizable bimodule is finite-dimensional and is equipped with a canonical trace.
Moreover, we show the dual is well defined up to unique unitary isomorphism.
In section~\ref{sec: statistical dimension and minimal index}, 
we define the statistical dimension of a dualizable bimodule
and introduce
the categorical definition of the  index of 
a subfactor, namely $[M:N]:=\dim({}_NL^2M_M)^2$.
In section~\ref{sec: functoriality of L2 and of Connes fusion}, 
we present our new results: the functoriality of $L^2$ and of Connes fusion.
Finally, in section~\ref{sec: Pimsner-Popa inequality}, 
we use the Pimsner--Popa inequality to show that the 
categorical definition of the 
index agrees with other definitions
\cite{Hiai(Minimizing-indices), 
Kosaki(Extension-of-Jones-index-to-arbitrary-factors), 
Kosaki(Type-III-factors-and-index-theory), Longo(Index-of-subfactors-and-statistics-of-quantum-fields-I), Pimsner-Popa}.
We end the paper with some useful inequalities for the index.

\subsection*{Acknowledgments}
We are grateful to Dmitri Pavlov for help with the proof of Proposition \ref{Prop. 8.3}, and to Hideki Kosaki for making available to us a digital 
version of his book \cite{Kosaki(Type-III-factors-and-index-theory)} and for useful email discussions.

AB was supported by the SFB 878 in M\"unster.

\section{Preliminaries}\label{sec: Preliminaries}

\subsection*{Von Neumann algebras}
Given a complex Hilbert space $H$, let $\bfB(H)$ denote its algebra of bounded operators. 
The ultraweak topology on $\bfB(H)$ is the topology of pointwise convergence with respect to the pairing with its predual, the trace class operators.

\begin{definition}
A von Neumann algebra is a topological $*$-algebra that is embeddable as a closed subalgebra of $\bfB(H)$ with respect to the ultraweak topology.
\end{definition}

\begin{definition}
Let $A$ be a von Neumann algebra.
A left (right) $A$-module is a Hilbert space $H$ equipped with a continuous homomorphism from $A$ (respectively $A^\op$) to $\bfB(H)$.
We will use the notation ${}_AH$ (respectively $H_A$) to denote the fact that $H$ is a left (right) $A$-module.
\end{definition}

The main distinguishing feature of the representation theory of von 
Neumann algebras is the following:

\begin{proposition}[{\cite[Remark~2.1.3.~(iii)]{Jones-Sunder(Intro-to-subfactors)}}]
  \label{prop:modules-are-summands}
  Let $A$ be a von Neumann algebra and let $H$ and $K$ be two faithful 
  left $A$-modules.
  Then $H\otimes \ell^2\cong K\otimes \ell^2$.
  In particular, any $A$-module is isomorphic to a direct summand of $H\otimes \ell^2$. \hfill $\square$
\end{proposition}

\nid If the Hilbert spaces $H$ and $K$ in this proposition are separable, then $\ell^2$ can be taken to mean $\ell^2(\mathbb N)$.
Otherwise, the proposition is true for $\ell^2=\ell^2(X)$, for $X$ some set of sufficiently large cardinality.

The spatial tensor product $A_1\bar\otimes A_2$ of von Neumann algebras 
$A_i\subset \bfB(H_i)$ is the closure in $\bfB(H_1\otimes H_2)$ of 
the algebraic tensor product $A_1\otimes_\alg A_2$; by the above proposition, this closure is independent of the choices of Hilbert spaces $H_1$ and $H_2$.  The spatial tensor product is a symmetric monoidal structure on the category of von Neumann algebras.

\subsection*{The Haagerup $L^2$-space} 

Given a von Neumann algebra $A$, the space of continuous linear functionals $A\to \IC$ forms a Banach space $A_*=L^1(A)$ called the predual of $A$.
It is equipped with two commuting $A$ actions given by $(a\phi b)(x):=\phi(bxa)$ and a cone $L^1_+(A):=\{\phi\in A_*\,|\,\phi(x)\ge0\,\, \forall x\in A_+\}$.
Here, $A_+:=\{a^*a\,|\,a\in A\}$ is the set of positive elements of $A$.

The Haagerup $L^2$-space of $A$ is an $A$-$A$-bimodule that is canonically associated to $A$.
It is denoted $L^2(A)$ and its construction does not depend on any 
choices \cite{Kosaki(PhD-thesis)}.
It is the completion of 
$$\bigoplus_{\phi\in L^1_+(A)} \IC\textstyle\sqrt{\phi}$$
with respect to some pre-inner product. 
We will provide more details of the construction of $L^2(A)$ at the beginning of section~\ref{sec: functoriality of L2 and of Connes fusion}. 

\begin{remark}\label{rem: algebra L_*(A) -- PRE}
At this point, $\sqrt{\phi}\in L^2A$ should be treated as a formal symbol. 
However, there exists a natural $*$-algebra structure on $\bigoplus_p L^pA$
in which $\sqrt{\phi}$ is the (unique positive) square root of $\phi\in L^1A$ --- see Remark \ref{rem: algebra L_*(A)} for further details.
As a consequence of that characterization, we learn that
\begin{equation}\label{eq: u sqrt(phi)u^*= sqrt(u phi u^*)}
u\sqrt{\phi}\,u^*=\sqrt{u\phi u^*}
\end{equation}
for every $\phi\in L^1_+(A)$ and every unitary $u\in A$.
\end{remark}

\begin{remark}\label{rem: L^2(A^op)}
There is an isomorphism 
$L^2(A) \cong L^2(A^\op)$ under which the left action of $A$ on $L^2A$
is equal to the right action of $A^\op$ on $L^2(A^\op)$, and the right action of $A$ on $L^2A$
is equal to the left action of $A^\op$ on $L^2(A^\op)$.
\end{remark}

The $L^2$ construction is compatible with direct sums, in the sense that $L^2(A\oplus B)=L^2(A)\oplus L^2(B)$.  This is a corollary of the relationship expressed in the following lemma, between the $L^2$-space construction and the operation of taking the corner algebra $pAp$ associated to a projection $p\in A$.

\begin{lemma}[{\cite[Lemma~2.6]{Haagerup(1975standard-form)}}]\label{lem: L^2(pAp) = pL^2(A)p}
Given any projection $p\in A$, there is a canonical unitary isomorphism $L^2(pAp)\cong p(L^2A)p$
sending $\sqrt\phi\in L^2(pAp)$ to $\sqrt{\phi\circ E}$, where $E(a)=pap$.
\hfill $\square$
\end{lemma}

The bimodule $L^2(A)$ may be characterized as follows.  It is a Hilbert space $H$ with faithful left and right actions of $A$, equipped with an antilinear isometric involution $J$ and a self-dual cone $P\subset H$ subject to the properties
	\begin{enumerate}
	\item $J A J = A'$ on $H$,
	\item $J c J = c^*$ for all $c \in Z(A)$,
	\item $J \xi = \xi$ for all $\xi \in P$,
	\item $a J a J (P) \subseteq P$ for all $a \in A$,  
	\item $\xi\, a = J a^* J \xi$ for all $\xi \in H$ and all $a \in A$.
	\end{enumerate}
Here, $A':=\{b\in\bfB(H)\,|\,[a,b]=0, \forall a \in A\}$ is the commutant of $A$; $J A J=\{JaJ\,|\,a\in A\}$; and the cone $P$ is called self-dual if $P=\{\eta\in H\,|\,\langle\xi,\eta\rangle\ge 0,\,\forall \xi\!\in\! P\}$.  The operator $J$ is called the modular conjugation.  A Hilbert space $H$, so equipped with a modular conjugation $J$ and a self-dual cone $P$, is called a \emph{standard form}.
Such a standard form is unique up to unique unitary isomorphism \cite{Haagerup(1975standard-form)}.

\begin{remark}
If $\phi$ is a faithful normal weight (an unbounded version of a state) on a von Neumann algebra $A$, then the GNS Hilbert space $L^2(A,\phi)$ is a standard form for $A$ \cite{Araki(Some-properties-of-modular-conjugation)} and therefore serves as a particular construction of the bimodule $L^2(A)$.
For example, taking $\phi$ to be the usual trace $\mathit{tr}$ on $\bfB(H)$, we see that the ideal of Hilbert-Schmidt operators on $H$ is a standard form for $\bfB(H)$.
\end{remark}

\begin{example}\label{ex:standard-forms-type-I} 
Let $H$ be a Hilbert space and $\bar H$ its complex conjugate.
Then $H \otimes \bar H$ is canonically identified with the ideal of Hilbert-Schmidt operators on $H$.
Let $P \subseteq H \otimes \bar H$ correspond to the positive Hilbert-Schmidt operators, and $J$ to the operation $x \mapsto x^*$, for $x$ a Hilbert-Schmidt operator.
Then $(H \otimes \bar H, J, P)$ is a standard form for $\bfB(H)$.
We have $J(\xi \otimes \bar \zeta) = \zeta \otimes \bar \xi$, and $\xi \otimes \bar \xi \in P$ for all $\xi \in H$.\footnote{Here, $\bar \xi\in \bar H$ is the image of $\xi\in H$ under the antilinear map $\mathrm{Id}_H:H\to \bar H$.}
\end{example}
 
\begin{example}
\label{ex:standard-forms-ox} 
Let $(H,J_A,P_A)$ and $(K, J_B, P_B)$ be standard forms for
von Neumann algebras $A$ and $B$.
Then there is a self-dual cone $P_{A \bar \otimes B}$ in 
$H \otimes K$ such that 
$(H \ox K, J_A \ox J_B, P_{A \bar \otimes B})$ is a standard form for $A \,\bar \otimes\, B$,
and such that $\xi \otimes \zeta \in P_{A \bar\otimes B}$ whenever $\xi \in P_A$ and $\zeta \in P_B$ \cite{Miura-Tomiyama(1984),Schmitt-Wittstock(1982)}.
Note that in general $P_{A \bar\otimes B}$ is strictly larger than the convex closure of $\{ \xi \otimes \zeta \mid \xi \in P_A, \zeta \in P_B \}$.
\end{example}

\subsection*{String diagrams}
String diagrams are a standard notation in monoidal categories 
and in bicategories \cite{Joyal-Street(The-geometry-of-tensor-calculus-I), Selinger(A-survey-of-graphical-languages-for-monoidal-categories)} 
and are often used
in the context of von Neuman algebras 
\cite{Jones(Planar-algebras-I), Ghosh(Planar-algebras:category-view)}.
We briefly recall this notation and discuss an extension that will be useful
later on.

In string diagrams, algebras are represented by shades, 
bimodules are represented by lines, and homomorphisms are nodes.
For example, an $A$-$B$-bilinear map $f$ between two bimodules ${}_AH_B$ and ${}_AK_B$ is depicted by the diagram
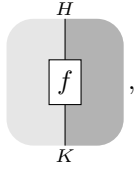
\[
\def\colA{black!10}
\def\colB{black!30}
\tikzmath[scale=\squarescale]
	{\fill[rounded corners=10, \colA] (-11,-12) rectangle (11,12);
	\fill[\colB]  (11,-12) -- (0,-12) -- (0,12) [rounded corners=10]-- (11,12) --  cycle;
	\draw (0,-12)node[below,yshift=2]{$\scriptstyle K$} -- (0,12)node[above,yshift=-2]{$\scriptstyle H$};
	\draw 	(0,0) node[fill=white, draw, inner xsep=3, inner ysep=4]{$f$};
	}\,,
\]
where the light shade corresponds to the algebra $A$ and the darker shade corresponds to the algebra $B$.
Other morphisms, such as $g:{}_AH\boxtimes_BK_C\to {}_AM_C$, $h:{}_AH_A\to{}_AL^2A_A$, or $k:{}_AL^2A_A\to {}_AH\boxtimes_BK_A$
are drawn similarly:
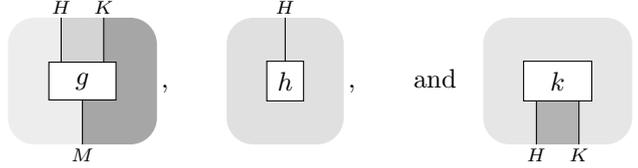
\[
\def\colA{black!7}
\def\colB{black!17}
\def\colC{black!35}
\tikzmath[scale=\squarescale]
	{\fill[rounded corners=10, \colA] (-14,-12) rectangle (14,12);
	\fill[\colC]  (14,-12) -- (0,-12) -- (0,12) [rounded corners=10]-- (14,12) --  cycle;
	\fill[\colB]  (-4,12) rectangle (4,0);
	\draw (0,-12)node[below,yshift=2]{$\scriptstyle M$} -- (0,0);
	\draw (-4,12)node[above,yshift=-2]{$\scriptstyle H$} -- (-4,0)
	(4,12)node[above,yshift=-2]{$\scriptstyle K$} -- (4,0);
	\draw 	(0,0) node[fill=white, draw, inner xsep=10, inner ysep=4]{$g$};
	}\,,\qquad
\def\colA{black!12}
\tikzmath[scale=\squarescale]
	{\fill[rounded corners=10, \colA] (-11,-12) rectangle (11,12);
	\draw (0,-12)node[below,yshift=2]{$\phantom {\scriptstyle K}$} (0,0) -- (0,12)node[above,yshift=-2]{$\scriptstyle H$};
	\draw 	(0,0) node[fill=white, draw, inner xsep=4, inner ysep=4]{$h$};
	}\,,\qquad\text{and}\quad
\def\colA{black!10}
\def\colB{black!30}
\tikzmath[scale=\squarescale]
	{\fill[rounded corners=10, \colA] (-14,-12) rectangle (14,12);
	\fill[\colB]  (-4,-12) rectangle (4,0);
	\draw (-4,-12)node[below,yshift=2]{$\scriptstyle H$} -- (-4,0)
	(4,-12)node[below,yshift=2]{$\scriptstyle K$} -- (4,0)
	(0,12)node[above,yshift=-2]{$\phantom {\scriptstyle H}$};
	\draw 	(0,0) node[fill=white, draw, inner xsep=10, inner ysep=4]{$k$};
	}\, .
\]
(Here, $\boxtimes$ is the operation of Connes fusion, which will be introduced in the following section, and ${}_AL^2A_A$ is the identity with respect to that operation.)
The identity morphism between bimodules is drawn as a single vertical line 
$
\def\colA{black!10}
\def\colB{black!30}
\tikzmath[scale=.085]
	{\fill[rounded corners=3.5, \colA] (-3,-2.2) rectangle (3,2.2);
	\fill[\colB]  (3,-2.2) -- (0,-2.2) -- (0,2.2) [rounded corners=3.5]-- (3,2.2) --  cycle;
	\draw (0,-2.2) -- (0,2.2);
	}
$\,.
Horizontal juxtaposition of pictures corresponds to Connes fusion, and vertical concatenation corresponds to composition of morphisms.
A more complicated composition of bimodule morphisms, such as 
\[
\tikzmath[scale=.1]{
\node at (-6.1,-1) {$\scriptstyle A$};
\node at (0,0) {$H\underset B\boxtimes K$};
\node at (5.9,-1) {$\scriptstyle D$};
}
\!\xrightarrow{1_H{\scriptscriptstyle\boxtimes}\, f}\!
\tikzmath[scale=.1]{
\node at (-9.5,-1) {$\scriptstyle A$};
\node at (0,0) {$H\underset B\boxtimes P\underset C\boxtimes N$};
\node at (9.4,-1) {$\scriptstyle D$};
}
\!\xrightarrow{g\,{\scriptscriptstyle\boxtimes} 1_N}\!
\tikzmath[scale=.1]{
\node at (-6.3,-1) {$\scriptstyle A$};
\node at (0,0) {$M\underset C\boxtimes N$};
\node at (5.9,-1) {$\scriptstyle D$};
}
\]
is denoted
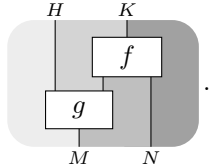
\[
\def\colA{black!7}
\def\colB{black!17}
\def\colC{black!25}
\def\colD{black!37}
\tikzmath[scale=\squarescale]
	{\fill[rounded corners=10, \colA] (-18,-12) rectangle (18,12);
	\fill[\colD]  (18,-12) -- (4,-12) -- (4,12) [rounded corners=10]-- (18,12) --  cycle;
	\fill[\colB]  (-9,12) -- (-9,-5) -- (0,-5) -- (0,5) -- (4.5,5) -- (4.5,12) -- cycle;
	\fill[\colC]  (9,-12) -- (9,5) -- (0,5) -- (0,-5) -- (-4.5,-5) -- (-4.5,-12) -- cycle;
	\draw (-4.5,-12)node[below,yshift=2]{$\scriptstyle M$} -- (-4.5,-5);
	\draw (-9,12)node[above,yshift=-2]{$\scriptstyle H$} -- (-9,-5)
	(9,-12)node[below,yshift=2]{$\scriptstyle N$} -- (9,5)
	(4.5,12)node[above,yshift=-2]{$\scriptstyle K$} -- (4.5,5)  (0,-5) -- (0,5);
	\draw 	(-4.5,-5) node[fill=white, draw, inner xsep=10, inner ysep=4]{$g$};
	\draw 	(4.5,5) node[fill=white, draw, inner xsep=10, inner ysep=3]{$f$};
	}\,.
\]

Our addition is the introduction of a notation for morphisms that are only left-linear, or only right-linear.
We denote them by nodes that extend to the right and to the left of the diagram, respectively.
Thus, an $A$-linear morphism $f$ between bimodules ${}_AH_B$ and ${}_AK_C$ is denoted
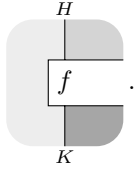
\[
\def\colA{black!7}
\def\colB{black!17}
\def\colC{black!35}
\tikzmath[scale=\squarescale]
	{\coordinate (upper right corner) at (11,12);
	\fill[\colA] (-11,-12) -- (0,-12) -- (0,12) [rounded corners=10]-- (-11,12) -- cycle;
	\draw 	(0,0) node[fill=white, inner xsep=3, inner ysep=4] (O) {$f$};
	\fill[\colB] (11,12) -- (upper right corner |- O.north) -- (O.north) -- (0,12) [rounded corners=10] -- cycle;
	\fill[\colC] (11,-12) -- (upper right corner |- O.south) -- (O.south) -- (0,-12) [rounded corners=10] -- cycle;
	\draw (0,-12)node[below,yshift=2]{$\scriptstyle K$} -- (O.south) (O.north) -- (0,12)node[above,yshift=-2]{$\scriptstyle H$}
	(upper right corner |- O.north) -- (O.north west) -- (O.south west) -- (upper right corner |- O.south);
	}\,.
\] 
We will always use the color white for the algebra $\IC$.
For example, a $B$-linear map $g$ from ${}_AH_B$ to some right $B$-module $K_B$ is drawn like this:
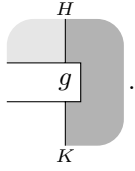
\[
\def\colA{black!10}
\def\colB{white}
\def\colC{black!30}
\tikzmath[scale=\squarescale]
	{\coordinate (upper left corner) at (-11,12);
	\fill[\colC] (11,-12) -- (0,-12) -- (0,12) [rounded corners=10]-- (11,12) -- cycle;
	\draw 	(0,0) node[fill=white, inner xsep=3, inner ysep=4] (O) {$g$};
	\fill[\colA] (-11,12) -- (upper left corner |- O.north) -- (O.north) -- (0,12) [rounded corners=10] -- cycle;
	\fill[\colB] (-11,-12) -- (upper left corner |- O.south) -- (O.south) -- (0,-12) [rounded corners=10] -- cycle;
	\draw (0,-12)node[below,yshift=2]{$\scriptstyle K$} -- (O.south) (O.north) -- (0,12)node[above,yshift=-2]{$\scriptstyle H$}
	(upper left corner |- O.north) -- (O.north east) -- (O.south east) -- (upper left corner |- O.south);
	}\,.
\]

Our conventions also allow us to speak about algebra elements using the same graphical notation, as
every right (left) $A$-linear morphism $L^2(A)\to L^2(A)$ is given by left (right) multiplication by an element $a\in A$.
Such an element will be denoted
\(
\def\colA{black!12}
\tikzmath[scale=\squarescale]
	{\coordinate (x) at (-10,0);
	\draw 	(-5,0) node(a)[fill=white, inner ysep=2]{$a$};
	\fill[fill=\colA] (-10,6) --(a.north east -| x) -- (a.north east) -- (a.south east) -- (a.south east -| x) [rounded corners=6]-- (-10,-6) -- (4,-6) -- (4,6) -- cycle;
	\draw (a.north east -| x) -- (a.north east) -- (a.south east) -- (a.south east -| x);
	}
\)\,,
or
\(
\def\colA{black!12}
\tikzmath[scale=\squarescale]
	{\coordinate (x) at (10,0);
	\draw 	(5,0) node(a)[fill=white, inner ysep=2]{$a$};
	\fill[fill=\colA] (10,6) -- (a.north west -| x) -- (a.north west) -- (a.south west) -- (a.south west -| x) [rounded corners=6]-- (10,-6) -- (-4,-6) -- (-4,6) -- cycle;
	\draw (a.north west -| x) -- (a.north west) -- (a.south west) -- (a.south west -| x);
	}
\)\,,
depending on whether we view it as acting on the left or on right on $L^2(A)$.
The fact that an $A$-linear morphism $f:{}_AH_B\to{}_AH_B$ commutes with the left action of an element $a\in A$ is then nicely rendered by the equation
\[
\def\colA{black!10}
\def\colB{black!30}
\tikzmath[scale=\squarescale]
	{\coordinate (upper right corner) at (8,12);
	\coordinate (lower right corner) at (8,-12);
	\coordinate (upper left corner) at (-13,12);
	\coordinate (lower left corner) at (-13,-12);
	\draw 	(-9,4) node(a)[fill=white, inner xsep=3, inner ysep=3]{$a$};
	\fill[\colA] (lower left corner) -- (0,-12) -- (0,12) [rounded corners=10]-- (upper left corner) [sharp corners]-- (a.north east -| upper left corner) -- (a.north east) -- (a.south east) -- (a.south east -| upper left corner) [rounded corners=10]-- cycle;
	\draw 	(0,-4) node[fill=white, inner xsep=3, inner ysep=2] (O) {$f$};
	\draw (a.north east -| upper left corner) -- (a.north east) -- (a.south east) -- (a.south east -| upper left corner);
	\fill[\colB] (upper right corner) -- (upper right corner |- O.north) -- (O.north) -- (0,12) [rounded corners=10] -- cycle;
	\fill[\colB] (lower right corner) -- (upper right corner |- O.south) -- (O.south) -- (0,-12) [rounded corners=10] -- cycle;
	\draw (0,-12) -- (O.south) (O.north) -- (0,12)
	(upper right corner |- O.north) -- (O.north west) -- (O.south west) -- (upper right corner |- O.south);
	}\,\,\,=\,\,\,
\tikzmath[scale=\squarescale]
	{\coordinate (upper right corner) at (8,12);
	\coordinate (lower right corner) at (8,-12);
	\coordinate (upper left corner) at (-13,12);
	\coordinate (lower left corner) at (-13,-12);
	\draw 	(-9,-4) node(a)[fill=white, inner xsep=3, inner ysep=3]{$a$};
	\fill[\colA] (lower left corner) -- (0,-12) -- (0,12) [rounded corners=10]-- (upper left corner) [sharp corners]-- (a.north east -| upper left corner) -- (a.north east) -- (a.south east) -- (a.south east -| upper left corner) [rounded corners=10]-- cycle;
	\draw 	(0,4) node[fill=white, inner xsep=3, inner ysep=2] (O) {$f$};
	\draw (a.north east -| upper left corner) -- (a.north east) -- (a.south east) -- (a.south east -| upper left corner);
	\fill[\colB] (upper right corner) -- (upper right corner |- O.north) -- (O.north) -- (0,12) [rounded corners=10] -- cycle;
	\fill[\colB] (lower right corner) -- (upper right corner |- O.south) -- (O.south) -- (0,-12) [rounded corners=10] -- cycle;
	\draw (0,-12) -- (O.south) (O.north) -- (0,12)
	(upper right corner |- O.north) -- (O.north west) -- (O.south west) -- (upper right corner |- O.south);
	}\,.
\]
Finally, we can also denote vectors graphically, given that an element $\xi \in H$ is equivalent to a map $\IC\to H$.
For example, a vector in a bimodule ${}_AH_B$ is denoted
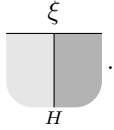
\[
\def\colA{black!10}
\def\colB{black!30}
\tikzmath[scale=\squarescale]{
\coordinate	(lower left corner) at (-9,-5);
\coordinate	(upper right corner) at (9,9);
\coordinate	(upper left corner) at (upper right corner -| lower left corner);
\coordinate(lower right corner) at (upper right corner |- lower left corner);
\fill[\colA] (lower left corner) --  (upper left corner) -- ($(upper right corner)!.5!(upper left corner)$) -- ($(lower right corner)!.5!(lower left corner)$) [rounded corners=10]-- cycle;
\fill[\colB] (lower right corner) --  (upper right corner) -- ($(upper right corner)!.5!(upper left corner)$) -- ($(lower right corner)!.5!(lower left corner)$) [rounded corners=10]-- cycle;
\draw (upper left corner) --node[above]{$\xi$} (upper right corner)
($(upper right corner)!.5!(upper left corner)$) -- ($(lower right corner)!.5!(lower left corner)$) node[below,yshift=2]{$\scriptstyle H$};
}\,.
\]
The node $\xi$ extends both to the right and to the left, as the map $\xi:\IC\to {}_AH_B$ is neither $A$- nor $B$-linear.
Also, the space above $\xi$ is white because the source of the above map is ${}_\IC\IC_\IC$.

\section{Connes fusion}\label{sec: Connes fusion}
\begin{definition}
Given two modules $H_A$ and ${}_AK$ over a von Neumann algebra $A$,
their Connes fusion $H\boxtimes_A K$ is the completion of
\begin{equation}\label{eq:def of CFus}
\hom\big(L^2(A)_A,H_A\big)\otimes_A L^2(A)\otimes_A \hom\big({}_AL^2(A),{}_AK\big)
\end{equation}
with respect to the inner product
$\big\langle\phi_1\otimes \xi_1\otimes \psi_1,\,\phi_2\otimes \xi_2\otimes \psi_2\big\rangle:=\big\langle(\phi_2^*\phi_1)\xi_1(\psi_1\psi_2^*),\xi_2\big\rangle$ 
\cite{Connes(Non-commutative-geometry), Popa(Correspondences), Sauvageot(Sur-le-produit-tensoriel-relatif), Wassermann(Operator-algebras-and-conformal-field-theory)}.
In the above equation, we have written the action of $\psi_i$ on the right, which means that
$\psi_1\psi_2^*$ stands for the composite $L^2(A)\xrightarrow{\psi_1} K\xrightarrow{\psi_2^*} L^2(A)$.
\end{definition}

The image in the Connes fusion of an element
\[
\phi\otimes \xi\otimes\psi \,\,\,=\,\,\, \def\colA{black!20}
\tikzmath[scale=\squarescale]{
\useasboundingbox (-33,-14) rectangle (33,13);

\coordinate	(lower left corner A) at (-27,-11);
\coordinate	(upper right corner A) at (-17,11);
\coordinate	(upper left corner A) at ($(upper right corner A -| lower left corner A) + (-6,0)$); \coordinate(lower right corner A) at (upper right corner A |- lower left corner A);
\draw 		(-26,0) node (a) {$\phantom{\varphi}$};
\draw 		(-27,0) node {$\phi$};
\fill[\colA] (lower left corner A) -- (a.south east -| lower left corner A)
-- (a.south east) -- (a.north east) -- (a.north east -| upper left corner A) {[rounded corners=10] --  (upper left corner A) -- (upper right corner A) -- (lower right corner A)} -- cycle;
\draw (lower left corner A) node[below,yshift=2]{$\scriptstyle H$} -- (a.south east -| lower left corner A)
(a.south east -| upper left corner A) -- (a.south east) -- (a.north east) -- (a.north east -| upper left corner A);

\node at (-12,0) {$\otimes$};

\coordinate	(lower left corner B) at (-7,-11);
\coordinate	(upper right corner B) at (7,10);
\coordinate	(upper left corner B) at (upper right corner B -| lower left corner B); \coordinate(lower right corner B) at (upper right corner B |- lower left corner B);
\def\v{5}
\fill[\colA, rounded corners=10] (lower left corner B) {[sharp corners] --  ($(upper left corner B) - (0,\v)$) -- ($(upper right corner B) - (0,\v)$)} -- (lower right corner B) -- cycle;
\draw ($(upper left corner B) - (0,\v)$) --node[above]{$\xi$} ($(upper right corner B) - (0,\v)$);

\node at (12,0) {$\otimes$};

\coordinate	(lower left corner C) at (17,-11);
\coordinate	(upper right corner C) at (33,11);
\coordinate	(upper left corner C) at (upper right corner C -| lower left corner C); \coordinate(lower right corner C) at ($(upper right corner C |- lower left corner C)-(6,0)$);
\draw 		(26,0) node (b) {$\phantom{\varphi}$};
\draw 		(27,0) node {$\psi$};
\fill[\colA] (lower right corner C) -- (b.south west -| lower right corner C)
-- (b.south west) -- (b.north west) -- (b.north west -| upper right corner C) {[rounded corners=10] --  (upper right corner C) -- (upper left corner C) -- (lower left corner C)} -- cycle;
\draw (lower right corner C)node[below,yshift=2]{$\scriptstyle K$} -- (b.south west -| lower right corner C)(b.south west -| upper right corner C) -- (b.south west) -- (b.north west) -- (b.north west -| upper right corner C);
	}
\vspace{-.2cm}\]
is equal to
\(
\def\colA{black!20}
\tikzmath[scale=\squarescale]{
\coordinate	(lower left corner) at (-8,-12);
\coordinate	(upper right corner) at (13,14);
\coordinate	(upper left corner) at ($(upper right corner -| lower left corner) + (-5,0)$); 
\coordinate(lower right corner) at ($(upper right corner |- lower left corner)-(5,0)$);
\draw 		(-8,0) node (a) {$\phantom{\varphi}$};
\draw 		(8,0) node (b) {$\phantom{\varphi}$};
\draw 		(a) node {$\phi$};
\draw 		(b) node {$\psi$};
\def\v{5}
\fill[\colA] (lower left corner) -- (a.south east -| lower left corner)
-- (a.south east) -- (a.north east) -- (a.north east -| upper left corner) 
--  ($(upper left corner) - (0,\v)$) -- ($(upper right corner) - (0,\v)$) 
-- (b.north west -| upper right corner) 
-- (b.north west) -- (b.south west) -- (b.south west -| lower right corner) -- (lower right corner) -- cycle;
\draw (lower left corner) node[below,yshift=2]{$\scriptstyle H$} -- (a.south east -| lower left corner)
(a.south east -| upper left corner) -- (a.south east) -- (a.north east) -- (a.north east -| upper left corner);
\draw (lower right corner)node[below,yshift=2]{$\scriptstyle K$} -- (b.south west -| lower right corner)(b.south west -| upper right corner) -- (b.south west) -- (b.north west) -- (b.north west -| upper right corner);
\draw ($(upper left corner) - (0,\v)$) --node[above]{$\xi$} ($(upper right corner) - (0,\v)$);
	}
\)\,.
Strictly speaking, the latter picture refers to the morphism
\[
\IC\xrightarrow{\xi} L^2A\cong L^2A\boxtimes_A L^2A \xrightarrow{\phi\boxtimes \psi} H\boxtimes_A K,
\]
but we can always identify a map from $\IC$ to some vector space with the vector that is the image of $1$ under that map.

\begin{remark}\label{rmk: alternative def of fusion}
It is useful to note that the completion map from \eqref{eq:def of CFus} to $H\boxtimes_A K$ factors through both
\(
H\!\otimes_A\! \hom({}_AL^2(A),{}_AK)
\)
and
\(
\hom(L^2(A)_A,H_A)\!\otimes_A\! K
\).
The Hilbert space $H\boxtimes_A K$ therefore admits two alternative asymmetric definitions, as completions of either of those tensor products.
\end{remark}

\begin{remark}
A pair of vectors $\xi\in H_A$, $\eta\in{}_AK$ does not represent anything in $H\boxtimes_A K$.
This is nicely reflected by the fact that it is not possible to assemble the pictures
$
\def\colA{white}
\def\colB{black!15}
\tikzmath[scale=\squarescale]{
\useasboundingbox (-5,-6) rectangle (5,8);
\coordinate	(lower left corner) at (-5,-4);
\coordinate	(upper right corner) at (5,4);
\coordinate	(upper left corner) at (upper right corner -| lower left corner);
\coordinate(lower right corner) at (upper right corner |- lower left corner);
\fill[\colA] (lower left corner) --  (upper left corner) -- ($(upper right corner)!.5!(upper left corner)$) -- ($(lower right corner)!.5!(lower left corner)$) [rounded corners=6]-- cycle;
\fill[\colB] (lower right corner) --  (upper right corner) -- ($(upper right corner)!.5!(upper left corner)$) -- ($(lower right corner)!.5!(lower left corner)$) [rounded corners=6]-- cycle;
\draw (upper left corner) --node[above, yshift=-1]{$\scriptstyle \xi$} (upper right corner)
($(upper right corner)!.5!(upper left corner)$) -- ($(lower right corner)!.5!(lower left corner)$) node[below,yshift=2]{$\scriptscriptstyle H$};
}
$\,\,
and\,
$
\def\colA{black!15}
\def\colB{white}
\tikzmath[scale=\squarescale]{
\useasboundingbox (-5,-6) rectangle (5,8);
\coordinate	(lower left corner) at (-5,-4);
\coordinate	(upper right corner) at (5,4);
\coordinate	(upper left corner) at (upper right corner -| lower left corner);
\coordinate(lower right corner) at (upper right corner |- lower left corner);
\fill[\colA] (lower left corner) --  (upper left corner) -- ($(upper right corner)!.5!(upper left corner)$) -- ($(lower right corner)!.5!(lower left corner)$) [rounded corners=6]-- cycle;
\fill[\colB] (lower right corner) --  (upper right corner) -- ($(upper right corner)!.5!(upper left corner)$) -- ($(lower right corner)!.5!(lower left corner)$) [rounded corners=6]-- cycle;
\draw (upper left corner) --node[above, yshift=-1]{$\scriptstyle \eta$} (upper right corner)
($(upper right corner)!.5!(upper left corner)$) -- ($(lower right corner)!.5!(lower left corner)$) node[below,yshift=2]{$\scriptscriptstyle K$};
}
$\,
into a meaningful diagram.
\end{remark}

\begin{remark} \label{rem:alternative-Connes-fusion}
There exist two algebraic alternatives to von Neumann bimodules and Connes fusion.
In both cases, the Connes fusion is replaced by a simpler, purely algebraic operation.

In the first alternative, bimodules are replaced by homomorphism (often endomorphisms) of von Neumann algebras, and one usually restricts attention to type {\it III} factors
\cite{Longo(Index-of-subfactors-and-statistics-of-quantum-fields-I), Longo(Index-of-subfactors-and-statistics-of-quantum-fields-II)}.
In this case, Connes fusion becomes merely the composition of homomorphisms.
The translation from homomorphisms back to bimodules is as follows:
given a homomorphism $\varphi:A\to B$, one precomposes the left action on $L^2(B)$ by $\varphi$ to get a bimodule ${}_\varphi L^2(B)$. 
Given a second homomorphism $\psi:B\to C$, there is a canonical isomorphism
${}_{\varphi} L^2(B) \boxtimes_B {}_\psi L^2(C) \cong {}_{\psi\circ\varphi} L^2(C)$ of $A$-$C$-bimodules.

The second alternative has been pointed out by 
Thom~\cite{Thom(Remark-Connes-fusion)}:
the functor $_A H_B \mapsto \hom\big(L^2(B)_B,H_B\big)$
provides an equivalence between the bicategory of dualizable bimodules (Definition \ref{def: finite bimodule})
that are topologically finitely generated 
(i.e., for which there is a finite set that spans a dense submodule)
both as left and as right modules,
and the bicategory of algebraic bimodules (i.e., no Hilbert space structure)
that are finitely generated projective both as left and as right modules.
Under that equivalence, Connes fusion corresponds to the usual 
algebraic tensor product.
Note that this does not provide a description of all dualizable bimodules.
For example, the bimodule $_{\bfB(H)} H_{\IC}$ is dualizable (it is even a Morita equivalence),
but the corresponding algebraic bimodule is certainly not finitely generated.
\end{remark}

\begin{lemma}\label{lem: bar H boxtimes_{bfB(H)} H}
Let $H$ be a Hilbert space. View its complex conjugate $\bar H$ as a right $\bfB(H)$-module by $\bar \xi a:=\overline{a^*\xi}$.  Then there is a canonical isomorphism
${\bar H \boxtimes_{\bfB(H)} H\cong \IC}$.
\end{lemma}

\begin{proof}
The Hilbert space $L^2(\bfB(H))$ is canonically isomorphic to the space of Hilbert-Schmidt operators on $H$, that is to $H\otimes \bar H$; see Example \ref{ex:standard-forms-type-I}.
Following Remark \ref{rmk: alternative def of fusion}, the Connes fusion $\bar H \boxtimes_{\bfB(H)} H$ is obtained from
\begin{equation}\label{eq: .. .. ..}
\hom((H\otimes \bar H)_{\bfB(H)},\bar H_{\bfB(H)})\otimes_{\bfB(H)} H
\end{equation}
by completing it with respect to the inner product
$\big\langle\phi_1\otimes \xi_1,\,\phi_2\otimes \xi_2\big\rangle:=\big\langle(\phi_2^*\phi_1)\xi_1,\xi_2\big\rangle$.
There is an isomorphism
\[
\begin{split}
\bar H&\to \hom((H\otimes \bar H)_{\bfB(H)},\bar H_{\bfB(H)})\\
\bar \eta \,&\mapsto\, \big(\xi\otimes \bar \zeta\mapsto \overline{\langle \eta,\xi\rangle \zeta}\big).
\end{split}
\]
Applying the inverse of this isomorphism to the first term of \eqref{eq: .. .. ..}, we obtain the vector space $\bar H \otimes_{\bfB(H)} H$ with inner product
$\big\langle\bar\eta_1\otimes \xi_1,\,\bar\eta_2\otimes \xi_2\big\rangle:=
\big\langle\eta_2\overline{\langle\eta_1,\xi_1\rangle},\xi_2\big\rangle=\langle\xi_1,\eta_1\rangle\langle\eta_2,\xi_2\rangle$.
The map $\bar\eta\otimes\xi\mapsto \langle\xi,\eta\rangle:\bar H \otimes_{\bfB(H)} H\to \IC$ is therefore a unitary isomorphism.
\end{proof}

\begin{remark} 
The functor $H\boxtimes_A -$
can be characterized by the existence of a right $A$-module isomorphism $H \boxtimes_A L^2(A) \cong H$ (see \cite{Sherman(Relative-tensor-products)}).
\end{remark}

Connes fusion shares the formal properties of the usual algebraic tensor product:

\begin{proposition}[\cite{Landsman(Bicategories-of-operator-algebras)}]
There is a bicategory whose objects are von Neumann algebras,
whose arrows are bimodules, and whose 2-morphisms are maps of bimodules.
The composition of arrows is given by Connes fusion.
\end{proposition}

\noindent Spatial tensor product of von Neumann algebras and tensor product of
   Hilbert spaces provides a symmetric tensor product on this bicategory,
   but since the formal definition of a symmetric monoidal bicategory is
   somewhat lengthy, we do not pursue this in detail here.

The invertible arrows of this bicategory are called \emph{Morita equivalences}, 
and have the following alternative characterization:

\begin{proposition}\label{prop: characterization of invertible bimodules}
A bimodule ${}_AH_B$ is invertible with respect to Connes fusion if and only if the two algebras act faithfully, and $B':=\{x\in\bfB(H)\,|\,[x,B]=0\}=A$.
In that case, the inverse bimodule is given by the complex conjugate Hilbert space $\overline H$, with actions 
$b\bar \xi := \overline{\xi b^*}$ and\, $\bar \xi a:= \overline{a^* \xi}$.
\end{proposition}

\begin{proof}
We first assume that the two actions are faithful and that $B' = A$.
Using a unitary $A$-module identification ${H\otimes\ell^2}\cong {L^2A\otimes\ell^2}$,
we get isomorphisms
\begin{equation} \label{eq:H-box-barH->-L2}
\begin{split}
{}_{A}H\boxtimes_B\overline H{}_{A}
&\cong {}_{A}(H\otimes \ell^2)\boxtimes_{B\bar \otimes \bfB(\ell^2)}\overline{H\otimes \ell^2}{}_{A}\\
&\cong {}_{A}(L^2A\otimes \ell^2)\boxtimes_{A\bar \otimes \bfB(\ell^2)}\overline{L^2A\otimes \ell^2}{}_{A}\\
&\cong {}_{A}L^2A\boxtimes_{A}\overline{L^2A}{}_{A}
\cong {}_{A}L^2A\boxtimes_{A}L^2A_{A}
\cong {}_{A}L^2A_{A}.
\end{split}
\end{equation}
The first isomorphism follows from Lemma \ref{lem: bar H boxtimes_{bfB(H)} H}, and the fourth one is given by the modular conjugation on $L^2A$.
Similarly, we have ${}_B\overline H\boxtimes_AH_B\cong {}_BL^2(B)_B$, and so ${}_AH_B$ is invertible.
Conversely, if ${}_{A}H\boxtimes_B\overline H{}_{A}\cong{}_{A}L^2A_{A}$, then the $A$-action is faithful and we have
\[
B'\subset \mathrm{End}\big(H\boxtimes_B\overline H{}_{A}\big)=\mathrm{End}\big(L^2A_{A}\big)=A,
\]
from which $B'=A$ follows.
Similarly, the faithfulness of the right $B$-action follows from the 
isomorphism ${}_B\overline H\boxtimes_AH_B\cong {}_BL^2(B)_B$.
\end{proof}

\begin{lemma} \label{lem:stay-injective}
  Let ${}_{A}H$ be a faithful $A$-module and let $f:K_A\to L_A$ be an $A$-linear map.
  Then $f$ is injective if and only if $f \otimes 1_H:K\boxtimes_A  H \to L\boxtimes_A H$ is injective. 
\end{lemma}

\begin{proof}
  Pick an $A$-module identification ${H\otimes\ell^2}\cong {L^2A\otimes\ell^2}$.
  We then have
\[
\begin{split}    f \,\, \text{is injective} \,\,\, &\Leftrightarrow\,\,\,
    K\boxtimes_A  L^2A \xrightarrow{f \otimes 1} L\boxtimes_A L^2A \,\,\,\,\,\,\,\, \text{is injective} \;  \\ &\Leftrightarrow\,\,\,
    K\boxtimes_A  L^2A \otimes \ell^2 \xrightarrow{f \otimes 1 \otimes 1} L\boxtimes_A L^2A \otimes \ell^2  \,\,\,\, \text{is injective} \; \\ &\Leftrightarrow\,\,\,
    K\boxtimes_A  H \otimes \ell^2 \xrightarrow{f \otimes 1 \otimes 1} L\boxtimes_A H \otimes \ell^2  \,\,\,\,\,\, \text{is injective} \; \\ & \Leftrightarrow\,\,\,
    K\boxtimes_A  H \xrightarrow{f \otimes 1} L\boxtimes_A H  \,\,\,\,\,\,\,\, \text{is injective}. \qedhere
\end{split}
\]
\end{proof}

\begin{remark} \label{rem:uniqueness-for-characterization-invertible-bimodules}
The construction of the isomorphism ${}_{A}H\boxtimes_B\overline H{}_{A} \cong {}_{A}L^2A_{A}$ in~\eqref{eq:H-box-barH->-L2} 
used the choice of an $A$-linear unitary $x \colon {H\otimes\ell^2}\cong {L^2A\otimes\ell^2}$.
Nevertheless, we claim that~\eqref{eq:H-box-barH->-L2} is canonical.
Note first that $x$ enters~\eqref{eq:H-box-barH->-L2} only through the isomorphism
\begin{equation}
\label{eq:using-varphi}
{}_{A}(H\otimes \ell^2)\boxtimes_{B\bar \otimes \bfB(\ell^2)}
 \overline{H\otimes \ell^2}{}_{A}
\cong 
{}_{A}(L^2A\otimes \ell^2)\boxtimes_{A\bar \otimes \bfB(\ell^2)}
\overline{L^2A\otimes \ell^2}{}_{A}.
\end{equation}
In order to understand~\eqref{eq:using-varphi} we provide a bit more notation.
Conjugation by $x$ yields an isomorphism $f \colon B\,\bar \otimes\, \bfB(\ell^2) \to A\,\bar \otimes\, \bfB(\ell^2)$.
Using this, $x$ can be viewed as an isomorphism of $A$-$B\,\bar \otimes\, \bfB(\ell^2)$-bimodules ${H\otimes\ell^2}\cong \left({L^2A\otimes\ell^2}\right)_{f}$, where 
the right action of $B\,\bar \otimes\, \bfB(\ell^2)$ on $L^2A\otimes\ell^2$ is defined by $f$.
Similarly, the complex conjugate $\overline x$ yields an isomorphism of $B\,\bar \otimes\, \bfB(\ell^2)$-$A$-bimodules
$\overline{H\otimes \ell^2} \cong \!\!\!{\phantom{\big|}}_f\Big(\overline{L^2A\otimes \ell^2}\Big)$.
The map~\eqref{eq:using-varphi} is then the composition of $x \boxtimes \bar x$ with the isomorphism
\[
\Big({L^2A\otimes\ell^2}\Big)\!\!\!{\phantom{\Big|}}_f \boxtimes_{B\bar \otimes \bfB(\ell^2)}
\!\!\!{\phantom{\Big|}}_f\Big(\overline{L^2A\otimes \ell^2}\Big)
\cong
(L^2A\otimes\ell^2) \boxtimes_{A\bar \otimes \bfB(\ell^2)} \overline{L^2A\otimes \ell^2}
\]
that sends $\varphi \otimes \xi\otimes \psi$ to $(\varphi\circ L^2(f^{-1}))\otimes L^2(f)\xi\otimes (\psi\circ L^2(f^{-1}))$.
Suppose that $y \colon {H\otimes\ell^2}\cong {L^2A\otimes\ell^2}$ is another left $A$-module identification. 
Conjugation by $y$ yields a second isomorphism $g \colon B\,\bar \otimes\, \bfB(\ell^2) \to A\,\bar \otimes\, \bfB(\ell^2)$.
Now, $yx^* \colon {L^2A\otimes\ell^2} \to {L^2A\otimes\ell^2}$ is left $A$-linear, and so there is a unitary $u \in A\,\bar \otimes\, \bfB(\ell^2)$ whose right action $R_u$ on $L^2A\otimes\ell^2$ is $yx^*$.
The left action $L_{u^*}$ on $\overline{L^2A\otimes \ell^2}$ is then given by $\bar y \bar x^*$.
Let also $v\in B\,\bar \otimes\, \bfB(\ell^2)$ be such that $R_v=x^*y$.
We then have $f(v)=u$, and $g^{-1}f=\mathrm{ad}(v)$.
Altogether, the two maps that we are trying to compare are along the top and along the bottom of the following diagram
\begin{center}
\begin{tikzpicture}
\node (m-1-2) at (4,3.2){$\displaystyle \Big({L^2A\otimes\ell^2}\Big)\!\!\!{\phantom{\Big|}}_f \boxtimes_{B\bar \otimes \bfB(\ell^2)}\!\!\!{\phantom{\Big|}}_f\Big(\overline{L^2A\otimes \ell^2}\Big)$};
\node (m-2-1) at (0,1.6) {$\displaystyle {}_{A}(H\otimes \ell^2)\boxtimes_{B\bar \otimes \bfB(\ell^2)} \overline{H\otimes \ell^2}{}_{A}$};
\node (m-2-3) at (7.8,1.6) {$\displaystyle{}_{A}(L^2A\otimes \ell^2)\boxtimes_{A\bar \otimes \bfB(\ell^2)} \overline{L^2A\otimes \ell^2}{}_{A}$};
\node (m-3-2) at (4,0) {$\displaystyle\Big({L^2A\otimes\ell^2}\Big)\!\!\!{\phantom{\Big|}}_g \boxtimes_{B\bar \otimes \bfB(\ell^2)} \!\!\!{\phantom{\Big|}}_g\Big(\overline{L^2A\otimes \ell^2}\Big)$};
\path[->,font=\scriptsize]
(m-2-1) edge node[above, pos=.40, xshift=-6] {$x \boxtimes \bar x$} (m-1-2)
(m-2-1) edge node[below, pos=.35, xshift=-5] {$y \boxtimes \bar y$} (m-3-2)
(m-1-2) edge node[fill=white, xshift=2, inner sep=2] {$R_u \boxtimes L_{u^*}$} (m-3-2)
(m-1-2) edge node[above, pos=.6] {$\cong$} (m-2-3)
(m-3-2) edge node[below, pos=.6] {$\cong$} (m-2-3);
\end{tikzpicture}
\end{center}
It is not hard to check that the left triangle commutes.
For the commutativity of the right triangle, we show that the vertical map $\varphi\otimes \xi\otimes \psi\mapsto (R_u\circ \varphi)\otimes \xi\otimes (L_{u^*}\circ \psi)$ agrees with the map
that goes down the right side of the diagram.
Indeed, recalling from \eqref{eq: u sqrt(phi)u^*= sqrt(u phi u^*)} that $L^2(g^{-1}f)=L^2(\mathrm{ad}(v))=L_vR_{v^*}$, that map is
\[
\begin{split}
\varphi\otimes \xi\otimes \psi\,\mapsto \,\,& (\varphi\circ L^2(f^{-1}g))\otimes L^2(g^{-1}f)\xi\otimes (\psi\circ L^2(f^{-1}g)\\
=\,\, & (\varphi\circ L_{v^*}R_v)\otimes v\,\xi\, v^*\otimes (\psi\circ L_{v^*}R_v)\\
=\,\,&(\varphi\circ R_v)\otimes \xi \otimes (\psi\circ L_{v^*}) =(R_u \circ \varphi)\otimes \xi \otimes (L_{u^*} \circ \psi).
\end{split}
\]
It follows that \eqref{eq:H-box-barH->-L2} is independent of $x$.
\end{remark}

\section{Dualizable bimodules}\label{sec: Dualizability}

A von Neumann algebra whose center is one dimensional is called a \emph{factor}.  A von Neumann algebra has finite-dimensional center if and only if it is a finite direct sum of factors.
Given an $A$-$B$-bimodule $H$ over von Neumann algebras with finite-dimensional center, we say that a $B$-$A$-bimodule $\bar H$ is dual to $H$ if it comes equipped with maps
\begin{equation}\label{eq:duality maps}
R:{}_AL^2(A)_A \rightarrow {}_AH\boxtimes_B \bar H_A\qquad\quad
S:{}_BL^2(B)_B \rightarrow  {}_B\bar H\boxtimes_A H_B
\end{equation}
subject to the duality equations $(R^*\otimes 1)(1\otimes S)=1$, $(S^*\otimes 1)(1\otimes R)=1$, and to
the normalization condition ${R^*(pxq\otimes 1)R} = {S^*(1\otimes pxq)S}$
for all $x\in \mathrm{End}({}_AH_B)$ and all minimal central projections $p\in Z(A)$ and $q\in Z(B)$.
The first two conditions are classical \cite{MacLane(Categories-for-the-working-mathematician)}.
The latter was inspired by \cite[Lemma 3.9]{Longo-Roberts(A-theory-of-dimension)}.
The above equations are best depicted using string diagrams.
The duality equations are given by
\begin{equation}\label{eq: duality equations}
\def\colA{black!10}
\def\colB{black!30}
\tikzmath[scale=\squarescale]
	{\fill[rounded corners=10, fill=\colA] (-19,-14) rectangle (19,14);
	\fill[rounded corners=9, fill=\colB] (12,-14) -- (12,6) -- (0,6) -- (0,-6) -- (-12,-6) [sharp corners]-- (-12,14)  [rounded corners=10]-- (19,14) -- (19,-14) [sharp corners]-- cycle;
	\draw[rounded corners=9] (12,-14) -- (12,6) -- (0,6) -- (0,-6) -- (-12,-6) -- (-12,14);
	\draw (-6,-6) node[fill=white, draw]{$R^*\!$} (6,6) node[fill=white, draw]{$S$};}
\,\,=\,\,\tikzmath[scale=\squarescale]
	{\fill[rounded corners=10, fill=\colA] (-9,-14) rectangle (9,14);
	\fill[fill=\colB] (0,-14) -- (0,14) [rounded corners=10] -- (9,14) -- (9,-14) [sharp corners]-- cycle;
	\draw (0,-14) -- (0,14);}
\!\qquad\text{and}\qquad\!\tikzmath[scale=\squarescale]
	{\fill[rounded corners=10, fill=\colA] (-19,-14) rectangle (19,14);
	\fill[rounded corners=9, fill=\colB] (12,-14) -- (12,6) -- (0,6) -- (0,-6) -- (-12,-6) [sharp corners]-- (-12,14)  [rounded corners=10]-- (-19,14) -- (-19,-14) [sharp corners]-- cycle;
	\draw[rounded corners=9] (12,-14) -- (12,6) -- (0,6) -- (0,-6) -- (-12,-6) -- (-12,14);
	\draw (-6,-6) node[fill=white, draw]{$S^*\!$} (6,6) node[fill=white, draw]{$R$};}
\,\,=\,\,\tikzmath[scale=\squarescale]
	{\fill[rounded corners=10, fill=\colA] (-9,-14) rectangle (9,14);
	\fill[fill=\colB] (0,-14) -- (0,14) [rounded corners=10] -- (-9,14) -- (-9,-14) [sharp corners]-- cycle;
	\draw (0,-14) -- (0,14);}
\end{equation}
and the normalization condition is
\begin{equation}\label{eq:duality normalization}
\def\colA{black!10}
\def\colB{black!30}
\tikzmath[scale=\squarescale]
	{\fill[rounded corners=10, fill=\colA] (-19,-18) rectangle (16,18);
	\draw[rounded corners=10, fill=\colB] (-9,-11) rectangle (9,11);
	\draw 	(-9,0) node[fill=white, draw, inner sep=4]{$x$}
			(0,-11) node[fill=white, draw]{$R^*\!$}
			(0,11) node[fill=white, draw] {$R$}
			(-1,0) node {$q$}
			(-16,0) node {$p$};
	}
\,\,=\,\,\tikzmath[scale=\squarescale]
	{\fill[rounded corners=10, fill=\colB] (19,-18) rectangle (-16,18);
	\draw[rounded corners=10, fill=\colA] (9,-11) rectangle (-9,11);
	\draw 	(9,0) node[fill=white, draw, inner sep=4]{$x$}
			(0,-11) node[fill=white, draw]{$S^*\!$}
			(0,11) node[fill=white, draw] {$S$}
			(1,0) node {$p$}
			(16,0) node {$q$};
	}\,\,.
\end{equation}
The two shades stand for the algebras $A$ and $B$, and the lines correspond to the bimodules ${}_AH_B$ and ${}_B\bar H_A$.
Note that the two sides of \eqref{eq:duality normalization} are in $p\,\mathrm{End}({}_AL^2(A)_A)\cong p\,Z(A)\cong \IC$ and $q\,\mathrm{End}({}_BL^2(B)_B)\cong q\,Z(B)\cong \IC$, 
respectively, and so it makes sense to ask for them to be equal.

\begin{definition}\label{def: finite bimodule}
A bimodule whose dual module exists is called \emph{dualizable}.
\end{definition}

We will show later, in Corollary \ref{cor: contragredient bimodule},
that the dual of a dualizable bimodule is canonically isomorphic to the complex conjugate of the bimodule.
For the time being, we now reserve the notation ${}_B\bar H_A$ for the dual.

\begin{remark}
In the literature, the term \emph{dual} typically refers to a solution of \eqref{eq: duality equations} only.
(When the conditions \eqref{eq: duality equations} are re-expressed purely in terms of $R$ and $S^*$ the triple $(\bar{H}, R, S^*)$ is called a right dual, and when the conditions are re-expressed in terms of $R^*$ and $S$ the triple $(\bar{H}, S, R^*)$ is called a left dual.)  Such a dual, if it exists, is well defined up to unique isomorphism.
However, in our Hilbert space context, having an object that is well defined up to unique isomorphism is not sufficient, as the isomorphism might fail to be unitary.
Condition \eqref{eq:duality normalization} is added to ensure that the dual is well defined up to unique unitary isomorphism --- see Theorem \ref{thm: well defined up to unique unitary iso}.
\end{remark}

\begin{lemma}\label{lem: characterization of duals}
Let ${}_AH_B$ and ${}_BK_A$ be irreducible bimodules. If $H$ is dualizable, then
\[
\hom_{A,A}\big(L^2(A),H\boxtimes_BK\big)\cong \begin{cases}
\IC&\text{if ${}_B\bar H_A\cong{}_BK_A$,}\\
0&\text{otherwise.}\\
\end{cases}
\]
\end{lemma}
\begin{proof}
The map $f\mapsto(S^*\otimes 1)(1\otimes f)$ is an isomorphism between the vector spaces
$\hom({}_AL^2(A)_A,{}_AH\boxtimes_BK_A)$ and $\hom({}_B\bar H_A,{}_BK_A)$.
\end{proof}

We will see later, in Lemma \ref{lem: direct summand of dualizable bimodule},
that given two $A$-$B$-bimodules, their direct sum is dualizable if and only if they are both dualizable. 
One direction is straightforward, and is given presently as Lemma \ref{lem: sum of dualizables}.
Similarly, given a non-zero $A$-$B$-bimodule and a non-zero $B$-$C$-bimodule, their Connes fusion is dualizable if and only if they are both dualizable.  Again one direction is easier, and is given here as Lemma \ref{lem: product of dualizables}.  The other direction is established in Corollary \ref{cor: HK finite ==> H finite}.

\begin{lemma}\label{lem: sum of dualizables}
Let ${}_AH_B$ and ${}_AK_B$ be dualizable bimodules, with respective structure maps $R$, $S$, $\tilde R$, and $\tilde S$.
Then ${}_A(H\oplus K)_B$ is dualizable, with dual ${}_A(\bar H\oplus \bar K)_B$, and structure maps
\[\quad
\textstyle\begin{pmatrix}R \\ 0 \\ 0\\ \tilde R\end{pmatrix}\!:L^2(A) \rightarrow (H\oplus K)\underset B \boxtimes (\bar H\oplus \bar K)\,,\quad
\textstyle\begin{pmatrix}S \\ 0 \\ 0\\ \tilde S\end{pmatrix}\!:L^2(B) \rightarrow  (\bar H\oplus \bar K)\underset A \boxtimes (H\oplus K). \hfill\quad\square
\]
\end{lemma}

\begin{lemma}\label{lem: product of dualizables}
Let ${}_AH_B$ and ${}_BK{}_C$ be dualizable bimodules, with respective structure maps $R$, $S$, and $\tilde R$, $\tilde S$.
Their fusion ${}_AH\boxtimes_BK{}_C$ is then also dualizable, with dual ${}_C\bar K\boxtimes_B\bar H_A$, and structure maps
$\widehat R := (1\otimes \tilde R\otimes 1)R$ and $\widehat S := (1\otimes S\otimes 1)\tilde S$.
\end{lemma}
\begin{proof}
The duality equations \eqref{eq: duality equations} for $\widehat R$ and $\widehat S$ are straightforward.
To verify the normalization condition \eqref{eq:duality normalization}, we make use of the graphical calculus introduced earlier:
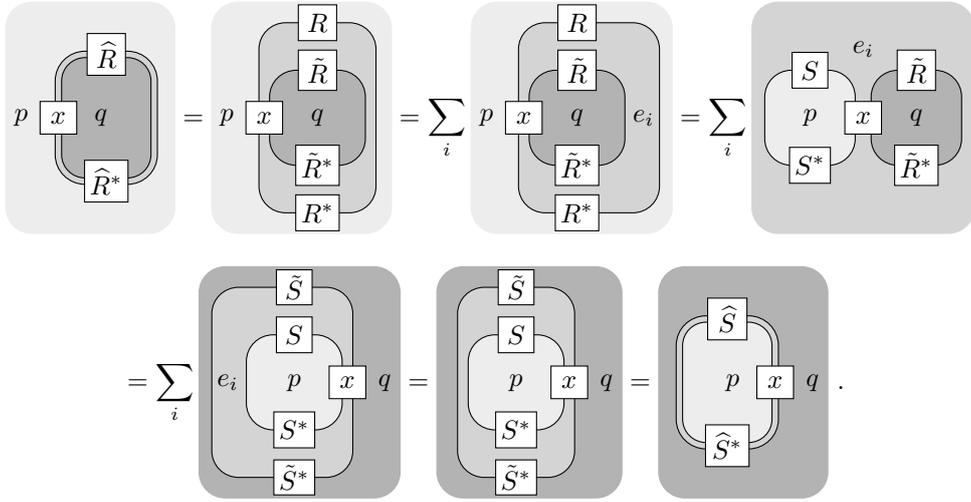
\[
\def\colA{black!7}
\def\colB{black!17}
\def\colC{black!30}
\tikzmath[scale=\squarescale]
	{\fill[rounded corners=10, fill=\colA] (-19,-22) rectangle (13,22);
	\draw[rounded corners=10, double, double distance=2, fill=\colC, double=\colB] (-9,-12) rectangle (9,12);
	\draw 	(-9,0) node[fill=white, draw, inner sep=4]{$x$}
			(0,-12) node[fill=white, draw, inner xsep=2, inner ysep=3]{$\widehat R^*$}
			(0,12) node[fill=white, draw, inner ysep=3] {$\widehat R$}
			(-1,0) node {$q$}
			(-16,0) node {$p$};
	}
=\tikzmath[scale=\squarescale]
	{\fill[rounded corners=10, fill=\colA] (-20,-22) rectangle (14,22);
	\draw[rounded corners=10, fill=\colB] (-11,-18) rectangle (11,18);
	\draw[rounded corners=10, fill=\colC] (-9,-9) rectangle (9,9);
	\draw 	(-10,0) node[fill=white, draw, inner sep=4]{$x$}
			(0,-18) node[fill=white, draw, inner xsep=2]{$R^*$}
			(0,18) node[fill=white, draw] {$R$}
			(0,-9) node[fill=white, draw, inner ysep=3, inner xsep=2]{$\tilde R^*$}
			(0,9) node[fill=white, draw, inner ysep=3] {$\tilde R$}
			(0,0) node {$q$}
			(-17,0) node {$p$};
	}
=\sum_i\tikzmath[scale=\squarescale]
	{\fill[rounded corners=10, fill=\colA] (-20,-22) rectangle (18,22);
	\draw[rounded corners=10, fill=\colB] (-11,-18) rectangle (15.5,18);
	\draw[rounded corners=10, fill=\colC] (-9,-9) rectangle (9,9);
	\draw 	(-10,0) node[fill=white, draw, inner sep=4]{$x$}
			(0,-18) node[fill=white, draw, inner xsep=2]{$R^*$}
			(0,18) node[fill=white, draw] {$R$}
			(0,-9) node[fill=white, draw, inner ysep=3, inner xsep=2]{$\tilde R^*$}
			(0,9) node[fill=white, draw, inner ysep=3] {$\tilde R$}
			(0,0) node {$q$}
			(-17,0) node {$p$}
			(12.5,0) node {$e_i$};
	}
=\sum_i\tikzmath[scale=\squarescale]
	{\fill[rounded corners=10, fill=\colB] (-21,-22) rectangle (21,22);
	\draw[rounded corners=10, fill=\colA] (-18.5,-9) rectangle (-1.5,9);
	\draw[rounded corners=10, fill=\colC] (1.5,-9) rectangle (18.5,9);
	\draw 	(0,0) node[fill=white, draw, inner sep=4]{$x$}
			(-10,-9) node[fill=white, draw, inner xsep=2]{$S^*$}
			(-10,9) node[fill=white, draw] {$S$}
			(10,-9) node[fill=white, draw, inner ysep=3, inner xsep=2]{$\tilde R^*$}
			(10,9) node[fill=white, draw, inner ysep=3] {$\tilde R$}
			(10,0) node {$q$}
			(-10,0) node {$p$}
			(0,13) node {$e_i$};
	}
\]\[
\def\colA{black!7}
\def\colB{black!17}
\def\colC{black!30}
=\sum_i\tikzmath[scale=\squarescale]
	{\fill[rounded corners=10, fill=\colC] (20,-22) rectangle (-18,22);
	\draw[rounded corners=10, fill=\colB] (11,-18) rectangle (-15.5,18);
	\draw[rounded corners=10, fill=\colA] (9,-9) rectangle (-9,9);
	\draw 	(10,0) node[fill=white, draw, inner sep=4]{$x$}
			(0,-18) node[fill=white, draw, inner ysep=2, inner xsep=2]{$\tilde S^*$}
			(0,18) node[fill=white, draw, inner ysep=2] {$\tilde S$}
			(0,-9) node[fill=white, draw, inner xsep=2]{$S^*$}
			(0,9) node[fill=white, draw] {$S$}
			(0,0) node {$p$}
			(17,0) node {$q$}
			(-12.5,0) node {$e_i$};
	}
=\tikzmath[scale=\squarescale]
	{\fill[rounded corners=10, fill=\colC] (20,-22) rectangle (-15,22);
	\draw[rounded corners=10, fill=\colB] (11,-18) rectangle (-11,18);
	\draw[rounded corners=10, fill=\colA] (-9,-9) rectangle (9,9);
	\draw 	(10,0) node[fill=white, draw, inner sep=4]{$x$}
			(0,-18) node[fill=white, draw, inner ysep=2, inner xsep=2]{$\tilde S^*$}
			(0,18) node[fill=white, draw, inner ysep=2] {$\tilde S$}
			(0,-9) node[fill=white, draw, inner xsep=2]{$S^*$}
			(0,9) node[fill=white, draw] {$S$}
			(0,0) node {$p$}
			(17,0) node {$q$};
	}
=\tikzmath[scale=\squarescale]
	{\fill[rounded corners=10, fill=\colC] (19,-22) rectangle (-13,22);
	\draw[rounded corners=10, double, double distance=2, fill=\colA, double=\colB] (9,-12) rectangle (-9,12);
	\draw 	(9,0) node[fill=white, draw, inner sep=4]{$x$}
			(0,-12) node[fill=white, draw, inner xsep=2.5, inner ysep=3]{$\widehat S^*$}
			(0,12) node[fill=white, draw, inner xsep=4, inner ysep=3] {$\widehat S$}
			(1,0) node {$p$}
			(16,0) node {$q$};
	}\,\,.
\]
Here, $e_i\in Z(B)$ are the minimal central projections of $B$.
The shades correspond to the algebras $A$, $B$, $C$, and the lines stand for the bimodules $H$, $\bar H$, $K$, and $\bar K$.
\end{proof}

{\def\colA{black!10} \def\colB{black!30}
We henceforth often abbreviate the maps $R:{}_AL^2(A)_A\to{}_AH\boxtimes_B\bar H{}_A$ and $S:{}_BL^2(B)_B \rightarrow  {}_B\bar H\boxtimes_A H_B$ as $
\tikzmath[scale=.085]{ \fill[fill=\colA] (-4,-2.5) [rounded corners=3.5]-- (-4,2.5) --(4,2.5) -- (4,-2.5) [sharp corners]-- (2,-2.5) [rounded corners = 4.8]-- (2,.5) -- (-2,.5) [sharp corners]-- (-2,-2.5) [rounded corners=3.5]-- cycle; \fill[fill=\colB] (-2,-2.5) -- (2,-2.5) [rounded corners = 4.8]-- (2,.5) -- (-2,.5) [sharp corners]-- cycle; \draw (2,-2.5) [rounded corners = 4.8]-- (2,.5) -- (-2,.5) -- (-2,-2.5);}$ 
and $
\tikzmath[scale=.085]{ \fill[fill=\colB] (-4,-2.5) [rounded corners=3.5]-- (-4,2.5) --(4,2.5) -- (4,-2.5) [sharp corners]-- (2,-2.5) [rounded corners = 4.8]-- (2,.5) -- (-2,.5) [sharp corners]-- (-2,-2.5) [rounded corners=3.5]-- cycle; \fill[fill=\colA] (-2,-2.5) -- (2,-2.5) [rounded corners = 4.8]-- (2,.5) -- (-2,.5) [sharp corners]-- cycle;\draw (2,-2.5) [rounded corners = 4.8]-- (2,.5) -- (-2,.5) -- (-2,-2.5);}$ 
}respectively.  We will show, in Theorem \ref{thm: normalizing duality data}, that a bimodule between von Neumann algebras with finite-dimensional centers is ``non-normalized dualizable" if and only if it is dualizable.  We first record two lemmas regarding consequences of the duality equations \eqref{eq: duality equations}.

\begin{lemma}\label{lem: D >= 1}
Let ${}_AH_B$ be a non-zero bimodule between factors.
If $R$ and $S$ are maps as in \eqref{eq:duality maps} satisfying \eqref{eq: duality equations}, then $R$ and $S$ are injective and $(R^*R)(S^*S)\ge 1$.
\end{lemma}
\begin{proof}
\def\colA{black!10}
\def\colB{black!30}
The expressions 
$R^*R=\tikzmath[scale=.085]{\fill[fill=\colA, rounded corners=3.2] (-3,-2.5) rectangle (3,2.5);\filldraw[fill=\colB] (0,0) circle (1.4);}$\, 
and
$S^*S=\tikzmath[scale=.085]{\fill[fill=\colB, rounded corners=3.2] (-3,-2.5) rectangle (3,2.5);\filldraw[fill=\colA] (0,0) circle (1.4);}$\, 
are in $\IC$ (in fact in $\IR$) because $A$ and $B$ are factors.
As $H$ is non-zero and $A$ and $B$ are factors, $H$ is faithful, both as
an $A$-module and a $B^{\op}$-module.
By~\eqref{eq: duality equations} and Lemma~\ref{lem:stay-injective}, this implies that $S$ and $R$ are injective.
In particular, 
$\tikzmath[scale=.085]{\fill[fill=\colA, rounded corners=3.2] (-3,-2.5) rectangle (3,2.5);\filldraw[fill=\colB] (0,0) circle (1.4);}$\, 
and 
$\tikzmath[scale=.085]{\fill[fill=\colB, rounded corners=3.2] (-3,-2.5) rectangle (3,2.5);\filldraw[fill=\colA] (0,0) circle (1.4);}$\, 
are nonzero.
Letting
$e_1:=\big(\tikzmath[scale=.085]{ \fill[fill=\colA, rounded corners=3.2] (-3,-2.5) rectangle (3,2.5); \filldraw[fill=\colB] (0,0) circle (1.4);} 
\big)^{-1}\!\!\cdot
\tikzmath[scale=\squarescale]{\fill[fill=\colA] (-4,-3) [rounded corners=2]-- (-4,3) [sharp corners]-- (-2,3) [rounded corners = 4]-- (-2,1) -- (2,1) [sharp corners]-- (2,3) --(5,3) -- (5,-3) -- (2,-3) [rounded corners = 4]-- (2,-1) -- (-2,-1) [sharp corners]-- (-2,-3) [rounded corners=2]-- cycle; \fill[fill=\colB] (-2,-3) -- (2,-3) [rounded corners = 4]-- (2,-1) -- (-2,-1) [sharp corners]-- cycle;
\fill[fill=\colB] (-2,3) -- (2,3) [rounded corners = 4]-- (2,1) -- (-2,1) [sharp corners]-- cycle; \draw (2,3) [rounded corners = 4]-- (2,1) -- (-2,1) -- (-2,3); \draw (2,-3) [rounded corners = 4]-- (2,-1) -- (-2,-1) -- (-2,-3); \fill[fill=\colB] (7,-3) [rounded corners=2]-- (7,3) [sharp corners]-- (5,3) -- (5,-3) [rounded corners=2]-- cycle; \draw (5,3) -- (5,-3);}$ 
and
$e_2:=\big(\tikzmath[scale=.085]{\fill[fill=\colB, rounded corners=3.2] (-3,-2.5) rectangle (3,2.5); \filldraw[fill=\colA] (0,0) circle (1.4);} 
\big)^{-1}\!\!\cdot
\tikzmath[scale=\squarescale]{\fill[fill=\colB] (4,-3) [rounded corners=2]-- (4,3) [sharp corners]-- (2,3) [rounded corners = 4]-- (2,1) -- (-2,1) [sharp corners]-- (-2,3) --(-5,3) -- (-5,-3) -- (-2,-3) [rounded corners = 4]-- (-2,-1) -- (2,-1) [sharp corners]-- (2,-3) [rounded corners=2]-- cycle; \fill[fill=\colA] (2,-3) -- (-2,-3) [rounded corners = 4]-- (-2,-1) -- (2,-1) [sharp corners]-- cycle;
\fill[fill=\colA] (2,3) -- (-2,3) [rounded corners = 4]-- (-2,1) -- (2,1) [sharp corners]-- cycle;\draw (-2,3) [rounded corners = 4]-- (-2,1) -- (2,1) -- (2,3);\draw (-2,-3) [rounded corners = 4]-- (-2,-1) -- (2,-1) -- (2,-3); \fill[fill=\colA] (-7,-3) [rounded corners=2]-- (-7,3) [sharp corners]-- (-5,3) -- (-5,-3) [rounded corners=2]-- cycle; \draw (-5,3) -- (-5,-3); }$\, 
(these are the Jones projections), we have
\[
e_1=e_1e_1\ge e_1e_2e_1 =
\big(\tikzmath[scale=.085]{\fill[fill=\colA, rounded corners=3.2] (-3,-2.5) rectangle (3,2.5);\filldraw[fill=\colB] (0,0) circle (1.4);} 
\cdot \tikzmath[scale=.085]{\fill[fill=\colB, rounded corners=3.2] (-3,-2.5) rectangle (3,2.5);\filldraw[fill=\colA] (0,0) circle (1.4);} 
\big)^{-1}\!
e_1 \Rightarrow 1\ge \big(\tikzmath[scale=.085]{\fill[fill=\colA, rounded corners=3.2] (-3,-2.5) rectangle (3,2.5);\filldraw[fill=\colB] (0,0) circle (1.4);} 
\cdot \tikzmath[scale=.085]{\fill[fill=\colB, rounded corners=3.2] (-3,-2.5) rectangle (3,2.5);\filldraw[fill=\colA] (0,0) circle (1.4);} 
\big)^{-1} \Rightarrow \, \tikzmath[scale=.085]{\fill[fill=\colA, rounded corners=3.2] (-3,-2.5) rectangle (3,2.5);\filldraw[fill=\colB] (0,0) circle (1.4);} 
\cdot \tikzmath[scale=.085]{\fill[fill=\colB, rounded corners=3.2] (-3,-2.5) rectangle (3,2.5);\filldraw[fill=\colA] (0,0) circle (1.4);} 
\ge 1. \qedhere
\]
\end{proof}

The next lemma is similar to \cite[Lemma 3.2]{Longo-Roberts(A-theory-of-dimension)}.

\begin{lemma}\label{lem: Lemma B}
Let ${}_AH_B$ be a bimodule between factors.
If there exist maps $R$ and $S$ as in \eqref{eq:duality maps} satisfying \eqref{eq: duality equations},
then ${}_AH_B$ is a finite direct sum of irreducible bimodules; its algebra of bimodule endomorphisms is therefore finite-dimensional.
Moreover, the (non-normalized) state $\varphi:\mathrm{End}({}_AH_B)\to \IC$ given by
\begin{equation}\label{eq: varphi:End(H) --> C}
\def\colA{black!10}
\def\colB{black!30}
\varphi\,:\,x\,\mapsto\,
\tikzmath[scale=\squarescale]
	{\fill[rounded corners=10, fill=\colA] (-13,-12) rectangle (12,12);
	\draw[rounded corners=10, fill=\colB] (-5,-7) rectangle (5,7);
	\draw 	(-5,0) node[fill=white, draw, inner sep=4]{$x$};
	}
\end{equation}
is faithful.
\end{lemma}

\begin{proof}
\def\colA{black!10}
\def\colB{black!30}
For any non-zero projection $p\in\mathrm{End}({}_AH_B)$, we have
\[
\begin{split}
1=\big\|p\big\|
&=\Bigg\|
\tikzmath[scale=\squarescale]
	{\fill[rounded corners=9, fill=\colA] (-15,-9) rectangle (15,9);
	\fill[rounded corners=9, fill=\colB] (9,9) -- (9,-7) -- (0,-7) -- (0,7) -- (-9,7) [sharp corners]-- (-9,-9)  [rounded corners=9]-- (15,-9) -- (15,9) [sharp corners]-- cycle;
	\draw[rounded corners=9] (9,9) -- (9,-7) -- (0,-7) -- (0,7) -- (-9,7) -- (-9,-9);
	\draw (-9,0) node[fill=white, draw]{$p$};}
\Bigg\|
\,\le\,
\Bigg\|
\tikzmath[scale=\squarescale]
	{\clip[rounded corners=10] (-15,-2) rectangle (14,16);
	\fill[rounded corners=10, fill=\colA] (-15,-16) rectangle (15,16);
	\fill[rounded corners=9, fill=\colB] (8,16) -- (8,-12) -- (0,-12) -- (0,12) -- (-9,12) [sharp corners]-- (-9,-16)  [rounded corners=10]-- (15,-16) -- (15,16) [sharp corners]-- cycle;
	\draw[rounded corners=9] (8,16) -- (8,-12) -- (0,-12) -- (0,12) -- (-9,12) -- (-9,-16);
	\draw (-9,5) node[fill=white, draw]{$p$};}
\Bigg\|\cdot
\Bigg\|
\tikzmath[scale=\squarescale]
	{\clip[rounded corners=10] (-14,-16) rectangle (15,2);
	\fill[rounded corners=10, fill=\colA] (-15,-16) rectangle (15,16);
	\fill[rounded corners=9, fill=\colB] (9,16) -- (9,-12) -- (0,-12) -- (0,12) -- (-8,12) [sharp corners]-- (-8,-16)  [rounded corners=10]-- (15,-16) -- (15,16) [sharp corners]-- cycle;
	\draw[rounded corners=9] (9,16) -- (9,-12) -- (0,-12) -- (0,12) -- (-8,12) -- (-8,-16);}
\Bigg\|
\\&=
\Bigg\|
\tikzmath[scale=\squarescale]
	{\clip[rounded corners=10] (-15,-2) rectangle (6,16);
	\fill[rounded corners=10, fill=\colA] (-15,-16) rectangle (15,16);
	\fill[rounded corners=9, fill=\colB] (9,16) -- (9,-12) -- (0,-12) -- (0,12) -- (-9,12) [sharp corners]-- (-9,-16)  [rounded corners=10]-- (15,-16) -- (15,16) [sharp corners]-- cycle;
	\draw[rounded corners=9] (9,16) -- (9,-12) -- (0,-12) -- (0,12) -- (-9,12) -- (-9,-16);
	\draw (-9,5) node[fill=white, draw]{$p$};}
\Bigg\|\cdot
\Bigg\|
\tikzmath[scale=\squarescale]
	{\clip[rounded corners=10] (-6,-16) rectangle (15,2);
	\fill[rounded corners=10, fill=\colA] (-15,-16) rectangle (15,16);
	\fill[rounded corners=9, fill=\colB] (9,16) -- (9,-12) -- (0,-12) -- (0,12) -- (-9,12) [sharp corners]-- (-9,-16)  [rounded corners=10]-- (15,-16) -- (15,16) [sharp corners]-- cycle;
	\draw[rounded corners=9] (9,16) -- (9,-12) -- (0,-12) -- (0,12) -- (-9,12) -- (-9,-16);}
\Bigg\|=\sqrt{\phantom{\big\|}\!\!\!\varphi(p)}\cdot
\sqrt{\tikzmath[scale=.085]{\fill[fill=\colB, rounded corners=3.2] (-3,-2.5) rectangle (3,2.5);\filldraw[fill=\colA] (0,0) circle (1.4);} 
}\,,
\end{split}
\]
where the last step follows from the general identity $\|a^*a\|=\|a\|^2$.
Let $c:=(\tikzmath[scale=.085]{\fill[fill=\colB, rounded corners=3.2] (-3,-2.5) rectangle (3,2.5);\filldraw[fill=\colA] (0,0) circle (1.4);})^{-1}$.
By the above estimate, we have $\varphi(p)\ge c$ for any non-zero projection $p$.
In particular, $\varphi$ is faithful.
If $H$ failed to be a finite direct sum of irreducible bimodules, we could pick countably many non-zero mutually orthogonal projections $p_n\in \mathrm{End}({}_AH_B)$, and get
\[
\varphi(1)> \varphi\big(\sum_{n=1}^N p_n\big)=\sum_{n=1}^N \varphi(p_n)\ge \sum_{n=1}^N c = Nc
\]
for every $N$. This is clearly impossible. Our bimodule is therefore a finite direct sum of irreducible ones and its endomorphism algebra is finite-dimensional.
\end{proof}

We can now prove that a bimodule that admits a not-necessarily normalized dual in fact admits a normalized dual:

\begin{theorem} \label{thm: normalizing duality data}
Let ${}_AH_B$ and ${}_B\bar H_A$ be bimodules between von Neumann algebras with finite-dimensional center, and let 
\begin{equation}\label{eq: tR and tS}
\tilde R:{}_AL^2(A)_A \rightarrow {}_AH\boxtimes_B \bar H_A\qquad \text{and}\qquad
\tilde S:{}_BL^2(B)_B \rightarrow  {}_B\bar H\boxtimes_A H_B
\end{equation}
be bimodule maps satisfying \eqref{eq: duality equations}.
Then it is possible to find new maps $R$ and $S$ as in \eqref{eq:duality maps} that satisfy both \eqref{eq: duality equations} and \eqref{eq:duality normalization}.
\end{theorem}

\begin{proof}
\def\colA{black!10}
\def\colB{black!30}
We first assume that $A$ and $B$ are factors.
For this proof, we write $
\tikzmath[scale=.085]{ \fill[fill=\colA] (-4,-2.5) [rounded corners=3.5]-- (-4,2.5) --(4,2.5) -- (4,-2.5) [sharp corners]-- (2,-2.5) [rounded corners = 4.8]-- (2,.5) -- (-2,.5) [sharp corners]-- (-2,-2.5) [rounded corners=3.5]-- cycle; \fill[fill=\colB] (-2,-2.5) -- (2,-2.5) [rounded corners = 4.8]-- (2,.5) -- (-2,.5) [sharp corners]-- cycle; \draw (2,-2.5) [rounded corners = 4.8]-- (2,.5) -- (-2,.5) -- (-2,-2.5);}$ 
for $\tilde R$ and $
\tikzmath[scale=.085]{ \fill[fill=\colB] (-4,-2.5) [rounded corners=3.5]-- (-4,2.5) --(4,2.5) -- (4,-2.5) [sharp corners]-- (2,-2.5) [rounded corners = 4.8]-- (2,.5) -- (-2,.5) [sharp corners]-- (-2,-2.5) [rounded corners=3.5]-- cycle; \fill[fill=\colA] (-2,-2.5) -- (2,-2.5) [rounded corners = 4.8]-- (2,.5) -- (-2,.5) [sharp corners]-- cycle;\draw (2,-2.5) [rounded corners = 4.8]-- (2,.5) -- (-2,.5) -- (-2,-2.5);}$ 
for $\tilde S$,
and let $\varphi,\psi:\mathrm{End}({}_AH_B)\to \IC$ be given by
\[
\varphi:m\mapsto
\tikzmath[scale=\squarescale]
	{\fill[rounded corners=10, fill=\colA] (-12,-11) rectangle (11,11);
	\draw[rounded corners=10, fill=\colB] (-5,-7) rectangle (5,7);
	\draw 	(-5,0) node[fill=white, draw, inner sep=4]{$m$};
	}\quad\text{and}\quad
\psi:m\mapsto
\tikzmath[scale=\squarescale]
	{\fill[rounded corners=10, fill=\colB] (-11,-11) rectangle (12,11);
	\draw[rounded corners=10, fill=\colA] (-5,-7) rectangle (5,7);
	\draw 	(5,0) node[fill=white, draw, inner sep=4]{$m$};
	}\,.
\]
The state $\varphi$ is faithful by the previous lemma, and so is $\psi$ for similar reasons.
Pick a trace $\tau:\mathrm{End}({}_AH_B)\to \IC$; one 
exists because the algebra is finite-dimensional.
Let $a,b\in\mathrm{End}({}_AH_B)$ be the unique solutions to the equations $\varphi = a\tau$ and
$\psi = b\tau$; here, we use the action of the algebra $\mathrm{End}({}_AH_B)$ on its 
 $L^1$-space, as introduced in section~\ref{sec: Preliminaries}.  Since $\varphi$ and $\psi$ are positive and faithful, $a$ and $b$ are positive and invertible.

The new structure maps $R$ and $S$ are given in terms of the old ones $\tilde R$ and $\tilde S$ by
\[
R:=(x\otimes 1)\tilde R=
\tikzmath[scale=\squarescale]
	{\clip[rounded corners=10] (-13,-6) rectangle (12,12); 
	\fill[rounded corners=10, fill=\colA] (-13,-12) rectangle (12,12);
	\draw[rounded corners=10, fill=\colB] (-5,-12) rectangle (5,8);
	\draw 	(-5,1) node[fill=white, draw]{$x$};
	}
\quad\text{and}\quad
S:=(1\otimes x^{-1})\tilde S=
\tikzmath[scale=\squarescale]
	{\clip[rounded corners=10] (-12,-6) rectangle (13,12); 
	\fill[rounded corners=10, fill=\colB] (-12,-12) rectangle (13,12);
	\draw[rounded corners=10, fill=\colA] (-5,-12) rectangle (5,8);
	\draw 	(5,1) node[fill=white, draw, inner sep=2]{$x^{-1}$};
	}
\]
for some appropriately chosen positive element $x\in \mathrm{End}({}_A H_B)$.
Clearly $R$ and $S$ satisfy the duality equations \eqref{eq: duality equations}.
To ensure that they also satisfy the normalization equation \eqref{eq:duality normalization},
the element $x$ needs to satisfy $\varphi(xyx)=\psi(x^{-1}yx^{-1})$ for all $y\in \mathrm{End}({}_A H_B)$,
which is to say $x\varphi x = x^{-1}\psi x^{-1}$ or, equivalently, $xax = x^{-1}bx^{-1}$.
That equation has a unique positive solution\footnote{Courtesy of {\sf http://mathoverflow.net/questions/70838}.}:
\[
\begin{split}
xax = x^{-1}bx^{-1}
&\Leftrightarrow x^2ax^2 = b \\
&\Leftrightarrow \sqrt{a}x^2ax^2\sqrt{a} = \sqrt{a}b\sqrt{a}\\
&\Leftrightarrow \sqrt{a}x^2\sqrt{a} = \sqrt{\sqrt{a}b\sqrt{a}\,}\\
&\Leftrightarrow x^2 = \sqrt{a}^{-1}\sqrt{\sqrt{a}b\sqrt{a}\,}\sqrt{a}^{-1}\\
&\Leftrightarrow x = \sqrt{\sqrt{a}^{-1}\sqrt{\sqrt{a}b\sqrt{a}\,}\sqrt{a}^{-1}}\, .
\end{split}
\]

When $A=\bigoplus A_i$ and $B=\bigoplus B_j$ are direct sums of factors, then we can write $H$ as a direct sum of $A_i$-$B_j$-bimodules $H=\bigoplus H_{ij}$, and similarly $\bar H=\bigoplus \bar H_{ji}$.
We also have $L^2A\cong\bigoplus L^2A_i$ and $L^2B\cong\bigoplus L^2B_j$ by Lemma \ref{lem: L^2(pAp) = pL^2(A)p}.
The maps \eqref{eq: tR and tS} induce structure maps $\tilde R_{ij}:L^2A_i \rightarrow H_{ij}\boxtimes_{B_j} \bar H_{ji}$
and $\tilde S_{ij}:L^2B_j \rightarrow  \bar H_{ji}\boxtimes_{A_i} H_{ij}$ to which we can apply the above argument and get
\[
R_{ij}:{}_{A_i}L^2(A_i)_{A_i} \rightarrow {}_{A_i}H_{ij}\boxtimes_{B_j} \bar H_{ji}\,{}_{A_i}\quad \text{and}\quad
S_{ij}:{}_{B_j}L^2(B_j)_{B_j} \rightarrow  {}_{B_j}\bar H_{ji}\boxtimes_{A_i} H_{ij}\,{}_{B_j}
\]
subject to \eqref{eq: duality equations} and \eqref{eq:duality normalization}.
The desired maps $R$ and $S$ are then given by
\[
L^2A\cong \bigoplus_i L^2A_i \xrightarrow{\bigoplus_{ij}R_{ij}} \bigoplus_{ij} \big(H_{ij}\underset{B_j}\boxtimes \bar H_{ji}\big)\subset \bigoplus_{ijk} \big(H_{ij}\underset{B_j}\boxtimes \bar H_{jk}\big) \cong H\boxtimes_B \bar H
\]
and
\[
L^2B\cong \bigoplus_j L^2B_j \xrightarrow{\bigoplus_{ij}S_{ij}} \bigoplus_{ij} \big(\bar H_{ji}\underset{A_i}\boxtimes H_{ij}\big)\subset \bigoplus_{lij} \big(\bar H_{li}\underset{A_i}\boxtimes H_{ij}\big) \cong \bar H\boxtimes_A H. \qedhere
\]
\end{proof}

\begin{remark}\label{remark: non-zero maps are enough}
We will see later, in Proposition \ref{Prop: existence of non-zero
maps is enough},
that when $H$ is irreducible the mere existence of non-zero maps $\tilde R:L^2(A) \rightarrow
H\boxtimes_B \bar H$ and
$\tilde S:L^2(B) \rightarrow  \bar H\boxtimes_A H$ implies that
${}_AH_B$ is dualizable.
\end{remark}

We now discuss two lemmas that we will need in order to prove, in Theorem \ref{thm: well defined up to unique unitary iso}, that the dual is well defined up to unique unitary isomorphism.

\begin{lemma}\label{lem: phi is a trace}
Let $A$ and $B$ be factors, and let ${}_AH_B$ be a dualizable bimodule, with structure maps $R$ and $S$ subject to \eqref{eq: duality equations} and \eqref{eq:duality normalization}.
Then the state $\varphi$ defined in \eqref{eq: varphi:End(H) --> C} is a trace.
\end{lemma}
\begin{proof}
By a few applications of \eqref{eq: duality equations} and some planar isotopies, we get
\def\colA{black!10}
\def\colB{black!30}
\begin{equation}\label{eq: twirlies}
\hspace{-.5cm}\tikzmath[scale=\squarescale]
	{\fill[rounded corners=10, fill=\colA] (-11,-19) rectangle (10,19);
	\draw[rounded corners=8, fill=\colB] (-4,-13) -- (5,-13) -- (5,13) -- (-4,13) -- cycle;
	\draw 	(-4,5) node[fill=white, draw, inner xsep=3, inner ysep=4]{$x$};
	\draw 	(-4,-5) node[fill=white, draw, inner sep=3]{$y$};
	}\,=\,
\tikzmath[scale=\squarescale]
	{\fill[rounded corners=10, fill=\colA] (-13,-22) rectangle (11,16);
	\draw[rounded corners=5, fill=\colB]  (-4,-7) 
	+(0,7) [rounded corners=6]-- +(-5.5,7) [rounded corners=5]-- +(-5.5,-5.5) [rounded corners=4]-- +(4,-5.5) [rounded corners=3]-- +(4,4) -- +(0,4) -- +(0,-4) -- +(-4,-4) [rounded corners=4]-- +(-4,5.5) [rounded corners=5]-- +(5.5,5.5) [rounded corners=6]-- +(5.5,-7) [rounded corners=3]-- +(-.15,-7) 
	[rounded corners=8]-- (-4.15,-20) -- (5,-20) -- (5,13) -- (-4,13) [rounded corners=4]-- cycle;
	\draw 	(-4,6) node[fill=white, draw, inner sep=3]{$y$};
	\draw 	(-4,-7) node[fill=white, draw, inner xsep=2, inner ysep=3]{$x$};
	}
\qquad\text{and}\qquad
\tikzmath[scale=\squarescale]
	{\fill[rounded corners=10, fill=\colB] (-10,-19) rectangle (11,19);
	\draw[rounded corners=8, fill=\colA] (-4,-13) -- (5,-13) -- (5,13) -- (-4,13) -- cycle;
	\draw 	(4,5) node[fill=white, draw, inner xsep=3, inner ysep=4]{$x$};
	\draw 	(4,-5) node[fill=white, draw, inner sep=3]{$y$};
	}\,=\,
\tikzmath[scale=\squarescale]
	{\fill[rounded corners=10, fill=\colB] (-11,-22) rectangle (13,16);
	\draw[rounded corners=5, fill=\colA]  (4,-7) 
	+(0,7) [rounded corners=6]-- +(5.5,7) [rounded corners=5]-- +(5.5,-5.5) [rounded corners=4]-- +(-4,-5.5)[rounded corners=3] -- +(-4,4) -- +(0,4) -- +(0,-4) -- +(4,-4) [rounded corners=4]-- +(4,5.5) [rounded corners=5]-- +(-5.5,5.5) [rounded corners=6]-- +(-5.5,-7) [rounded corners=3]-- +(.15,-7) 
	[rounded corners=8]-- (4.15,-20) -- (-5,-20) -- (-5,13) -- (4,13) [rounded corners=4]-- cycle;
	\draw 	(4,6) node[fill=white, draw, inner sep=3]{$y$};
	\draw 	(4,-7) node[fill=white, draw, inner xsep=2, inner ysep=3]{$x$};
	}\,.
\end{equation}
Combining these equations with \eqref{eq:duality normalization} yields
\[
\hspace{.5cm}\tikzmath[scale=\squarescale]
	{\fill[rounded corners=10, fill=\colA] (-13,-22) rectangle (11,16);
	\draw[rounded corners=5, fill=\colB]  (-4,-7) 
	+(0,7) [rounded corners=6]-- +(-5.5,7) [rounded corners=5]-- +(-5.5,-5.5) [rounded corners=4]-- +(4,-5.5) [rounded corners=3]-- +(4,4) -- +(0,4) -- +(0,-4) -- +(-4,-4) [rounded corners=4]-- +(-4,5.5) [rounded corners=5]-- +(5.5,5.5) [rounded corners=6]-- +(5.5,-7) [rounded corners=3]-- +(-.15,-7) 
	[rounded corners=8]-- (-4.15,-20) -- (5,-20) -- (5,13) -- (-4,13) [rounded corners=4]-- cycle;
	\draw 	(-4,6) node[fill=white, draw, inner sep=3]{$y$};
	\draw 	(-4,-7) node[fill=white, draw, inner xsep=2, inner ysep=3]{$x$};
	}\,=\,
\tikzmath[scale=\squarescale]
	{\fill[rounded corners=10, fill=\colA] (-11,-19) rectangle (10,19);
	\draw[rounded corners=8, fill=\colB] (-4,-13) -- (5,-13) -- (5,13) -- (-4,13) -- cycle;
	\draw 	(-4,5) node[fill=white, draw, inner xsep=3, inner ysep=4]{$x$};
	\draw 	(-4,-5) node[fill=white, draw, inner sep=3]{$y$};
	}\,=\,
\tikzmath[scale=\squarescale]
	{\fill[rounded corners=10, fill=\colB] (-10,-19) rectangle (11,19);
	\draw[rounded corners=8, fill=\colA] (-4,-13) -- (5,-13) -- (5,13) -- (-4,13) -- cycle;
	\draw 	(5,5) node[fill=white, draw, inner xsep=3, inner ysep=4]{$x$};
	\draw 	(5,-5) node[fill=white, draw, inner sep=3]{$y$};
	}\,=\,
\tikzmath[scale=\squarescale]
	{\fill[rounded corners=10, fill=\colB] (-11,-22) rectangle (13,16);
	\draw[rounded corners=5, fill=\colA]  (4,-7) 
	+(0,7) [rounded corners=6]-- +(5.5,7) [rounded corners=5]-- +(5.5,-5.5) [rounded corners=4]-- +(-4,-5.5) [rounded corners=3]-- +(-4,4) -- +(0,4) -- +(0,-4) -- +(4,-4) [rounded corners=4]-- +(4,5.5) [rounded corners=5]-- +(-5.5,5.5) [rounded corners=6]-- +(-5.5,-7) [rounded corners=3]-- +(.15,-7) 
	[rounded corners=8]-- (4.15,-20) -- (-5,-20) -- (-5,13) -- (4,13) [rounded corners=4]-- cycle;
	\draw 	(4,6) node[fill=white, draw, inner sep=3]{$y$};
	\draw 	(4,-7) node[fill=white, draw, inner xsep=2, inner ysep=3]{$x$};
	}\,=\,
\tikzmath[scale=\squarescale]
	{\fill[rounded corners=10, fill=\colA] (-13,-22) rectangle (11,16);
	\draw[rounded corners=5, fill=\colB]  (-4,-7) 
	+(0,7) [rounded corners=6]-- +(5.5,7) [rounded corners=5]-- +(5.5,-5.5) [rounded corners=4]-- +(-4,-5.5) [rounded corners=3]-- +(-4,4) -- +(0,4) -- +(0,-4) -- +(4,-4) [rounded corners=4]-- +(4,5.5) [rounded corners=5]-- +(-5.5,5.5) [rounded corners=6]-- +(-5.5,-7) [rounded corners=3]-- +(.15,-7) 
	[rounded corners=8]-- (-3.85,-20) -- (5,-20) -- (5,13) -- (-4,13) [rounded corners=4]-- cycle;
	\draw 	(-4,6) node[fill=white, draw, inner sep=3]{$y$};
	\draw 	(-4,-7) node[fill=white, draw, inner xsep=2, inner ysep=3]{$x$};
	}\,.
\]
The latter being true for any $y\in\mathrm{End}({}_AH_B)$ and the state $\varphi$ being faithful by Lemma \ref{lem: Lemma B},
it follows that
\[
\hat x:=\,\tikzmath[scale=\squarescale]
	{\clip[rounded corners=10] (-4,-7) +(-10,-13) rectangle +(10,13);
	\fill[rounded corners=10, fill=\colA] (-15,-22) rectangle (11,16);
	\draw[rounded corners=5, fill=\colB]  (-4,-7) 
	+(0,8) [rounded corners=6]-- +(-6,8) [rounded corners=5]-- +(-6,-6) [rounded corners=4]-- +(4,-6) [rounded corners=3]-- +(4,4) -- +(0,4) -- +(0,-4) -- +(-4,-4) [rounded corners=4]-- +(-4,6) [rounded corners=5]-- +(6,6) [rounded corners=6]-- +(6,-8) [rounded corners=3]-- +(0,-8) 
	[rounded corners=8]-- (-4,-25) -- (7,-25) -- (7,13) -- (-4,13) [rounded corners=4]-- cycle;
	\draw 	(-4,-7) node[fill=white, draw, inner xsep=2, inner ysep=3]{$x$};
	}\,=\,
\tikzmath[scale=\squarescale]
	{\clip[rounded corners=10] (-4,-7) +(-10,-13) rectangle +(10,13);
	\fill[rounded corners=10, fill=\colA] (-15,-22) rectangle (11,16);
	\draw[rounded corners=5, fill=\colB]  (-4,-7) 
	+(0,8) [rounded corners=6]-- +(6,8) [rounded corners=5]-- +(6,-6) [rounded corners=4]-- +(-4,-6) [rounded corners=3]-- +(-4,4) -- +(0,4) -- +(0,-4) -- +(4,-4) [rounded corners=4]-- +(4,6) [rounded corners=5]-- +(-6,6) [rounded corners=6]-- +(-6,-8) [rounded corners=3]-- +(0,-8) 
	[rounded corners=8]-- (-4,-25) -- (7,-25) -- (7,13) -- (-4,13) [rounded corners=4]-- cycle;
	\draw 	(-4,-7) node[fill=white, draw, inner xsep=2, inner ysep=3]{$x$};
	}\,.
\]
Equivalently, the map $x\mapsto \hat x$ is an involution.

As in the proof of the previous lemma, pick a trace $\tau$ and a positive invertible element $a$ such that $\varphi=a\tau$.
Our goal is to show that $\varphi$ is a trace; this is true provided $a$ is central.
Equation \eqref{eq: twirlies} implies $a\hat x=x a$ for all $x$.
Equivalently, we have $\hat x=a^{-1}xa$.
Because the map $x\mapsto \hat x$ is an involution, we have $x=\hat{\hat x}=a^{-2}xa^2$.  Since $a$ is positive and its square is central, $a$ is also central.
\end{proof}

As a corollary of the above proof, we see
\begin{equation}\label{eq: left half twirl = right half twirl}
\def\colA{black!10}
\def\colB{black!30}
\hspace{-.5cm}\tikzmath[scale=\squarescale]
	{\clip[rounded corners=10] (-4,-7) +(-10,-12) rectangle +(10,12);
	\fill[rounded corners=10, fill=\colA] (-15,-22) rectangle (11,16);
	\draw[rounded corners=5, fill=\colB]  (-4,-7) 
	+(0,8) [rounded corners=6]-- +(-6,8) [rounded corners=5]-- +(-6,-6) [rounded corners=4]-- +(4,-6) [rounded corners=3]-- +(4,4) -- +(0,4) -- +(0,-4) -- +(-4,-4) [rounded corners=4]-- +(-4,6) [rounded corners=5]-- +(6,6) [rounded corners=6]-- +(6,-8) [rounded corners=3]-- +(0,-8) 
	[rounded corners=8]-- (-4,-25) -- (7,-25) -- (7,13) -- (-4,13) [rounded corners=4]-- cycle;
	\draw 	(-4,-7) node[fill=white, draw, inner xsep=2, inner ysep=3]{$x$};
	}\,=\,
\tikzmath[scale=\squarescale]
	{\clip[rounded corners=10] (-4,-7) +(-10,-12) rectangle +(10,12);
	\fill[rounded corners=10, fill=\colA] (-15,-22) rectangle (11,16);
	\draw[rounded corners=5, fill=\colB]  (-4,-22) rectangle (11,16);
	\draw 	(-4,-7) node[fill=white, draw, inner xsep=3, inner ysep=4]{$x$};
	}\,,\quad\text{and thus also}\quad
\bar x:=\tikzmath[scale=\squarescale]
	{\clip[rounded corners=10] (-4,-7) +(-10,-12) rectangle +(10,12);
	\fill[rounded corners=10, fill=\colB] (-15,-22) rectangle (11,16);
	\draw[rounded corners=6, fill=\colA]  (-4,-7) 
	+(-6,20) -- +(-6,-6) -- +(0,-6) -- +(0,6) -- +(6,6) -- +(6,-20) -- (7,-25) -- (7,13) -- cycle;
	\draw 	(-4,-7) node[fill=white, draw, inner xsep=3, inner ysep=4]{$x$};
	}\,=\,
\tikzmath[scale=\squarescale]
	{\clip[rounded corners=10] (-4,-7) +(-10,-12) rectangle +(10,12);
	\fill[rounded corners=10, fill=\colB] (-15,-22) rectangle (11,16);
	\draw[rounded corners=6, fill=\colA]  (-4,-7) 
	+(6,20) -- +(6,-6) -- +(0,-6) -- +(0,6) -- +(-6,6) -- +(-6,-20) -- (7,-25) -- (7,13) -- cycle;
	\draw 	(-4,-7) node[fill=white, draw, inner xsep=3, inner ysep=4]{$x$};
	}\,.
\end{equation}

\begin{remark}
The first equation in \eqref{eq: left half twirl = right half twirl} is essentially the same as \cite[Theorem 4.1.18]{Jones(Planar-algebras-I)}
or \cite[Corollaries~2.35 and~2.39]{Jones-Penneys(The-embedding-theorem)},
which states that the $n$th power of the operation
\begin{equation}\label{eq: Jones rho}
\def\colA{black!10}
\def\colB{black!30}
\raisebox{2pt}{$\rho_n\,\,:\,\,\,\,$}\tikzmath[scale=\squarescale]
	{\def\h{10}\def\X{19}
	\useasboundingbox (-\X,-14) rectangle (\X,\h);
	\fill[\colA] (-\X,-\h) -- (-12,-\h) -- (-12,\h) [rounded corners=10]-- (-\X,\h) -- cycle;
	\fill[\colA]  \foreach \x/\y in {-8/-4,0/4,8/12} {(\x,-\h) rectangle (\y,\h)};
	\fill[\colB]  \foreach \x/\y in {-12/-8,-4/0,4/8} {(\x,-\h) rectangle (\y,\h)};
	\fill[\colB] (\X,-\h) -- (12,-\h) -- (12,\h) [rounded corners=10]-- (\X,\h) -- cycle;
	\draw \foreach \x in {-12,-8,-4,0,4,8,12}{(\x,\h) -- (\x,-\h)};
	\draw 	(0,0) node[fill=white, draw, inner xsep=25, inner ysep=3]{$x$};
	\node[rotate = 90] at (0,-11) {$\scriptstyle\left\{\phantom{\begin{matrix}.\\.\\.\\\Big|^i\end{matrix}}\right.$};
	\node at (0,-16) {$\scriptstyle n$};
	}
	\quad\raisebox{2pt}{$\mapsto$}\quad
\tikzmath[scale=\squarescale]
	{\def\h{11}\def\X{19}
	\useasboundingbox (-\X,-14) rectangle (\X,\h);
	\clip[rounded corners=10] (-\X,-\h) rectangle (\X,\h);
	\fill[\colA] (-\X,-\h) rectangle (\X,\h);
	\filldraw[fill=\colB, rounded corners=4] (-12,-14) -- (-12,-9) -- (-4,-5) -- (-4,5) -- (4,9) -- (4,14) -- (8,14) -- (8,9) -- (0,5) -- (0,-5) -- (-8,-9) -- (-8,-14) -- cycle;
	\filldraw[fill=\colB, rounded corners=4] (-4,-14) -- (-4,-9) -- (4,-5) -- (4,5) -- (12,9) -- (12,14) -- (21,14) -- (21,-14) -- (4,-14) -- (4,-9) -- (12,-5) [rounded corners = 3]-- (12,4.5) -- (15,4.5) -- (15,-5)
	[rounded corners = 4] -- (8,-9) -- (8,-14) -- (12,-14) -- (12,-9) -- (17,-5.5) [rounded corners = 7.5]-- (17,7) [rounded corners = 4]-- (11,6.7) [rounded corners = 3]-- (8,4) [rounded corners = 4]-- (8,-5) -- (0,-9) -- (0,-14) -- cycle;
	\filldraw[fill=\colB, rounded corners=4]  (-4,14) -- (-4,9) -- (-12,5) [rounded corners = 3]-- (-12,-4.5) -- (-15,-4.5) -- (-15,5)
	[rounded corners = 4] -- (-8,9) -- (-8,14) -- (-12,14) -- (-12,9) -- (-17,5.5) [rounded corners = 7.5]-- (-17,-7) [rounded corners = 4]-- (-11,-6.7) [rounded corners = 3]-- (-8,-4) [rounded corners = 4]-- (-8,5) -- (-0,9) -- (-0,14) -- cycle;
	\draw 	(0,0) node[fill=white, draw, inner xsep=24, inner ysep=3]{$x$};
	}
\end{equation}
is the identity.
One should note that Jones' rotation $\rho_n$ does not always agree with our way of interpreting figure \eqref{eq: Jones rho}.
It agrees when the type $\mathit{II}_1$ subfactor $A\subset B$ is extremal, that is, when the normalized traces on 
$A'$ and $B$ coincide on $A'\cap B$ or, equivalently, when the minimal conditional expectation $B\to A$ is equal to the trace preserving one.  See also Warning \ref{warn: Jones index}.
\end{remark}

\begin{lemma}\label{lem: direct summand of dualizable bimodule}
\def\colA{black!10}
\def\colB{black!30}
Let ${}_AH_B$ be a dualizable bimodule with dual ${}_B\bar H_A$, and let $p\in \mathrm{End}({}_AH_B)$ be a projection.
The $A$-$B$-bimodule $pH$ is then dualizable and its dual is given by $\bar p \bar H$, where
\begin{equation}\label{eq: pibar}
\bar p\,:=\,\,\tikzmath[scale=\squarescale]
	{\clip[rounded corners=10] (-4,-7) +(-10,-9) rectangle +(10,9);
	\fill[rounded corners=10, fill=\colB] (-15,-22) rectangle (11,16);
	\draw[rounded corners=6, fill=\colA]  (-4,-7) 
	+(-6,20) -- +(-6,-6) -- +(0,-6) -- +(0,6) -- +(6,6) -- +(6,-20) -- (7,-25) -- (7,13) -- cycle;
	\draw 	(-4,-7) node[fill=white, draw]{$p$};
	}\,\,=\,\,
\tikzmath[scale=\squarescale]
	{\clip[rounded corners=10] (-4,-7) +(-10,-9) rectangle +(10,9);
	\fill[rounded corners=10, fill=\colB] (-15,-22) rectangle (11,16);
	\draw[rounded corners=6, fill=\colA]  (-4,-7) 
	+(6,20) -- +(6,-6) -- +(0,-6) -- +(0,6) -- +(-6,6) -- +(-6,-20) -- (7,-25) -- (7,13) -- cycle;
	\draw 	(-4,-7) node[fill=white, draw]{$p$};
	}\,\,\in\, \mathrm{End}({}_B\bar H_A)\,.
\end{equation}
Moreover, its statistical dimension (see Definition \ref{def: statistical dimension}) is given by
$
{\dim(pH)\,=\,
\tikzmath[scale=\squarescale]
	{\fill[rounded corners=5, fill=\colA] (-9,-7) rectangle (4,7);
	\draw[rounded corners=6, fill=\colB] (-5,-5) rectangle (1,5);
	\draw 	(-5,0) node[fill=white, draw, inner ysep=2, inner xsep=2]{$p$};
	}
}$\,.
\end{lemma}

\begin{proof}
\def\colA{black!10}
\def\colB{black!30}
Let $\pi:H\twoheadrightarrow pH$, $\bar\pi:\bar H\twoheadrightarrow \bar p\bar H$ be the orthogonal projections, so that $p=\pi^*\pi$ and $\bar p=\bar \pi^*\bar \pi$.
The maps 
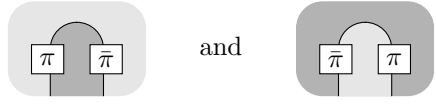
\[
\tikzmath[scale=\squarescale]
	{\clip[rounded corners=10] (-13,-6) rectangle (13,12); 
	\fill[rounded corners=10, fill=\colA] (-13,-12) rectangle (13,12);
	\draw[rounded corners=10, fill=\colB] (-5,-12) rectangle (5,8);
	\draw 	(-5.5,1) node[fill=white, draw, inner xsep=3]{$\pi$};
	\draw 	(5.5,1) node[fill=white, draw, inner ysep=2.65, inner xsep=3]{$\bar \pi$};
	}
\qquad\text{and}\qquad
\tikzmath[scale=\squarescale]
	{\clip[rounded corners=10] (-13,-6) rectangle (13,12); 
	\fill[rounded corners=10, fill=\colB] (-13,-12) rectangle (13,12);
	\draw[rounded corners=10, fill=\colA] (-5,-12) rectangle (5,8);
	\draw 	(-5.5,1) node[fill=white, draw, inner ysep=2.65, inner xsep=3]{$\bar \pi$};
	\draw 	(5.5,1) node[fill=white, draw, inner xsep=3]{$\pi$};
	}
\]
exhibit $\bar p\bar H$ as dual to $pH$.
The statistical dimension is therefore given by
\[
\dim\big(pH\big)=\,\tikzmath[scale=\squarescale]
	{\fill[rounded corners=10, fill=\colA] (-12,-12) rectangle (12,12);
	\draw[rounded corners=10, fill=\colB] (-5,-10.5) rectangle (5,10.5);
	\draw 	(-5.5,3.8) node[fill=white, draw, inner xsep=3]{$\pi$};
	\draw 	(5.5,3.8) node[fill=white, draw, inner ysep=2.65, inner xsep=3]{$\bar \pi$};
	\draw 	(-5.5,-3.8) node[fill=white, draw, inner ysep=2, inner xsep=1.3]{$\pi^*$};
	\draw 	(5.5,-3.8) node[fill=white, draw, inner ysep=2, inner xsep=1.3]{$\bar \pi^*$};
	}
\,=\,\tikzmath[scale=\squarescale]
	{\fill[rounded corners=10, fill=\colA] (-12,-12) rectangle (12,12);
	\draw[rounded corners=10, fill=\colB] (-5,-7.1) rectangle (5,7.1);
	\draw 	(-5.5,0) node[fill=white, draw]{$p$};
	\draw 	(5.5,0) node[fill=white, draw, inner ysep=2.65]{$\bar p$};
	}
\,=\,\tikzmath[scale=\squarescale]
	{\fill[rounded corners=10, fill=\colA] (-12,-12) rectangle (17,12);
	\draw[rounded corners=8, fill=\colB] (-5,-8) -- (-5,8) -- (3,8) [rounded corners=5]-- (3,-5) -- (8,-5) -- (8,5) -- (13,5) [rounded corners=10]-- (13,-8) -- cycle;
	\draw 	(-5.5,0) node[fill=white, draw, inner xsep=3]{$p$};
	\draw 	(8,0) node[fill=white, draw, inner ysep=2.5, inner xsep=2.5]{$p$};
	}
\,=\,\tikzmath[scale=\squarescale]
	{\fill[rounded corners=10, fill=\colA] (-12,-12) rectangle (11,12);
	\draw[rounded corners=10, fill=\colB] (-5,-7) rectangle (5,7);
	\draw 	(-5,0) node[fill=white, draw]{$p$};
	}\,. \qedhere 
\] 
\end{proof}

\begin{theorem}\label{thm: well defined up to unique unitary iso}
Let ${}_AH_B$ be a dualizable bimodule.
Then its dual $({}_B\bar H_A, R, S)$ is well defined up to unique unitary isomorphism.
\end{theorem}

\begin{proof}
\def\colA{black!10}
\def\colB{black!30}
Let ${}_B\bar H_A$ and ${}_B\bar H'_A$ be two bimodules that are dual to ${}_AH_B$, with respective structure maps $R$, $S$, $R'$, $S'$:
\[
R\,=\,\tikzmath[scale=.1]{\fill[fill=\colA] (-4,-2.5) [rounded corners=4]-- (-4,3) --
(4,3) -- (4,-2.5) [sharp corners]-- (2,-2.5) [rounded corners = 5.5]-- (2,.5) -- (-2,.5) [sharp corners]-- (-2,-2.5) [rounded corners=4]-- cycle;
\fill[fill=\colB] (-2,-2.5) -- (2,-2.5) [rounded corners = 5.5]-- (2,.5) -- (-2,.5) [sharp corners]-- cycle;
\draw[rounded corners = 5.5] (0,.5) -- (-2,.5) -- (-2,-2.5);\draw[rounded corners = 5.5, very thick] (0,.5) -- (2,.5) -- (2,-2.5);}\,,
\quad
S\,=\,\tikzmath[scale=.1]{\fill[fill=\colB] (-4,-2.5) [rounded corners=4]-- (-4,3) --
(4,3) -- (4,-2.5) [sharp corners]-- (2,-2.5) [rounded corners = 5.5]-- (2,.5) -- (-2,.5) [sharp corners]-- (-2,-2.5) [rounded corners=4]-- cycle;
\fill[fill=\colA] (-2,-2.5) -- (2,-2.5) [rounded corners = 5.5]-- (2,.5) -- (-2,.5) [sharp corners]-- cycle;
\draw[rounded corners = 5.5] (0,.5) -- (2,.5) -- (2,-2.5);\draw[rounded corners = 5.5, very thick] (0,.5) -- (-2,.5) -- (-2,-2.5);}\,,
\quad
R'\,=\,\tikzmath[scale=.1]{\fill[fill=\colA] (-4,-2.5) [rounded corners=4]-- (-4,3) --
(4,3) -- (4,-2.5) [sharp corners]-- (2,-2.5) [rounded corners = 5.5]-- (2,.5) -- (-2,.5) [sharp corners]-- (-2,-2.5) [rounded corners=4]-- cycle;
\fill[fill=\colB] (-2,-2.5) -- (2,-2.5) [rounded corners = 5.5]-- (2,.5) -- (-2,.5) [sharp corners]-- cycle;
\draw[rounded corners = 5.5] (0,.5) -- (-2,.5) -- (-2,-2.5);\draw[rounded corners = 5.5, very thick, densely dotted] (0,.5) -- (2,.5) -- (2,-2.5);}\,,
\quad
S'\,=\,\tikzmath[scale=.1]{\fill[fill=\colB] (-4,-2.5) [rounded corners=4]-- (-4,3) --
(4,3) -- (4,-2.5) [sharp corners]-- (2,-2.5) [rounded corners = 5.5]-- (2,.5) -- (-2,.5) [sharp corners]-- (-2,-2.5) [rounded corners=4]-- cycle;
\fill[fill=\colA] (-2,-2.5) -- (2,-2.5) [rounded corners = 5.5]-- (2,.5) -- (-2,.5) [sharp corners]-- cycle;
\draw[rounded corners = 5.5] (0,.5) -- (2,.5) -- (2,-2.5);\draw[rounded corners = 5.5, very thick, densely dotted] (0,.5) -- (-2,.5) -- (-2,-2.5);}\,.
\]
Here, thick lines represent $\bar H$, and thick dotted lines represent $\bar H'$.
The isomorphism between $\bar H$ and $\bar H'$ is given by
$
v:=(S^*\otimes1)(1\otimes R')=\tikzmath[scale=.1]{
\fill[fill=\colA, rounded corners=4]
(6,2.75) -- (6,-2.75) [sharp corners]-- (3,-2.75) [rounded corners = 4]-- (3,1.5) -- (0,1.5) -- (0,-1.5) -- (-3,-1.5) [sharp corners]-- (-3,2.75) [rounded corners=4]-- cycle;
\fill[fill=\colB, rounded corners=4]
(-6,-2.75) -- (-6,2.75) [sharp corners]-- (-3,2.75) [rounded corners = 4]-- (-3,-1.5) -- (0,-1.5) -- (0,1.5) -- (3,1.5) [sharp corners]-- (3,-2.75) [rounded corners=4]-- cycle;
\draw[rounded corners = 4, very thick] (-1.5,-1.5) -- (-3,-1.5) -- (-3,2.75);
\draw[rounded corners = 4] (-1.5,-1.5)  -- (0,-1.5) -- (0,1.5) -- (1.5,1.5);
\draw[rounded corners = 4, very thick, densely dotted] (1.5,1.5) -- (3,1.5) -- (3,-2.75);}
$\,.
This isomorphism is certainly the unique isomorphism intertwining $R$ and $R'$, and $S$ and $S'$.  Our goal is to show that $v$ is unitary.
In other words, we need to show that $v$ is equal to $v^*{}^{-1}$; note that $\tikzmath[scale=.1]{
\fill[fill=\colA, rounded corners=4]
(6,-2.75) -- (6,2.75) [sharp corners]-- (3,2.75) [rounded corners = 4]-- (3,-1.5) -- (0,-1.5) -- (0,1.5) -- (-3,1.5) [sharp corners]-- (-3,-2.75) [rounded corners=4]-- cycle;
\fill[fill=\colB, rounded corners=4]
(-6,2.75) -- (-6,-2.75) [sharp corners]-- (-3,-2.75) [rounded corners = 4]-- (-3,1.5) -- (0,1.5) -- (0,-1.5) -- (3,-1.5) [sharp corners]-- (3,2.75) [rounded corners=4]-- cycle;
\draw[rounded corners = 4, very thick, densely dotted] (-1.5,1.5) -- (-3,1.5) -- (-3,-2.75);
\draw[rounded corners = 4] (-1.5,1.5)  -- (0,1.5) -- (0,-1.5) -- (1.5,-1.5);
\draw[rounded corners = 4, very thick] (1.5,-1.5) -- (3,-1.5) -- (3,2.75);}$
is the inverse of $v^*$.
We can rewrite $R'$ and $S'$ as
\[
R'=\tikzmath[scale=\squarescale]{\clip[rounded corners=10] (-12,-6) rectangle (13,12); 
	\fill[rounded corners=10, fill=\colA] (-12,-12) rectangle (13,12);\fill[rounded corners=10, fill=\colB] (-5,-12) rectangle (5,8);
	\draw[rounded corners=10] (-5,-6) -- (-5,8) -- (0,8);\draw[rounded corners=10, very thick] (5,1) -- (5,8) -- (0,8);
	\draw[rounded corners=10, very thick, densely dotted] (5,-6) -- (5,1);\draw (5,1) node[fill=white, draw]{$v$};}\qquad
S'=\tikzmath[scale=\squarescale]{\clip[rounded corners=10] (-13,-6) rectangle (12,12); 
	\fill[rounded corners=10, fill=\colB] (-13,-12) rectangle (12,12);\fill[rounded corners=10, fill=\colA] (-5,-12) rectangle (5,8);
	\draw[rounded corners=10] (5,-6) -- (5,8) -- (0,8);\draw[rounded corners=10, very thick] (-5,1) -- (-5,8) -- (0,8);
	\draw[rounded corners=10, very thick, densely dotted] (-5,-6) -- (-5,1);\draw (-5,1) node[fill=white, draw, inner sep=2]{$v^*{}^{-1}$};}\,.
\]
Equation \eqref{eq:duality normalization} for $R'$ and $S'$ then reads
\begin{equation}\label{eq: normalization for R' and S'}
\tikzmath[scale=\squarescale]
	{\fill[rounded corners=10, fill=\colA] (-15,-17) rectangle (15,17);\fill[rounded corners=11, fill=\colB] (-6,-14) rectangle (6,14);
	\draw[rounded corners=11] (0,-14) -- (-6,-14) -- (-6,14) -- (0,14);\draw[rounded corners=11, very thick] (6,5) -- (6,14) -- (0,14);
	\draw[rounded corners=11, very thick] (6,-5) -- (6,-14) -- (0,-14);\draw[rounded corners=11, very thick, densely dotted] (6,-5) -- (6,5);
	\draw (6,7) node[fill=white, draw, inner xsep=4]{$v$};\draw (6,-7) node[fill=white, draw, inner sep=2]{$v^*$};
	\draw (-6,0) node[fill=white, draw]{$x$};\node at (-11.5,0) {$p$};\node at (-.5,0) {$q$};}
\,\,\,=\,\,\,\tikzmath[scale=\squarescale]
	{\fill[rounded corners=10, fill=\colB] (-15,-17) rectangle (15,17);\fill[rounded corners=11, fill=\colA] (-6,-14) rectangle (6,14);
	\draw[rounded corners=11] (0,-14) -- (6,-14) -- (6,14) -- (0,14);\draw[rounded corners=11, very thick] (-6,5) -- (-6,14) -- (0,14);
	\draw[rounded corners=11, very thick] (-6,-5) -- (-6,-14) -- (0,-14);\draw[rounded corners=11, very thick, densely dotted] (-6,-5) -- (-6,5);
	\draw (-6,7) node[fill=white, draw, inner xsep=0, inner ysep=2]{$\,v^*{}^{-1}$};\draw (-6,-7) node[fill=white, draw, inner sep=2]{$v^{-1}$};
	\draw (6,0) node[fill=white, draw]{$x$};\node at (.5,0) {$p$};\node at (11.5,0) {$q$};}\,.
\end{equation}

Given minimal central projections $p\in A$ and $q\in B$, the map
\begin{equation*}\label{eq: tr_pq}
tr_{pq}\,\,\,:\,\,\,y\,\,\mapsto\,\,\tikzmath[scale=0.075]{\fill[rounded corners=10, fill=\colB] (-14,-12) rectangle (11,12);
\draw[rounded corners=10, fill=\colA] (-5,-7.5) rectangle (5,7.5);\draw[rounded corners=10, very thick] (0,-7.5) -- (-5,-7.5) -- (-5,7.5) -- (0,7.5);
\draw (-5,0) node[fill=white, draw]{$y$};\node at (-10.5,0) {$q$};\node at (.5,0) {$p$};
}
\end{equation*}
is a trace on $\mathrm{End}({}_B\bar H_A)$, as can be seen by applying Lemma \ref{lem: phi is a trace} to the bimodule ${}_{qB}(q\bar Hp)_{pA}$.
Applying Lemma \ref{lem: Lemma B} to each summand in the decomposition
$\bar H=\bigoplus_{pq}q\bar Hp$, and using the fact that $\mathrm{End}(\bar H)=\bigoplus_{pq}\mathrm{End}(q\bar Hp)$, it follows that the traces $tr_{pq}$ are jointly faithful.
That is, given a positive element $y$, there exists at least one $tr_{pq}$ such that $tr_{pq}(y)\not=0$.
Letting $\bar x$ be as in \eqref{eq: left half twirl = right half twirl}, equation \eqref{eq: normalization for R' and S'} implies
\[
tr_{pq}(v^*v\,\bar x)=tr_{pq}(v^{-1}v^*{}^{-1}\bar x)\qquad\forall  x\in \mathrm{End}({}_AH_B).
\]
This being true for all $p$, $q$, it follows that $v^*v=v^{-1}v^*{}^{-1}$\!. In other words, $v^*v=(v^*v)^{-1}$.
Since $v^*v$ is positive, we must have $v^*v=1$.
\end{proof}

\section{Statistical dimension and index}\label{sec: statistical dimension and minimal index}

The following definition is well known. Our approach follows \cite{Longo-Roberts(A-theory-of-dimension)}.

\begin{definition}\label{def: statistical dimension}
If $A$ and $B$ are factors, then the \emph{statistical dimension} of a dualizable bimodule ${}_AH_B$ is given by :
\begin{equation*}
\qquad\dim ({}_A H_B) := R^*R = S^* S\,\in\,\IR_{\ge 0}.
\end{equation*}
For non-dualizable bimodules, one simply declares $\dim({}_AH_B)$ to be $\infty$.
\end{definition}

The basic properties of the statistical dimension can be found in many places 
\cite{Kosaki(Type-III-factors-and-index-theory), Kosaki-Longo(A-remark-on-the-minimal-index),
Longo(Index-of-subfactors-and-statistics-of-quantum-fields-I), Longo-Roberts(A-theory-of-dimension)}.
We include some proofs for completeness.

\begin{proposition}\label{prop:basic properties of index} 
The statistical dimension of a non-zero bimodule ${}_AH_B$ is always $\ge\!1$, and is equal to $1$ if and only if $H$ is invertible.
The statistical dimension is additive under direct sums, and multiplicative under Connes fusion\footnote{For this to always be true, it is appropriate to use the convention $0\hspace{-.05cm}\cdot\hspace{-.05cm} \infty = 0$.}.
It is also multiplicative under external tensor product.
In other words, we have:
\begin{alignat}{1}
\dim({}_AH_B) \in \{0\}\cup[1,\infty&], \text{ \rm and it is $0$ iff } H=0\label{eq:properties of index 1}\\
\dim({}_AH_B) =&\, 1 \text{ \rm iff } A'=B\label{eq:properties of index 2}\\
\dim({}_A(H\oplus K)_B) =&\, \dim({}_AH_B)+ \dim({}_AK_B)\label{eq:properties of index 3}\\
\dim({}_AH\boxtimes_B K_C) =&\, \dim({}_AH_B) \dim({}_BK_C)\label{eq:properties of index 4}\\
\dim(({}_AH_B)\otimes_\IC ({}_CK_D)) =&\, \dim({}_AH_B) \dim({}_CK_D)\label{eq:properties of index 5}
\end{alignat}
\end{proposition}

\begin{proof} 
\def\colA{black!10}
\def\colB{black!30}
{\it i.}
If $H\not = 0$, then $\dim({}_AH_B)\ge 1$ by Lemma \ref{lem: D >= 1}.
If $H=0$, then clearly $R^*R=0$.

{\it ii.}
Let $e_1$, $e_2$ be as in Lemma \ref{lem: D >= 1}.
If $\dim({}_AH_B)=1$, then $e_1 = e_1 e_2 e_1$ and $e_2 = e_2 e_1 e_2$.
As $e_1$ and $e_2$ are projections, the first equation implies $e_2 \geq e_1$,
while the second implies $e_1 \geq e_2$. 
Thus $e_1 = e_2$.
From this (and a reflection along a vertical axis of the argument so far), we get
$
\tikzmath[scale=\squarescale]{
\clip[rounded corners=2] (-4,-3) rectangle (10,3);
\fill[fill=\colA] 
(-4,-3) [rounded corners=2]-- (-4,3) [sharp corners]-- (-2,3) [rounded corners = 4]-- (-2,1) -- (2,1) [sharp corners]-- (2,3) --
(5,3) -- (5,-3) -- (2,-3) [rounded corners = 4]-- (2,-1) -- (-2,-1) [sharp corners]-- (-2,-3) [rounded corners=2]-- cycle;
\fill[fill=\colB] (-2,-3) -- (2,-3) [rounded corners = 4]-- (2,-1) -- (-2,-1) [sharp corners]-- cycle;
\fill[fill=\colB] (-2,3) -- (2,3) [rounded corners = 4]-- (2,1) -- (-2,1) [sharp corners]-- cycle;
\draw (2,3) [rounded corners = 4]-- (2,1) -- (-2,1) -- (-2,3);
\draw (2,-3) [rounded corners = 4]-- (2,-1) -- (-2,-1) -- (-2,-3);
\fill[fill=\colB] (8,-3) -- (8,3) -- (5,3) -- (5,-3) -- cycle;
\draw (5,3) -- (5,-3);
\fill[fill=\colA] (10,-3) [rounded corners=2]-- (10,3) [sharp corners]-- (8,3) -- (8,-3) [rounded corners=2]-- cycle;
\draw (8,3) -- (8,-3);
}\,=\,
\tikzmath[scale=\squarescale]{
\clip[rounded corners=2] (-7,-3) rectangle (7,3);
\fill[fill=\colB] 
(-5,-3) -- (-5,3) -- (-2,3) [rounded corners = 4]-- (-2,1) -- (2,1) [sharp corners]-- (2,3) --
(5,3) -- (5,-3) -- (2,-3) [rounded corners = 4]-- (2,-1) -- (-2,-1) [sharp corners]-- (-2,-3) -- cycle;
\fill[fill=\colA] (-2,-3) -- (2,-3) [rounded corners = 4]-- (2,-1) -- (-2,-1) [sharp corners]-- cycle;
\fill[fill=\colA] (-2,3) -- (2,3) [rounded corners = 4]-- (2,1) -- (-2,1) [sharp corners]-- cycle;
\draw (2,3) [rounded corners = 4]-- (2,1) -- (-2,1) -- (-2,3);
\draw (2,-3) [rounded corners = 4]-- (2,-1) -- (-2,-1) -- (-2,-3);
\fill[fill=\colA] (7,-3) [rounded corners=2]-- (7,3) [sharp corners]-- (5,3) -- (5,-3) [rounded corners=2]-- cycle;
\draw (5,3) -- (5,-3);
\fill[fill=\colA] (-7,-3) [rounded corners=2]-- (-7,3) [sharp corners]-- (-5,3) -- (-5,-3) [rounded corners=2]-- cycle;
\draw (-5,3) -- (-5,-3);
}\,=\,
\tikzmath[scale=\squarescale]{
\clip[rounded corners=2] (-10,-3) rectangle (4,3);
\fill[fill=\colA] 
(4,-3) [rounded corners=2]-- (4,3) [sharp corners]-- (2,3) [rounded corners = 4]-- (2,1) -- (-2,1) [sharp corners]-- (-2,3) --
(-5,3) -- (-5,-3) -- (-2,-3) [rounded corners = 4]-- (-2,-1) -- (2,-1) [sharp corners]-- (2,-3) [rounded corners=2]-- cycle;
\fill[fill=\colB] (2,-3) -- (-2,-3) [rounded corners = 4]-- (-2,-1) -- (2,-1) [sharp corners]-- cycle;
\fill[fill=\colB] (2,3) -- (-2,3) [rounded corners = 4]-- (-2,1) -- (2,1) [sharp corners]-- cycle;
\draw (-2,3) [rounded corners = 4]-- (-2,1) -- (2,1) -- (2,3);
\draw (-2,-3) [rounded corners = 4]-- (-2,-1) -- (2,-1) -- (2,-3);
\fill[fill=\colB] (-8,-3) -- (-8,3) -- (-5,3) -- (-5,-3) -- cycle;
\draw (-5,3) -- (-5,-3);
\fill[fill=\colA] (-10,-3) [rounded corners=2]-- (-10,3) [sharp corners]-- (-8,3) -- (-8,-3) [rounded corners=2]-- cycle;
\draw (-8,3) -- (-8,-3);
}
$.
As $A$ is a factor and ${}_AH_B \neq 0$, the latter is a faithful $A$-module.
Lemma~\ref{lem: D >= 1} implies that the projection
$R R^* = 
\def\colA{black!10}
\def\colB{black!30}
\tikzmath[scale=\squarescale]{
\clip[rounded corners=2] (-4,-3) rectangle (4,3);
\fill[fill=\colA] 
(-3.9,-3) [rounded corners=2]-- (-3.9,3) [sharp corners]-- (-2,3) [rounded corners = 4]-- (-2,1) -- (2,1) [sharp corners]-- (2,3) [rounded corners=2]--
(3.9,3) -- (3.9,-3) [sharp corners]-- (2,-3) [rounded corners = 4]-- (2,-1) -- (-2,-1) [sharp corners]-- (-2,-3) [rounded corners=2]-- cycle;
\fill[fill=\colB] (-2,-3) -- (2,-3) [rounded corners = 4]-- (2,-1) -- (-2,-1) [sharp corners]-- cycle;
\fill[fill=\colB] (-2,3) -- (2,3) [rounded corners = 4]-- (2,1) -- (-2,1) [sharp corners]-- cycle;
\draw (2,3) [rounded corners = 4]-- (2,1) -- (-2,1) -- (-2,3);
\draw (2,-3) [rounded corners = 4]-- (2,-1) -- (-2,-1) -- (-2,-3);
}$
is non-trivial.
Thus, the previous equation implies 
$
\def\colA{black!10}
\def\colB{black!30}
\tikzmath[scale=\squarescale]{
\fill[fill=\colA] 
(-3.9,-3) [rounded corners=2]-- (-3.9,3) [sharp corners]-- (-2,3) [rounded corners = 4]-- (-2,1) -- (2,1) [sharp corners]-- (2,3) [rounded corners=2]--
(3.9,3) -- (3.9,-3) [sharp corners]-- (2,-3) [rounded corners = 4]-- (2,-1) -- (-2,-1) [sharp corners]-- (-2,-3) [rounded corners=2]-- cycle;
\fill[fill=\colB] (-2,-3) -- (2,-3) [rounded corners = 4]-- (2,-1) -- (-2,-1) [sharp corners]-- cycle;
\fill[fill=\colB] (-2,3) -- (2,3) [rounded corners = 4]-- (2,1) -- (-2,1) [sharp corners]-- cycle;
\draw (2,3) [rounded corners = 4]-- (2,1) -- (-2,1) -- (-2,3);
\draw (2,-3) [rounded corners = 4]-- (2,-1) -- (-2,-1) -- (-2,-3);
}\,=\,
\tikzmath[scale=\squarescale]{
\fill[fill=\colA] (-3.9,-3) [rounded corners=2]-- (-3.9,3) [sharp corners]-- (-1.7,3) -- (-1.7,-3) [rounded corners=2]-- cycle;
\fill[fill=\colA] (3.9,-3) [rounded corners=2]-- (3.9,3) [sharp corners]-- (1.7,3) -- (1.7,-3) [rounded corners=2]-- cycle;
\fill[fill=\colB] (-1.7,-3) rectangle (1.7,3);
\draw (1.7,3) -- (1.7,-3)(-1.7,3) -- (-1.7,-3);
}
$\,. 
The map $R$ is therefore invertible, and similarly for $S$.
Having shown ${}_B\bar H\boxtimes_A H_B\cong L^2B$ and 
${}_A H\boxtimes_B\bar H_A\cong L^2A$,
the result follows from 
Proposition \ref{prop: characterization of invertible bimodules}.

Conversely, if $H$ is invertible, there exist unitary maps 
$\tilde R \colon L^2(A) \to H \boxtimes_B \bar H$ and 
$S \colon L^2(B) \to \bar H \boxtimes_A H$.
Since ${}_{A}H_B$ is irreducible, 
$\lambda := (\tilde R^*\otimes 1)(1\otimes S)$ is a scalar, and so $R :=  \lambda \tilde R$ and $S$ satisfy~\eqref{eq: duality equations}.
Again because ${}_{A}H_B$ is irreducible (and $R$ and $S$ are unitary),
the normalization condition~\eqref{eq:duality normalization} is satisfied
as well.
Thus  $d= R^* R = 1$.

{\it iii.} If either $H$ or $K$ is not dualizable, then both sides of \eqref{eq:properties of index 3} are infinite by Lemma~\ref{lem: direct summand of dualizable bimodule}.
If they are both dualizable, then Lemma \ref{lem: sum of dualizables} provides a description of the duality maps for $H\oplus K$, which we can use to compute
\[
\dim(H\oplus K)=\!\tikzmath{
\node[scale = .7] at (0,0) {$\textstyle\begin{pmatrix}R \\ 0 \\ 0\\ \tilde R\end{pmatrix}$};
\node[scale = .7] at (.7,0) {$\textstyle\begin{pmatrix}R \\ 0 \\ 0\\ \tilde R\end{pmatrix}$};
\node[scale = .7] at (.3,.55) {$*$};
}\!
=R^*R+\tilde R^*\tilde R\,=\,\dim(H)+ \dim(K).
\]

{\it iv.} 
If both $H$ and $K$ are dualizable, then using the duality maps described in Lemma \ref{lem: product of dualizables},
we compute:
\[
\dim(H\boxtimes_B K)=R^*(1\otimes \tilde R^*\!\otimes 1)(1\otimes \tilde R\otimes 1)R=R^*\dim(K)R=\dim(H)\dim(K).
\]
If either $H$ or $K$ is zero, then the equation clearly holds.
The remaining case $H\not =0$, $\dim(K)=\infty$ requires different techniques\footnote{Note that the special case $\dim(H)=1$, $\dim(K)=\infty$ is straightforward, as fusing with an invertible bimodule certainly doesn't change the property of having a dual or not.\label{fn: tensor with invertible}
} and will be treated later, in Corollary~\ref{cor: HK finite ==> H finite}.

{\it v.} Apply equation \eqref{eq:properties of index 4} to the decomposition
\[
({}_AH_B)\otimes_\IC ({}_CK_D)\cong 
\big(({}_AH_B)\otimes_\IC ({}_CL^2C_C)\big)
\underset{B\bar\otimes C}\boxtimes
\big(({}_BL^2B_B)\otimes_\IC ({}_CK_D)\big). \qedhere
\]
\end{proof}

\begin{remark}
As was shown in the celebrated papers \cite{Jones(Index-for-subfactors), Kosaki(Extension-of-Jones-index-to-arbitrary-factors)},
equation \eqref{eq:properties of index 1} can be improved: the statistical dimension of a bimodule can only take values in the set
$\{2\cos(\frac \pi n)\,;\,n=2,3,4,\ldots\}\cup[2,\infty]$.
\end{remark}

If the von Neumann algebras $A=\bigoplus A_i$ and $B=\bigoplus B_j$ are finite direct sums of factors (in other words have finite-dimensional centers), then any $A$-$B$-bimodule $H$ can be written as a direct sum
\begin{equation}\label{eq: H = oplus H_ij}
H=\bigoplus H_{ij}
\end{equation}
of $A_i$-$B_j$-bimodules.
The statistical dimension of ${}_AH_B$ is then best taken to be a matrix of numbers
\cite{Teruya(Index-vN-algebra-fd-center)}:
\[
\dim({}_AH_B)_{ij} := \dim({}_{A_i}{H_{ij}}\,{}_{B_j}).
\]
The matrix-valued statistical dimension satisfies the same formal properties \eqref{eq:properties of index 1}--\eqref{eq:properties of index 5} as above,
provided the right hand sides of \eqref{eq:properties of index 4} and \eqref{eq:properties of index 5} are interpreted in terms of matrix and Kronecker products, respectively.

As will be shown later, in Corollary \ref{cor: justification of the name minimal},
the following definition of index is equivalent to other definitions that exist in the literature \cite{Hiai(Minimizing-indices), Kosaki(Extension-of-Jones-index-to-arbitrary-factors), Kosaki(Type-III-factors-and-index-theory), Longo(Index-of-subfactors-and-statistics-of-quantum-fields-I), Pimsner-Popa}:

\begin{definition}\label{def: min index}
The \emph{index} $[B:A]$ of an inclusion of factors ${\iota:A \to B}$ is the square of the statistical dimension of ${}_AL^2B_B$.
\end{definition}

\begin{warning}\label{warn: Jones index}
The above definition does not always agree with Jones' original definition \cite{Jones(Index-for-subfactors)}.
It agrees if and only if the type $\mathit{II}_1$ subfactor $A \subset B$ is extremal, that is, the normalized traces on 
$A'$ and $B$ coincide on $A'\cap B$.
\end{warning}

Let $\iota:A\to B$ be a subfactor.
If the  index $[B:A]$ is finite, we say that $\iota$ is a finite homomorphism.
More generally, if $A$ and $B$ are von Neumann algebras with finite-dimensional centers, we say that a homomorphism $A \to B$ is finite if all the matrix entries of $\dim({}_AL^2B_B)$ are finite.  Of course, this simply amounts to the following definition:

\begin{definition}\label{def: finite homomrphism}
A homomorphism $A \to B$ between von Neumann algebras with finite-dimensional centers is \emph{finite} if the associated bimodule ${}_AL^2B_B$ is dualizable.
\end{definition}

When dealing with inclusions of von Neumann algebras with finite-dimensional center, the matrix $\dim({}_AL^2B_B)$ is much better behaved than the 
corresponding matrix of indices.
We propose a new notation for it:

\begin{definition}\label{def: [[ B : A ]]}
Given a finite homomorphism $f:A\to B$ between von Neumann algebras with finite-dimensional center,
we let $\llbracket B : A\rrbracket:=\dim({}_AL^2B_B)$ denote the matrix of statistical dimensions of ${}_AL^2B_B$.
\end{definition}

\noindent Following \eqref{eq:properties of index 4}, the matrix of statistical dimensions satisfies
\begin{equation}\label{eq:properties of matrix of stat dim}
\llbracket B : A\rrbracket\llbracket C : B\rrbracket = \llbracket C : A\rrbracket.
\end{equation}

As an corollary of Lemma \ref{lem: Lemma B}, we have:

\begin{lemma}\label{lem: rel comm is finite dim}
Let $f:A\to B$ be a finite homomorphism between von Neumann algebras with finite-dimensional center.
Then the relative commutant of $f(A)$ in $B$ is finite-dimensional.
\end{lemma}

\begin{proof}
The relative commutant of $f(A)$ in $B$ is the endomorphism algebra of the bimodule ${}_AL^2(B)_B$.
Apply Lemma \ref{lem: Lemma B} to every summand in the decomposition \eqref{eq: H = oplus H_ij} of that bimodule.
\end{proof}

\begin{lemma}\label{lem: ind = dim}
Let ${}_AH_B$ be a bimodule between von Neumann algebras with finite-dimensional center.
Assume $B$ acts faithfully, and let $B'\supset A$ 
be the commutant of $B$ on $H$.
Then $\dim({}_AH_B)=\llbracket B':A\rrbracket$.
\end{lemma}

\begin{proof}
The bimodule ${}_{B'}H_B$ is a Morita equivalence, and its matrix of
statistical dimensions is therefore an identity matrix.
We have
\[
\dim({}_AH_B)=\dim({}_AL^2B'\boxtimes_{B'}H_B)=\dim({}_AL^2B'_{B'})\dim({}_{B'}H_B)=\dim({}_AL^2B'_{B'}).
\]
The last expression is the definition of the matrix $\llbracket B':A\rrbracket$.
\end{proof}

\begin{corollary}\label{cor: [B:A]=[A':B']}
If $A\subset B\subset \bfB(H)$ are von Neumann algebras with finite-dimensional centers, then $\llbracket B:A\rrbracket=\llbracket A':B'\rrbracket^T$.
In particular, if $A$ and $B$ are factors, then $[B:A]=[A':B']$.
\end{corollary}

\begin{proof}
Let $\overline{H}$ denote the complex conjugate of $H$, with
actions as in Proposition \ref{prop: characterization of invertible bimodules}.  Applying Lemma \ref{lem: ind = dim} twice, we have
$\llbracket B \!:\! A\rrbracket=\dim({}_AH_{B'})\!=\dim({}_{B'}\overline{H}_A)^T\!=\llbracket A' \!: B'\rrbracket^T$.
\end{proof}

\begin{lemma}\label{lem: non-factor in factor}
Let $B$ be a factor, and let $A\subset B$ be a subalgebra with finite-dimensional center. Call its minimal central projections $p_1, \ldots, p_n$.
Then $\sum [p_iBp_i:p_iA] = \| \llbracket B:A \rrbracket \|^2$, where $\|\,\,\|$ stands for the $\ell^2$-norm of a vector.
\end{lemma}
\begin{proof}
The $i$th entry in the vector $\llbracket B:A \rrbracket=\dim({}_AL^2B_B)$ is by definition
\[
\dim({}_{p_iA}(p_iL^2B)_B)=\dim({}_{p_iA}(p_iL^2Bp_i)_{p_iBp_i})=\dim({}_{p_iA}L^2(p_iBp_i)_{p_iBp_i}),
\]
where the first equality holds because ${}_B(L^2Bp_i)_{p_iBp_i}$ is an invertible bimodule, and the second one follows from Lemma \ref{lem: L^2(pAp) = pL^2(A)p}.
Therefore, $\dim({}_{p_iA}(p_iL^2B)_B)^2=[p_iBp_i:p_iA]$. The results now follows by summing over all the indices $i$.
\end{proof}

For more results about statistical dimension and index, we refer the reader to
\cite{Kawakami-Watatani, Kosaki(Extension-of-Jones-index-to-arbitrary-factors), 
Kosaki(Type-III-factors-and-index-theory), Kosaki-Longo(A-remark-on-the-minimal-index), Longo(Minimal-index-and-braided-subfactors), Longo-Rehren(Nets-of-subfactors)}.

\section{Functoriality of the $L^2$-space and of Connes fusion}\label{sec: functoriality of L2 and of Connes fusion}

\subsection*{The inner product on \emph{L}${}^\mathbf 2$(\hspace{-.15mm}\emph{A}\hspace{-.15mm})}
We mentioned earlier that for a von Neumann algebra $A$, its $L^2$-space is 
a completion of the vector space $\bigoplus_{\phi\in L^1_+(A)} \IC\textstyle\sqrt{\phi}$ with respect to some pre-inner product.
To define $\langle\sqrt\phi,\sqrt\psi\,\rangle$, 
one considers the function $f(t):=\phi([D\phi:D\psi]_t)$, where 
$[D\phi:D\psi]_t \in A$ denotes Connes' Radon-Nikodym 
derivative\footnote{We work with a definition of the Radon-Nikodym derivative $[D\phi:D\psi]_t$ that does not require $\phi$ and $\psi$ to be faithful; it satisfies  $[D\phi:D\psi]_t\in \bfs_\phi A\, \bfs_\psi$ where
  $\bfs_\phi$ and $\bfs_\psi$ are the support projections
  of $\phi$ and $\psi$.}.
The function $f$ can be analytically continued from $\IR$ to the strip 
$\Im \mathrm m(t)\in[0,1]$, and the value of the inner product is then given 
by $f(i/2)$:
\begin{equation}\label{eq: def inner product on L^2A}
\langle\sqrt\phi,\sqrt\psi\,\rangle 
    :=\, \underset{t\to i/2}{\text{anal.\,cont.}}\,\,\phi([D\phi:D\psi]_t).
\end{equation}
In particular, we have $\|\sqrt{\phi}\,\|^2=\phi(1)$.

The cone of positive elements in $L^2A$ is given by 
$L^2_+(A):=\{\sqrt\phi\,\,|\,\phi\in L^1_+(A)\}$, and
the two actions of $A$ on $L^2A$ are prescribed by
\[
\langle a\sqrt\phi\, b,\sqrt\psi\rangle 
  :=\, \underset{t\to i/2}{\text{anal.\,cont.}}\,\,
           \phi\big([D\phi:D\psi]_t\sigma_t^\psi(b)a\big),
\] 
where $\sigma_t^\psi$ is the modular flow\footnote{We do not assume that $\psi$ is faithful in defining the modular flow $\sigma_t^\psi$.  For $a \in A$, we have $\sigma_t^\psi(a)\in \bfs_\psi A\, \bfs_\psi$.}.
The space $L^2A$ is also equipped with the modular conjugation $J_A$, that
sends $\lambda\sqrt{\phi}$ to $\bar\lambda\sqrt{\phi}$ for $\lambda\in\IC$, 
and satisfies
\begin{equation}\label{eq: main property of J}
J_A(a\xi b)=b^*J_A(\xi)a^*.
\end{equation}
Altogether, the triple $(L^2(A),J_A,L^2_+(A))$ is a standard form for the 
von Neumann algebra $A$; compare \cite[p.528]{Connes(Non-commutative-geometry)}.

The  above constructions are compatible with spatial tensor product in 
the sense that there is a natural isomorphism 
$L^2(A\,\bar\otimes\, B)\cong L^2(A)\otimes L^2(B)$
that respects the left and right $A\,\bar\otimes\, B$-actions, 
and intertwines the modular involutions --- see Example~\ref{ex:standard-forms-ox}.

\begin{remark}[The modular algebra]\label{rem: algebra L_*(A)}
  The construction of $L^2A$ is best understood in the larger context of
  the modular algebra \cite{Pavlov(PhD-thesis), 
              Yamagami(Algebraic-aspects-in-modular-theory)} --- recall Remark \ref{rem: algebra L_*(A) -- PRE}.
  The modular algebra is 
  \[
  L^*A :=  \bigoplus_{p\in \IC^\times_{\Re \mathrm e\ge 0} \cup \{\infty\}} L^pA,
  \]
  and can be represented as an algebra of unbounded operators on a Hilbert space.
  The product sends $L^p(A)\times L^q(A)$ to $L^\frac1{1/p+1/q}(A)$, and $L^\infty (A)$ is a synonym for $A$.
  Given $p \in \IC^\times_{\Re \mathrm e\ge 0}$, then for every $\phi \in L^1_+A$,
  its $p$th root $\phi^{1/p}$ (in the sense of functional calculus) belongs to $L^pA$.
  In particular, we have $\sqrt\phi\equiv\phi^{1/2}\in L^2(A)$.
  The modular conjugation $J_A:L^2(A)\to L^2(A)$ is then simply the restriction of the $*\text{-operation}$ on $L^*A$.
  There is also a faithful normal trace $ Tr:L^*A\to \IC$ given by
\[
\qquad Tr(\phi)=\begin{cases}\phi(1)&\text{for}\,\, \phi\in L^1(A)\\
0&\text{for}\,\, \phi\in L^p(A),\,\,\,\,p\not =1.
\end{cases}
\]
  By definition, it satisfies $Tr(\phi a) = \phi(a)$ for $\phi \in L^1 A$ and $a \in A$.

  Using complex exponentiation in the algebra $L^*A$, the Radon-Nikodym derivative and the modular flow can be recovered\footnote{Unfortunately, one  
  cannot use \eqref{eq: [Dphi:Dpsi] and sigma_t^psi(a)} to \emph{define} $[D\phi:D\psi]_t$ and $\sigma_t^\psi$, as the Radon-Nikodym derivative and the modular flow are needed for the construction of the modular algebra --- see \cite{Yamagami(Algebraic-aspects-in-modular-theory)}.} as
\begin{equation}\label{eq: [Dphi:Dpsi] and sigma_t^psi(a)}
\begin{split}\qquad\qquad
[D\phi:D\psi]_t =\,&\phi^{it}\psi^{-it}\\
\sigma_t^\psi(a) =\psi^{it}&a\,\psi^{-it}
\end{split}\qquad\quad\,\,\,\,\, (t \in \IR).
\end{equation}
  We can therefore rewrite the quantity that appears in the right hand side of \eqref{eq: def inner product on L^2A} as
  \[
   \phi([D\phi:D\psi]_t) = 
      Tr(\phi [D\phi:D\psi]_t ) = Tr(\phi \phi^{it} \psi^{-it}) 
       =  Tr( \phi^{1+it} \psi^{-it}).
  \]
  The last expression $Tr( \phi^{1+it} \psi^{-it})$ can be evaluated for any $t$ in the strip $\Im \mathrm m(t)\in[0,1]$, because $\Re \mathrm e(1+it)$ and $\Re \mathrm e(-it)$ are both non-negative there.
  Moreover, the dependence on $t$ is analytic by \cite[Corollary 2.6]{Yamagami(Algebraic-aspects-in-modular-theory)}.
  One can therefore rewrite the inner product on $L^2(A)$ as
\[
\langle \sqrt{\phi}, \sqrt{\psi}\, \rangle = Tr( \phi^{1+it} \psi^{-it})|_{t=i/2} = Tr ( \phi^{1/2} \psi^{1/2}),
\]
  and the fact that it is symmetric follows from the trace property.
  The inner product also admits the following alternative definition:
\[
\langle\sqrt\phi,\sqrt\psi\,\rangle 
    :=\, \underset{t\to -i/2}{\text{anal.\,cont.}}\,\,\psi([D\phi:D\psi]_t).
\]
  This definition agrees with definition \eqref{eq: def inner product on L^2A}
  because $\psi([D\phi:D\psi]_t) = Tr (\psi \phi^{it} \psi^{-it}) = Tr (\phi^{it}\psi^{1-it})$ and $Tr (\phi^{it}\psi^{1-it})|_{t=-i/2}=Tr ( \phi^{1/2} \psi^{1/2})$.
\end{remark}

We will need the following lemma later on in order to identify the dual of the bimodule ${}_AL^2B_B$ associated to a finite homomorphism $A \to B$.
\begin{lemma}\label{lem: p_i sqrt(phi) p_j = 0}
Let $\{p_i\in A\}$ be orthogonal projections adding up to $1$.
If $\phi\in L^1_+(A)$ satisfies $p_i \phi\, p_j = 0$ for all $i$ and $j$ with $i \neq j$, then $p_i \sqrt{\phi}\, p_j = 0$ for all $i$ and $j$ with $i \neq j$.
\end{lemma}

\begin{proof}
Applying functional calculus to an (unbounded) operator in block diagonal form yields an operator in block diagonal form.
The result follows since the modular algebra has a representation by unbounded operators \cite{Yamagami(Algebraic-aspects-in-modular-theory)},
and $\sqrt{\phi}$ is obtained from $\phi$ by functional calculus.
\end{proof}

In our analysis of conditional expectations in section \ref{sec: Pimsner-Popa inequality}, we will use the following general fact relating Radon-Nikodym 
derivatives in different 
algebras --- see \cite[Lemma 1.4.4]{Connes(Une-classification-type-III)} and \cite[Theorem 4.7]{Haagerup(Operator-valued-weights)}.
Let $A\subset B$ be a subalgebra, and let $E \colon B \to A$ be a faithful completely positive normal map such that
$E(axb) =a E(x) b$ for $x \in B$, $a,b \in A$; in this case,
\begin{equation}
  \label{eq:Radon-Nikodym-E}
   [D(\phi \circ E) : D(\psi \circ E)]_t = [D \phi : D \psi]_t.
\end{equation}

\subsection*{Functoriality of the \emph{L}${}^\mathbf 2$-space}

The following theorem is closely related to some known results \cite{Kawakami-Watatani, Kosaki-Longo(A-remark-on-the-minimal-index)}. Nevertheless, it appears to be new:

\begin{theorem}\label{thm: functoriality of L2}
The assignment $A\mapsto L^2(A)$ defines a functor from the category of von Neumann algebras with finite-dimensional center and finite homomorphisms,
to the category of Hilbert spaces and bounded linear maps.
\end{theorem}

\begin{proof}
Given a finite homomorphism $A\to B$ between von Neumann algebras with finite-dimensional center,
let $E_{A,B}:B\to A$ be the map given by
\begin{equation}\label{eq: minimal conditional expectation}
E_{A,B}(b)\xi:=R^*(b\otimes 1)R\,\xi\,,\qquad\quad\xi \in L^2(A),
\end{equation}
where $R:{}_AL^2(A)_A\to{}_AL^2(B)\boxtimes_B\overline{L^2(B)}{}_A$ is as in \eqref{eq:duality maps},
and the $b$ that appears in the right hand side of \eqref{eq: minimal conditional expectation}
acts by left multiplication on $L^2(B)$.
Graphically, this~is:
\[
\def\colA{black!10}
\def\colB{black!30}
\tikzmath[scale=\squarescale]
	{\fill[rounded corners=10, fill=\colA] (-15,-16) rectangle (13,16);
	\coordinate (x) at (-15,0);
	\draw 	(-5,0) node(a)[fill=white]{$E_{A,B}(b)$};
	\draw (a.north east -| x) -- (a.north east) -- (a.south east) -- (a.south east -| x);
	}
\;\;:=\;\;
\tikzmath[scale=\squarescale]
	{\fill[rounded corners=10, fill=\colA] (-15,-16) rectangle (13,16);
	\draw[rounded corners=10, fill=\colB] (-6,-9) rectangle (5,9);
	\coordinate (x) at (-15,0); \coordinate (x') at (-15.1,0);
	\draw 	(-6,0) node(a)[fill=white]{$b$};
	\fill[fill=white] (a.north east -| x') -- (a.north east) -- (a.south east) -- (a.south east -| x');
	\draw (a.north east -| x) -- (a.north east) -- (a.south east) -- (a.south east -| x);
	\node(a) at (-8,0) [inner sep=0]{$b$};
	}\,.
\]
As before, the two shades represent the algebras $A$ and $B$, and the lines stand for the bimodule ${}_AL^2B_B$ and its dual.
The fact that the box labeled $E_{A,B}(b)$ extends to the left of the picture refers to the fact that the map
$E_{A,B}(b):{}_AL^2A_A\to{}_AL^2A_A$ is only right $A$-linear.

The map \eqref{eq: minimal conditional expectation} satisfies $E_{A,B}(aba')=aE_{A,B}(b)a'$ for any $a,a'\in A$ and $b\in B$.
Moreover, for every sequence $A\to B\to C$ of composable arrows, we have
\begin{equation}\label{eq:E circ E = E}
E_{A,B}\circ E_{B,C}=E_{A,C}
\end{equation}
by Lemma~\ref{lem: product of dualizables}.
The map $L^2(f):L^2(A)\to L^2(B)$ associated to the finite homomorphism $f: A \to B$ is then defined by
\begin{equation}\label{eq: sqrt(phi o E)}
L^2(f)\,:\,\sqrt\phi\mapsto\sqrt{\phi\circ E_{A,B}}.
\end{equation}
To see that this map is well defined and bounded, we exhibit a constant $C$ such that
\[
\qquad\Big\|\sum_jc_j\sqrt{\phi_j\circ E_{A,B}}\Big\|^2\le\, C\cdot\Big\|\sum_jc_j\sqrt{\phi_j}\Big\|^2\qquad \forall \,c_j\in \IC,\,\,\phi_j\in L^1_+(A).
\]
Let $\{p_\alpha\}$ be the minimal central projections of $A$.
Since $E_{A,B}(1)$ is central, we can write it as
$E_{A,B}(1)=\sum_\alpha C_\alpha p_\alpha$ for some given constants $C_\alpha$.
Let
\[
C:=\max_\alpha\, C_\alpha=\|E_{A,B}(1)\|.
\]
Using the shorthand notation $\phi_{j,\alpha}:=\phi_j p_\alpha$, we then have
\[
\begin{split}
\textstyle\big\|\sum_j&c_j\textstyle\sqrt{\phi_j\circ E_{A,B}}\big\|^2
=\sum_{j,k}c_j\bar c_k \big\langle\sqrt{\phi_j\circ E_{A,B}},\sqrt{\phi_k\circ E_{A,B}}\,\big\rangle\\
&\textstyle\quad=\sum_{j,k}c_j\bar c_k\, \underset{t\to i/2}{\text{anal.\,cont.}}\,\,\phi_j\circ E_{A,B}\big([D(\phi_j\circ E_{A,B}):D(\phi_k\circ E_{A,B})]_t\big)\\
&\textstyle\quad=\sum_{j,k}c_j\bar c_k\, \underset{t\to i/2}{\text{anal.\,cont.}}\,\,\phi_j\circ E_{A,B}\big([D\phi_j:D\phi_k]_t\big)\\
&\textstyle\quad=\sum_{j,k}c_j\bar c_k\, \underset{t\to i/2}{\text{anal.\,cont.}}\,\,\phi_j \big(E_{A,B}(1)\,[D\phi_j:D\phi_k]_t\big)\\
&\textstyle\quad=\sum_{\alpha,j,k}c_j\bar c_k\, \underset{t\to i/2}{\text{anal.\,cont.}}\,\,\phi_{j,\alpha} \big(C_\alpha[D\phi_{j,\alpha}:D\phi_{k,\alpha}]_t\big)\\
&\textstyle\quad=\sum_\alpha C_\alpha \textstyle\big\|\sum_jc_j\sqrt{\phi_{j,\alpha}}\big\|^2
\le C\cdot\sum_\alpha \textstyle\big\|\sum_jc_j\sqrt{\phi_{j,\alpha}}\big\|^2
=C\cdot\big\|\sum_jc_j\sqrt{\phi_j}\big\|^2,
\end{split}
\]
where the third equality follows from~\eqref{eq:Radon-Nikodym-E}
and the fourth one follows from the $A$-linearity of~$E_{A,B}$.

The compatibility of \eqref{eq: sqrt(phi o E)} with composition follows from \eqref{eq:E circ E = E}.
\end{proof}

\begin{remark}
Given a finite homomorphism $f:A\to B$ between von Neumann algebras with finite-dimensional centers, one can also define
$$L^p(f):L^pA\to \,\,L^pB;\, \phi^{1/p}\,\,\mapsto\,\,{(\phi\circ E_{A,B})^{1/p}}.$$
These assemble to a $*$-algebra homomorphism $\bigoplus L^pA\to \bigoplus L^pB$; see \cite[section~3]{Yamagami(Algebraic-aspects-in-modular-theory)}.
\end{remark}

\begin{corollary}\label{cor: contragredient bimodule}
Let ${}_AH_B$ be a bimodule between von Neumann algebras with finite-dimensional center.
Then its dual bimodule, if it exists, is canonically isomorphic to the complex conjugate Hilbert space, with actions given by $b\bar \xi a:= \overline{a^* \xi b^*}$.
\end{corollary}
\begin{proof}
Let ${}_AH_B$ be dualizable.
By Lemma \ref{lem: Lemma B} and the decomposition \eqref{eq: H = oplus H_ij}, this bimodule is a finite direct sum of 
irreducible bimodules.
Both duals and complex conjugates being compatible with the direct sum operation, it is enough to treat the irreducible case.
We assume for simplicity that the action $\rho:A\to\bfB(H)$ is faithful.
The general case follows. 

Let ${}_B \overline H^c\!\!{}_A$ denote the complex conjugate of ${}_AH_B$, and let $B'$ be the commutant of $B$ on $H$.
By Proposition \ref{prop: characterization of invertible bimodules}, we have ${}_{B'}H\boxtimes_B\overline H^c\!\!{}_{B'}\cong {}_{B'}L^2(B')_{B'}$, and so
\[
\begin{split}
{}_AH\boxtimes_B\overline H^c\!\!{}_A 
&\cong {}_AL^2(B') \boxtimes_{B'} H\boxtimes_B\overline H^c \boxtimes_{B'} L^2(B')_A\\
&\cong\, {}_AL^2(B')\,\boxtimes_{B'} L^2(B')\boxtimes_{B'} L^2(B')_A
\,\cong {}_AL^2(B')_A.
\end{split}
\]
By Theorem \ref{thm: functoriality of L2}, 
we therefore get a map ${}_AL^2(A)_A\to {}_AH\boxtimes_B\overline H^c\!\!{}_A$
which is non-trivial by construction --- see for instance equation \eqref{eq: norm of L2(iota)}.
The result now follows from Lemma \ref{lem: characterization of duals}.
\end{proof}

\begin{remark}
The isomorphism between any dual and the complex conjugate bimodule constructed in the proof of Corollary \ref{cor: contragredient bimodule} is in fact unitary.
We do not include a proof --- see Proposition \ref{prop: unitarity of (6.14)} for a related result.
\end{remark}

In the special case of the bimodule ${}_AL^2(B)_B$ associated to a finite homomorphism $f:A\to B$, together with a chosen dual $({}_B\overline{L^2B}_A, R, S)$, the isomorphism $\overline{L^2B} \cong \overline{L^2B}^c$ is given by
\begin{equation*}
\begin{split}
\overline{L^2B}\cong\overline{L^2B}\boxtimes_A L^2A 
\xrightarrow{\!1\otimes L^2\!(f)\!}\, 
\overline{L^2B}\boxtimes_A L^2B                                         
    \cong \,\,&
\overline{L^2B}\boxtimes_A L^2B\boxtimes_B L^2B                       
          \\\xrightarrow{S^*\otimes 1}\,&
L^2B\boxtimes_B L^2B\cong L^2B 
   \xrightarrow{J}
\overline{L^2B}^c,
\end{split}
\end{equation*} 
where $J$ is the modular conjugation.
This isomorphism $\overline{L^2B} \cong \overline{L^2B}^c$ is chosen so as to make the composite
\begin{equation*}
\begin{split}
{}_AL^2(A)_A \xrightarrow {\,\,R\,\,} {}_A&L^2(B)\boxtimes_B \overline{L^2(B)}{}_A\\ \cong {}_A&L^2(B)\boxtimes_B \overline{L^2(B)}^c\!\!{}_A \xrightarrow {1\otimes J}
{}_AL^2(B)\boxtimes_B L^2(B)_A\cong {}_AL^2(B)_A
\end{split}
\end{equation*} 
equal to $L^2(f)$.

Instead of identifying the dual of ${}_AL^2B_B$ with ${}_B\overline{L^2B}^c\!\!{}_A$, we can identify it with ${}_BL^2B_A$, as follows.  There is an isomorphism $\Phi$ between any dual of ${}_AL^2(B)_B$ and ${}_BL^2(B)_A$ given by
\begin{equation}\label{eq: It's L2(f)}
\begin{split}
{}_B\overline{L^2B}_A\cong\,&{}_B\overline{L^2B}\boxtimes_A L^2A_A  \xrightarrow{\!1\otimes L^2\!(f)\!}\, 
{}_B\overline{L^2B}\boxtimes_A L^2B_A                                              \\&\cong\,\,
{}_B\overline{L^2B}\boxtimes_A L^2B\boxtimes_B L^2B_A      
           \xrightarrow{S^*\otimes 1}
{}_BL^2 B \boxtimes_B L^2B \cong {}_BL^2B_A \, .
\end{split}
\end{equation} 
In graphical notation we have
\[
\def\colA{black!10}
\def\colB{black!30}
\Phi\,:=\,\,\,\tikzmath[scale=\squarescale]
	{\fill[rounded corners=10, fill=\colA] (-19,-14) rectangle (26,14);
	\fill[rounded corners=9, fill=\colB] (19,-14) -- (19,6) -- (0,6) -- (0,-6) -- (-12,-6) [sharp corners]-- (-12,14)  [rounded corners=10]-- (-19,14) -- (-19,-14) [sharp corners]-- cycle;
	\draw[rounded corners=9] (19,-14) -- (19,6) -- (0,6) -- (0,-6) -- (-12,-6) -- (-12,14);
	\node[above,yshift=-2] at (-12,14) {$\scriptstyle {}_B\overline{L^2B}_A$};
	\node[below,yshift=2] at (19,-14) {$\scriptstyle {}_B L^2B_A$};
	\draw (-6,-6) node[fill=white, draw]{$S^*\!$} (9.5,6) node[fill=white, draw, inner xsep=2]{$L^2(f)$};}\,.
\]
The isomorphism $\Phi$ makes the following diagram commutative:
\begin{equation}\label{eq: iso between dual of A--L2B--B and B--L2B--A}
\tikzmath{
\matrix [matrix of math nodes,column sep=1.5cm,row sep=8mm]
{ 
|(a)| {}_AL^2A_A \pgfmatrixnextcell |(b)| {}_AL^2(B)\boxtimes_B \overline{L^2(B)}{}_A \\ 
|(c)| {}_AL^2B_A \pgfmatrixnextcell |(d)| {}_AL^2(B)\boxtimes_B L^2(B){}_A\\ 
}; 
\draw[->] (a) --node[above]{$\scriptstyle R$} (b);
\draw[->] (a) --node[left]{$\scriptstyle L^2(f)$} (c);
\draw[->] (b) --node[right]{$\scriptstyle 1\otimes\Phi$} (d);
\draw[->] (c) --node[above]{$\scriptstyle \cong$} (d);
}\, .
\end{equation} 

\begin{proposition}\label{prop: unitarity of (6.14)}
Let $f: A \to B$ be a finite homomorphism, and let $(\overline{L^2B}, R, S)$ be a chosen dual to the bimodule ${}_AL^2B_B$ associated to $f$.  The isomorphism $\Phi := (S^* \otimes 1)(1 \otimes L^2(f))$ from ${}_B\overline{L^2B}_A$ to ${}_BL^2B_A$ is unitary.
\end{proposition}

\begin{proof}
The algebra $B\cap A'=\mathrm{End}({}_AL^2B_B)$ is finite-dimensional by Lemma \ref{lem: Lemma B} and decomposition \eqref{eq: H = oplus H_ij}.
Let $p_1, \ldots, p_n\in B\cap A'$ be mutually orthogonal minimal projections adding up to $1$,
and let $\bar p_i:=(S^*\otimes 1)(1\otimes p_i\otimes 1)(1\otimes R)$ be the dual projection defined in equation \eqref{eq: pibar}.
Let $E:B\to A$ be as in \eqref{eq: minimal conditional expectation}.
For every $i\neq j$ and $\phi\in L^1_+A$, the element $p_i(\phi\circ E)p_j\in L^1(B)$ is zero, as $(\,p_i(\phi\circ E)p_j)(b)=\phi\circ E(p_jb\,p_i)$ 
and
\[
\def\colA{black!10}
\def\colB{black!30}
E(p_jb\,p_i)=\,
\tikzmath[scale=\squarescale]
	{\fill[rounded corners=10, fill=\colA] (-15,-16) rectangle (12,16);
	\draw[rounded corners=10, fill=\colB] (-6,-13.5) rectangle (5,13.5);
	\coordinate (x) at (-15,0); \coordinate (x') at (-15.1,0);
	\draw 	(-5.35,0) node(a)[fill=white, inner ysep=3]{$b$};
	\fill[fill=white] (a.north east -| x') -- (a.north east) -- (a.south east) -- (a.south east -| x');
	\draw (a.north east -| x) -- (a.north east) -- (a.south east) -- (a.south east -| x);
	\node at (-6,0) [inner sep=0]{$b$};
	\node[draw, fill=white,inner sep=3] at (-6,8) {$p_i$};
	\node[draw, fill=white,inner sep=2.5] at (-6,-8) {$p_j$};
	}
\,=\,\tikzmath[scale=\squarescale]
	{\fill[rounded corners=10, fill=\colA] (-15,-16) rectangle (13,16);
	\draw[rounded corners=10, fill=\colB] (-6,-12) rectangle (5,12);
	\coordinate (x) at (-15,0); \coordinate (x') at (-15.1,0);
	\draw 	(-5.35,5) node(a)[fill=white, inner ysep=3]{$b$};
	\fill[fill=white] (a.north east -| x') -- (a.north east) -- (a.south east) -- (a.south east -| x');
	\draw (a.north east -| x) -- (a.north east) -- (a.south east) -- (a.south east -| x);
	\node at (-6,5) [inner sep=0]{$b$};
	\node[draw, fill=white,inner sep=3] at (5,0) {$\bar p_i$};
	\node[draw, fill=white,inner sep=2.5] at (-6,-5) {$p_j$};
	}
\,=\,\tikzmath[scale=\squarescale]
	{\fill[rounded corners=10, fill=\colA] (-15,-16) rectangle (12,16);
	\draw[rounded corners=10, fill=\colB] (-6,-13.5) rectangle (5,13.5);
	\coordinate (x) at (-15,0); \coordinate (x') at (-15.1,0);
	\draw 	(-5.35,7.8) node(a)[fill=white, inner ysep=3]{$b$};
	\fill[fill=white] (a.north east -| x') -- (a.north east) -- (a.south east) -- (a.south east -| x');
	\draw (a.north east -| x) -- (a.north east) -- (a.south east) -- (a.south east -| x);
	\node at (-6,7.8) [inner sep=0]{$b$};
	\node[draw, fill=white,inner sep=3] at (-6,-8) {$p_i$};
	\node[draw, fill=white,inner sep=2.5] at (-6,-.2) {$p_j$};
	}
=0.
\]
It follows from Lemma \ref{lem: p_i sqrt(phi) p_j = 0} that $p_i\sqrt{\phi\circ E\,}p_j=0$ for $i \neq j$.
The map $L^2(f):\sqrt{\phi}\mapsto\sqrt{\phi\circ E\,}$ therefore factors as
\[
\tikzmath{
\node[name = a, inner xsep=4] at (0,0) {$L^2(f)\,\,:\,\,L^2A$};
\node[name = b] at (5,0) {$L^2(B) \cong \bigoplus_{ij} p_iL^2(B) p_j$};
\node[name = c] at (5.7,-1) {$\bigoplus_i p_iL^2(B) p_i$};
\draw[->] (a) -- (b.west |- a);
\draw[->] ($(c.north) + (-.1,.05)$) arc (-180:0:.05) -- (b.south -| c);
\draw[->, dashed] ($(a.east)+(0,-.15)$) --node[below]{$\scriptstyle \bigoplus [L^2(f)]_i$} ($(c.west)+(0,.1)$);}\,.
\]

\noindent We have a similar factorization of $R$ by Lemma \ref{lem: sum of dualizables}:
\[
\tikzmath{
\node[name = a, inner xsep=4] at (0,0.1) {$R\,\,:\,\,L^2A$};
\node[name = b] at (5,0) {$L^2(B)\underset B\boxtimes \overline{L^2(B)} \cong \bigoplus_{ij} p_iL^2(B)\underset B\boxtimes \bar p_j\overline{L^2(B)}$};
\node[name = c] at (6.47,-1) {$\bigoplus_i p_iL^2(B)\boxtimes_B \bar p_i\overline{L^2(B)}$};
\draw[->] (a) -- (b.west |- a);
\draw[->] ($(c.north) + (-.1,.05)$) arc (-180:0:.05) -- (b.south -| c);
\draw[->, dashed] ($(a.east)+(0,-.15)$) --node[below]{$\scriptstyle \bigoplus R_i$} ($(c.west)+(0,.1)$);}\,.
\]

Let us write
\[
\Phi_{jk}\,:\,\,\bar p_j\overline{L^2(B)} \,\rightarrow\, L^2(B) p_k
\]
for the components of $\Phi$.
Given that ${}_B\,\bar p_j\overline{L^2(B)}{}_A$ and ${}_BL^2(B)p_k\,{}_A$ are irreducible bimodules, the maps $\Phi_{jk}$ are either zero or a scalar multiple of some unitary.
By the commutativity of \eqref{eq: iso between dual of A--L2B--B and B--L2B--A} (and since $R_i\neq 0$),
the subspace $\bigoplus_i p_iL^2(B)\boxtimes_B \bar p_i\overline{L^2(B)}$ of $L^2(B)\boxtimes_B \overline{L^2(B)}$ goes to $\bigoplus_i p_iL^2(B) p_i$ under the map
\[
\textstyle 1\otimes \Phi:\bigoplus_{ij}p_iL^2(B)\underset B \boxtimes \bar p_j\overline{L^2(B)}\to \bigoplus_{ik}p_iL^2(B)\boxtimes_B L^2(B)p_k\cong \bigoplus_{ik}p_iL^2Bp_k.
\]
It follows that $\Phi_{jk}=0$ whenever $j\neq k$. We can therefore rewrite $\Phi$ as
\[
\textstyle \Phi=\bigoplus \Phi_i: \bigoplus_i\bar p_i\overline{L^2(B)} \to \bigoplus_iL^2(B)p_i,
\]
where each $\Phi_i$ is a scalar multiple of some unitary.

To finish the argument, we show that each $\Phi_i$ has norm $1$.
Let $q_i\in Z(A')=Z(A)$ be the central support projection of $p_i\in A'$.
The maps
\[
\begin{split}
&[L^2(f)]_i:{}_AL^2A_A\to {}_A\,p_iL^2(B) p_i{}_A\quad\,\,\,\text{and}\\
&\,\,R_i\,:\,{}_AL^2A_A\to {}_A\,p_iL^2(B)\boxtimes_B \bar p_i\overline{L^2(B)}{}_A
\end{split}
\]
factor through ${}_A\,q_i(L^2A)_A$, and are therefore scalar multiples of partial isometries.
Given $\phi\in q_i(L^1_+A)$, we have
\[
\def\colA{black!10}
\def\colB{black!30}
\begin{split}
\big\|[L^2(f)]_i({\textstyle\sqrt{\phi}})\big\|^2&=\big\|p_iL^2(f)({\textstyle\sqrt{\phi}})\big\|^2=\big\|p_i{\textstyle\sqrt{\phi\circ E}}\big\|^2\\
&=\big\langle p_i{\textstyle\sqrt{\phi\circ E}},{\textstyle\sqrt{\phi\circ E}}\big\rangle=\phi\circ E(p_i)=E(p_i)\cdot \phi(1),
\end{split}
\]
where $E(p_i)\in q_iZ(A)\cong \IC$.
Similarly, we have
\[\vspace{.2cm}
\def\colA{black!10}
\def\colB{black!30}
\begin{split}
\big\|R_i(&\textstyle\sqrt{\phi})\big\|^2=\big\|p_iR(\sqrt{\phi})\big\|^2\\
&=\,\Big\langle\,\,
\tikzmath[scale=\squarescale]{\clip (-9.5,-4.5) rectangle (4,13);\fill[\colA] (-9.5,-7) rectangle (4,7);
\draw[rounded corners=6, fill=\colB] (-5,-10) rectangle (1,5);\draw (-5,0) node[fill=white, draw, inner ysep=2, inner xsep=2]{$p_i$};
\draw (-9.5,7) --node[above, xshift = -.5, yshift = -.5]{$\scriptstyle \sqrt{\phi}$} (4,7);}\,,\,
\tikzmath[scale=\squarescale]
{\clip (-9.5,-4.5) rectangle (4,13);\fill[\colA] (-9.5,-7) rectangle (4,7);\draw[rounded corners=6, fill=\colB] (-5,-10) rectangle (1,5);
\draw (-5,0) node[fill=white, draw, inner ysep=2, inner xsep=2]{$p_i$};\draw (-9.5,7) --node[above, xshift = -.5, yshift = -.5]{$\scriptstyle \sqrt{\phi}$} (4,7);}\,\,\Big\rangle
\,=\,
\tikzmath[scale=\squarescale]{\useasboundingbox (-9.5,-10) rectangle (4,10);\fill[\colA] (-9.5,-7) rectangle (4,7);\draw[rounded corners=6, fill=\colB] (-5,-5) rectangle (1,5);
\draw (-5,0) node[fill=white, draw, inner ysep=2, inner xsep=2]{$p_i$};\draw (-9.5,7) --node[above, xshift = -.5, yshift = -.5]{$\scriptstyle \sqrt{\phi}$} (4,7);
\draw (-9.5,-7) --node[below, xshift = 1, yshift = .5]{$\scriptstyle {\sqrt{\phi}\,}^*$} (4,-7);}
\,=\,\tikzmath[scale=\squarescale]
{\fill[rounded corners=5, fill=\colA] (-9.5,-7) rectangle (4,7);\draw[rounded corners=6, fill=\colB] (-5,-5) rectangle (1,5);
\draw (-5,0) node[fill=white, draw, inner ysep=2, inner xsep=2]{$p_i$};}\cdot
\textstyle\langle\sqrt{\phi},\sqrt{\phi}\,\rangle=E(p_i)\cdot\phi(1).
\end{split}
\]
It follows that $\|R_i\|^2=\|[L^2(f)]_i\|^2=E(p_i)$.
Since \eqref{eq: iso between dual of A--L2B--B and B--L2B--A} is commutative, we thus get $\|\Phi_i\|=\|[L^2(f)]_i\|/\|R_i\|=1$,
and the map $\Phi=\bigoplus \Phi_i$ is therefore unitary.
\end{proof}

The reader may wonder whether the condition of finite-dimensional center was really needed in Theorem \ref{thm: functoriality of L2}.
We saw in Theorem \ref{thm: normalizing duality data} that a bimodule between von Neumann algebras with finite-dimensional centers is dualizable if and only if
there exist maps $R$ and $S$ satisfying \eqref{eq: duality equations}: 
though a priori dualizability requires both conditions \eqref{eq: duality equations} and \eqref{eq:duality normalization}, in fact it is detected by condition \eqref{eq: duality equations} only.
If the centers of $A$ and $B$ are not atomic (that is, if one of them contains $L^\infty([0,1])$), then we do not know how to formulate \eqref{eq:duality normalization}.
We therefore do not have a good notion of duality in that context; however, we may still define a homomorphism $f:A\to B$ between arbitrary von Neumann algebras to be finite
if there exist maps $R$ and $S$ satisfying \eqref{eq: duality equations}, that is giving a not-necessarily normalized dual for the bimodule ${}_AL^2B_B$.

\begin{conjecture}\label{conj: L2 functoriality}
The assignment $A\mapsto L^2A$ extends to a functor from the category of all von Neumann algebras and finite homomorphisms
to the category of Hilbert spaces and bounded linear maps.
\end{conjecture}

The following two lemmas describe how the functor $L^2$ interacts with the basic operations of taking corner and block-diagonal subalgebras.
Recall from Lemma \ref{lem: L^2(pAp) = pL^2(A)p} that the $L^2$-space of the corner algebra $A_0:=pAp$ is given by $L^2(A_0)=p(L^2A)p$.

\begin{lemma}\label{lem: L^2 of pAp --> pBp}
Let $f:A\to B$ be a finite homomorphism between von Neumann algebras with finite-dimensional centers.
Given a projection $p\in A$, let $A_0:=pAp$, $B_0:=pBp$, and $f_0:=f|_{A_0}:A_0\to B_0$, where we identify $p$ with its image $f(p)\in B$.
Then the homomorphism $f_0$ is finite, and we have
\[
L^2(f_0)=L^2(f)|_{L^2(A_0)},
\]
where we have identified $L^2(A_0)$ and $L^2(B_0)$ 
with the subspaces $pL^2(A)p$ and $pL^2(B)p$ of $L^2(A)$ and $L^2(B)$ respectively.
\end{lemma}

\begin{proof}
The structure maps \eqref{eq:duality maps} for the dual of ${}_AL^2(B)_B$ restrict to maps
\[
\begin{split}
R_0:{}_{A_0}L^2A_0{}_{A_0}={}_{A_0}pL^2\!Ap_{A_0}&\to
{}_{A_0}pL^2B\boxtimes_{B}\overline{L^2B}p{}_{A_0}=
{}_{A_0}pL^2Bp\boxtimes_{B_0}p\overline{L^2B}p{}_{A_0},\\
S_0:{}_{B_0}L^2B_0{}_{B_0}={}_{B_0}pL^2\!Bp_{B_0}&\to
{}_{B_0}p\overline{L^2B}\boxtimes_{A} L^2B p{}_{B_0}=
{}_{B_0}p\overline{L^2B}p\boxtimes_{A_0}pL^2Bp{}_{B_0}.
\end{split}
\]
Here we use the invertibility of ${}_BL^2Bp_{B_0}$ to rewrite
the targets of $R_0$ and $S_0$. 
These satisfy the duality equations \eqref{eq: duality equations} and the normalization \eqref{eq:duality normalization}, and therefore exhibit ${}_{B_0}(p\overline{L^2(B)}p){}_{A_0}$ as the dual of 
${}_{A_0}L^2({B_0})_{B_0}$.
For every $b\in B_0$, we have
\[
\begin{split}
E_{A,B}(b)\xi = R^*(b\otimes 1)R\xi = R_0^*(b\otimes 1)R_0\xi = E_{A_0,B_0}(b)\xi\qquad&\text{for } \xi\in L^2(A_0)\,,\\
E_{A,B}(b)\xi = R^*(b\otimes 1)R\xi = R^*(pb\otimes 1)R(1-p)\xi = 0\hspace{1.1cm}&\text{for } \xi\in L^2(A_0)^\bot\!,
\end{split}
\]
from which it follows that $E_{A_0,B_0}=E_{A,B}|_{B_0}$.
Given a state $\phi:A_0\to \IC$, the image of $L^2(f_0)(\sqrt\phi)$
in $L^2(B)$ is the square root of
\[
b\mapsto \phi(E_{A_0,B_0}(pbp))=\phi(E_{A,B}(pbp))=\phi(pE_{A,B}(b)p),
\]
and is thus equal to the image of $\sqrt{a\mapsto \phi(pap)}$ under $L^2(f)$.
\end{proof}

\begin{lemma}\label{lem: L^2 of (+)p_iAp_i --> A}
Let $A$ be a factor, and $p_1,\ldots,p_n\in A$ be a collection of orthogonal projections that add up to $1$.
Let $\iota:\bigoplus p_iAp_i\to A$ denote the inclusion.
Then $L^2(\iota)$ is the natural inclusion
\[
L^2\big({\textstyle\bigoplus}\, p_iA\,p_i\big)\,\cong\, \bigoplus\, p_iL^2(A)\,p_i\,\hookrightarrow\, L^2(A),
\]
where the first isomorphism is given by Lemma \ref{lem: L^2(pAp) = pL^2(A)p}. In particular, $L^2(\iota)$ is an isometry.
\end{lemma}

\begin{proof}
We write $A_i$ for $p_iAp_i$.
The inclusions
\[
R:{}_{\oplus A_i}L^2\big({\textstyle\bigoplus} A_i\big){}_{\oplus A_i} \cong {\textstyle\bigoplus} p_i L^2(A) p_i\hookrightarrow L^2(A)\cong
{}_{\oplus A_i}L^2(A)\boxtimes_A L^2(A)_{\oplus A_i}
\]
\[
S:{}_AL^2(A)_A\hookrightarrow \underset i{\textstyle\bigoplus} L^2(A) \cong
\underset i{\textstyle\bigoplus} L^2(A)p_i\underset{\oplus A_i}\boxtimes p_iL^2(A)\cong {}_A L^2(A)\underset{\oplus A_i}\boxtimes L^2(A)_A
\]
exhibit ${}_{\oplus A_i}L^2(A)_A$ as the dual of ${}_A L^2(A)_{\oplus A_i}$.
For $\xi_i\in L^2(A_i)$ and $a\in A$, equation \eqref{eq: minimal conditional expectation} reads
\[
\underset i{\textstyle\bigoplus}\,\xi_i\stackrel R\mapsto \underset i{\textstyle\sum}\, \xi_i \stackrel a \mapsto \underset i{\textstyle\sum}\, a\xi_i \stackrel {R^*}\mapsto 
\underset j{\textstyle\bigoplus}\, p_j(\underset i{\textstyle\sum}a\xi_i)p_j = \underset i{\textstyle\bigoplus}\, p_iap_i\xi_i = \big(\underset i{\textstyle\bigoplus}p_iap_i\big)(\underset i{\textstyle\bigoplus}\xi_i).
\]
The map $E:=E_{\oplus A_i,\, A}$ is therefore given by $E(a)=\bigoplus q_i(a)$,
where $q_i(a):=p_iap_i$.
It follows that $L^2(\iota)(\sqrt{\oplus\phi_i})=\sqrt{(\oplus\phi_i)\circ E}=\sqrt{{\scriptstyle\sum}\, \phi_i\circ q_i}=\sum \sqrt{\phi_i\circ q_i}$.
Since $\sqrt{\phi_i}\in L^2(A_i)$ maps to $\sqrt{\phi_i\circ q_i}\in L^2(A)$ under the map described in Lemma \ref{lem: L^2(pAp) = pL^2(A)p}, this finishes the proof.
\end{proof}

One drawback of the construction presented in Theorem \ref{thm: functoriality of L2}
is that the maps $L^2( f):L^2(A)\to L^2(B)$ are not isometric.
For example, if $\iota:A\to B$ is a finite map between factors, then $L^2(\iota)$ is $\sqrt[4]{[B:A]}=\sqrt{\llbracket B : A\rrbracket}$ times an isometry.
This can be checked on positive vectors: since $\|\sqrt{\phi}\,\|^2=\phi(1)$ and $\|\sqrt{\phi\circ E_{A,B}}\,\|^2=\phi(E_{A,B}(1))=E_{A,B}(1)\phi(1)$,
it follows that
\begin{equation} \label{eq: norm of L2(iota)}
{\|L^2(\iota)({\textstyle\sqrt{\phi}})\|}\big/{\textstyle\|\sqrt{\phi}\|}=\sqrt{\phantom{R*|}\hspace{.55cm}}\hspace{-1.25cm}{E_{A,B}(1)}=\sqrt{R^*R\phantom{|}}=\sqrt{\dim(_{A}L^2B_B)}
\end{equation}
for any $\sqrt{\phi}\in L^2_+(A)$.
In some sense, that is inevitable.
Assuming that $\iota$ is injective, let $L^2(\iota)_\mathrm{iso}$ denote the isometry in the polar decomposition of $L^2(\iota)$.
The assignment
\begin{equation}\label{eq: the functor L^2_iso}
\big(\iota:A\to B\big)\,\,\mapsto\,\, \big(L^2(\iota)_\mathrm{iso}:L^2(A)\to L^2(B)\big)
\end{equation}
is not a functor --- this issue is already visible with finite-dimensional commutative von Neumann algebras.
Nevertheless, we have:

\begin{proposition}\label{prop: functoriality of L^2_iso}
When restricted to the subcategory of von Neumann algebras with finite-dimensional center and injective finite homomorphisms $\iota:A\to B$ that satisfy $Z(B)\subset \iota(A)$,
the assignment $\iota \mapsto L^2(\iota)_\mathrm{iso}$ is a functor.
\end{proposition}

\begin{proof}
We can write $\iota:A\to B$ as a direct sum of maps $\iota_j:A_j\to B_j$, where each $B_j$ is a factor.
Let us decompose each $A_j$ as a direct sum of factors $A_j=\bigoplus_i A_{ij}$, where $A_{ij}=p_{ij}A_j$, and $p_{ij}$ are the minimal central projections of $A_j$.
We can then factor $\iota$ as
\[
\iota\,\,:\,\,\textstyle A=\underset{ij}\bigoplus\, A_{ij}\to \underset{ij}\bigoplus\, p_{ij}B_j p_{ij} \to \underset{j}\bigoplus\, B_j = B.
\]
Applying the functor $L^2$ (as defined in Theorem \ref{thm: functoriality of L2}) to the above maps, we get
\[
L^2(\iota)\,\,:\,\,\textstyle L^2(A)=\underset{ij}\bigoplus\, L^2(A_{ij})\to \underset{ij}\bigoplus\, L^2(p_{ij}B_j p_{ij}) \xrightarrow{\star} \underset{j}\bigoplus\, L^2(B_j) = L^2(B).
\]
The map $\star$ is an isometry by Lemma \ref{lem: L^2 of (+)p_iAp_i --> A}.
The isometry $L^2(\iota)_{\mathrm{iso}}$ is therefore the composite of 
$L^2_{\mathrm{iso}}:\bigoplus\, L^2(A_{ij})\to \bigoplus\, L^2(p_{ij}B_j p_{ij})$
with the natural inclusion $\bigoplus\, L^2(p_{ij}B_j p_{ij}) \hookrightarrow \bigoplus\, L^2(B_j)$
described in Lemma \ref{lem: L^2(pAp) = pL^2(A)p}.

Given two composable inclusions $\iota:A\to B$ and $\kappa:B\to C$ with $Z(B) \subset \iota(A)$ and $Z(C) \subset \kappa(B)$,
we now show that $L^2(\kappa\circ \iota)_\mathrm{iso}=L^2(\kappa)_\mathrm{iso}\circ L^2(\iota)_\mathrm{iso}$.
Let us write $C=\bigoplus C_k$, $B=\bigoplus B_{jk}$, and $A=\bigoplus A_{ijk}$ as sums of factors, where $\iota(A_{ijk})\subset B_{jk}$ and $\kappa(B_{jk})\subset C_k$.
The corresponding minimal central projections are denoted $p_{ijk}\in A_{ijk}$ and $q_{jk}\in B_{jk}$.
To compare $L^2(\kappa\circ \iota)_\mathrm{iso}$ with $L^2(\kappa)_\mathrm{iso}\circ L^2(\iota)_\mathrm{iso}$, we consider the following diagram
\[
\qquad\tikzmath{
\matrix [matrix of math nodes,column sep=1cm,row sep=5mm]
{ 
|(a)| \bigoplus_{jk} L^2(B_{jk}) \pgfmatrixnextcell |(b)| \bigoplus_{jk} L^2(q_{jk}Cq_{jk}) \pgfmatrixnextcell |(c)[xshift=.4]| \bigoplus_{k} L^2(C_{k})\\ 
|(d)| \bigoplus_{ijk} L^2(p_{ijk}Bp_{ijk}) \pgfmatrixnextcell |(e)| \bigoplus_{ijk} L^2(p_{ijk}Cp_{ijk})\\ 
|(f)[yshift=-.1]| \bigoplus_{ijk} L^2(A_{ijk})\\
}; 
\foreach \source/\target/\pos in {a/b/above,d/e/above,f/d/left,f/e/below}
{\draw[->] (\source) -- node[\pos]{$\scriptstyle L^2_\mathrm{iso}$} (\target);}
\foreach \source/\target in {d/a,e/b}
{\draw[->] ($(\source.north) + (-.1,.05)$) arc (-180:0:.05) -- (\target);}
\draw[->] ($(b.east) + (.05,.1)$) arc (90:270:.05) -- (c);
\draw[->] ($(e.north) + (.9,.05)$) arc (110:290:.05) -- (c);}
\]
The upper right triangle is a diagram of inclusions and commutes for obvious reasons.
The upper left rectangle commutes by the functoriality of the $L^2$ construction (Theorem \ref{thm: functoriality of L2}) and
by the compatibility of polar decomposition with the operation of composing with an isometry.
Finally, note that whenever we have a subfactor inclusion $f:N\to M$
then, by equation \eqref{eq: norm of L2(iota)}, the corresponding map $L^2(f)$ is a scalar multiple of an isometry.
The commutativity of the bottom triangle thus holds because $A_{ijk}\hookrightarrow p_{ijk}Bp_{ijk}\hookrightarrow p_{ijk}Cp_{ijk}$ are subfactor inclusions.
\end{proof}

\subsection*{Functoriality of Connes fusion}

By construction, the operation of Connes fusion $(H_A,\, {}_AK)\mapsto H\boxtimes_AK$ is a functor in $H$ and $K$.
We now investigate in what sense it is a functor of the \emph{three} variables $H$, $A$, and $K$.
Consider the following category.
Its objects are triples $(H,A,K)$ consisting of a von Neumann algebra $A$ with finite-dimensional center, a right module $H$, and a left module $K$.
A morphism from $(H_1,A_1,K_1)$ to $(H_2,A_2,K_2)$ is a  triple $\alpha:A_1\to A_2$, $h:H_1\to H_2$, $k:K_1\to K_2$,
where $\alpha$ is a finite homomorphism, and $h$ and $k$ are $A_1$-linear maps.

\begin{theorem}\label{thm: functoriality of Connes fusion}
The assignment
\[
(H,\,A,\,K)\;\mapsto\; H\boxtimes_AK
\]
extends to a functor from the category described above to the category of Hilbert spaces and bounded linear maps.
\end{theorem}

\begin{proof}
Given a morphism $(h,\alpha,k):(H_1,A_1,K_1)\to(H_2,A_2,K_2)$ of the above category, 
we describe the induced map $h\boxtimes_\alpha k:H_1\boxtimes_{A_1}K_1\to H_2\boxtimes_{A_2}K_2$.
Recall that the composite \eqref{eq: It's L2(f)} provides an isomorphism $\Phi$ between the dual of the bimodule ${}_{A_1}L^2({A_2})_{A_2}$ and the bimodule ${}_{A_2} L^2({A_2})_{A_1}$.
Let
\[
\begin{split}
R\,\,&:\,\,{}_{A_1}L^2({A_1})_{A_1}\to {}_{A_1}L^2({A_2})\boxtimes_{A_2} \overline{L^2({A_2})}_{A_1}\xrightarrow{1\otimes \Phi} {}_{A_1}L^2({A_2})\boxtimes_{A_2}
L^2({A_2})_{A_1}\\
S\,\,&:\,\,{}_{A_2}L^2({A_2})_{A_2}\to {}_{A_2}\overline{L^2({A_2})}\boxtimes_{A_1} L^2({A_2})_{A_2}\xrightarrow{\Phi \otimes 1} {}_{A_2}L^2({A_2})\boxtimes_{A_1} L^2({A_2})_{A_2}
\end{split}
\]
denote the composition of the normalized duality maps \eqref{eq:duality maps} with the aforementioned isomorphism.

We define the image of an element 
\begin{equation*}
\phi_1\otimes\xi_1\otimes \psi_1 \in \hom_{A_1}\!\big(L^2A_1,H_1\big)\otimes L^2A_1\otimes \hom_{A_1}\!\big(L^2A_1,K_1\big)
\end{equation*}
under the map $h\boxtimes_\alpha k$ to be $\phi_2\otimes\xi_2\otimes \psi_2$,
where $\phi_2\in\hom_{A_2}(L^2A_2,H_2)$ and $\psi_2\in\hom_{A_2}(L^2A_2,K_2)$ are given by
\[
\begin{split}
\phi_2:\,L^2A_2\cong L^2A_1\underset{A_1}\boxtimes L^2A_2\xrightarrow{\phi_1\otimes 1} &H_1\underset{A_1}\boxtimes L^2A_2
\xrightarrow{h\otimes 1} H_2\underset{A_1}\boxtimes L^2A_2\\&\cong
H_2\underset{A_2}\boxtimes L^2A_2\underset{A_1}\boxtimes L^2A_2
\xrightarrow{1\otimes S^*\!\!} H_2\underset{A_2}\boxtimes L^2A_2\cong H_2\phantom{\,,}\end{split}
\]
\[
\begin{split}
\psi_2:\,L^2A_2\cong L^2A_2\underset{A_1}\boxtimes L^2A_1\xrightarrow{1\otimes\psi_1} &L^2A_2\underset{A_1}\boxtimes K_1
\xrightarrow{1\otimes k} L^2A_2\underset{A_1}\boxtimes K_2\\&\cong
L^2A_2\underset{A_1}\boxtimes L^2A_2\underset{A_2}\boxtimes K_2
\xrightarrow{S^*\!\otimes 1} L^2A_2\underset{A_2}\boxtimes K_2\cong K_2\,,
\end{split}
\]
and $\xi_2:=R(\xi_1)\in L^2(A_2)\boxtimes_{A_2} L^2(A_2) \cong L^2(A_2)$.
Note that $\xi_2=L^2(\alpha)(\xi_1)$ by diagram \eqref{eq: iso between dual of A--L2B--B and B--L2B--A}.
Graphically, the above map sends
\[
\def\colA{black!10}
\def\colB{black!30}
\tikzmath[scale=\squarescale]{
\coordinate	(lower left corner) at (-10,-16);
\coordinate	(upper right corner) at (2,16);
\coordinate	(upper left corner) at ($(upper right corner -| lower left corner) + (-5,0)$); \coordinate(lower right corner) at (upper right corner |- lower left corner);
\draw 		(-9,0) node (a) {$\phantom{O.}$};
\draw 		(a) node {$\phi_1$};
\fill[\colA] (lower left corner) -- (a.south east -| lower left corner)
-- (a.south east) -- (a.north east) -- (a.north east -| upper left corner) {[rounded corners=10] --  (upper left corner) -- (upper right corner) -- (lower right corner)} -- cycle;
\draw (lower left corner) node[below,yshift=2,xshift=1]{$\scriptstyle H_1$} -- (a.south east -| lower left corner)
(a.south east -| upper left corner) -- (a.south east) -- (a.north east) -- (a.north east -| upper left corner);
	}
\otimes
\tikzmath[scale=\squarescale]{
\coordinate	(lower left corner) at (-7,-16);
\coordinate	(upper right corner) at (7,16);
\coordinate	(upper left corner) at (upper right corner -| lower left corner); \coordinate(lower right corner) at (upper right corner |- lower left corner);
\useasboundingbox ($(lower left corner)-(0,5.4)$) rectangle (upper right corner);
\def\v{5}
\fill[\colA, rounded corners=10] (lower left corner) {[sharp corners] --  ($(upper left corner) - (0,\v)$) -- ($(upper right corner) - (0,\v)$)} -- (lower right corner) -- cycle;
\draw ($(upper left corner) - (0,\v)$) --node[above]{$\xi_1$} ($(upper right corner) - (0,\v)$);
	}
\otimes
\tikzmath[scale=\squarescale]{
\coordinate	(lower left corner) at (-2,-16);
\coordinate	(upper right corner) at (15,16);
\coordinate	(upper left corner) at (upper right corner -| lower left corner); \coordinate(lower right corner) at ($(upper right corner |- lower left corner)-(5,0)$);
\draw 		(9,0) node (a) {$\phantom{O.}$};
\draw 		(a) node {$\psi_1$};
\fill[\colA] (lower right corner) -- (a.south west -| lower right corner)
-- (a.south west) -- (a.north west) -- (a.north west -| upper right corner) {[rounded corners=10] --  (upper right corner) -- (upper left corner) -- (lower left corner)} -- cycle;
\draw (lower right corner)node[below,yshift=2,xshift=1]{$\scriptstyle K_1$} -- (a.south west -| lower right corner)(a.south west -| upper right corner) -- (a.south west) -- (a.north west) -- (a.north west -| upper right corner);
	}
\quad\,\,\text{to}\,\,\quad
\tikzmath[scale=\squarescale]{
\def \roundedsmall {8}
\coordinate	(lower left corner) at (-12,-16);
\coordinate	(upper right corner) at (6,16);
\coordinate	(lower right corner of middle line) at (0,-12);
\coordinate	(upper left corner) at ($(upper right corner -| lower left corner)-(3,0)$); \coordinate(lower right corner) at (upper right corner |- lower left corner);
\draw 		(-9,6) node (a) {$\phantom{O.}$};
\draw 		(a) node {$\phi_1$};
\draw 		(-9,-4) node (b) {$\phantom{O.}$};
\draw 		(-10,-4) node {$h$};
\fill[\colB] (lower left corner) -- 
(b.south east -| lower left corner) -- (b.south)[rounded corners=\roundedsmall] -- (b.south |- lower right corner of middle line) -- (lower right corner of middle line)[sharp corners] -- 
(lower right corner of middle line |- upper left corner)[rounded corners=10] --
(upper right corner) -- (lower right corner)[sharp corners] -- cycle;
\fill[\colA](b.south) -- (b.south east) -- (b.north east) -- ($(b.north east -| lower left corner)+(2,0)$) --
($(a.south east -| lower left corner)+(2,0)$) -- (a.south east) -- (a.north east) -- (a.north east -| upper left corner)[rounded corners=10] -- 
(upper left corner)[sharp corners] -- (lower right corner of middle line |- upper left corner)[rounded corners=\roundedsmall]
-- (lower right corner of middle line) -- (b.south |- lower right corner of middle line)[sharp corners] -- cycle;
\draw (lower left corner)node[below,yshift=2,xshift=1]{$\scriptstyle H_2$} -- 
(b.south east -| lower left corner)(b.south east -| upper left corner) -- (b.south east) -- (b.north east) -- (b.north east -| upper left corner)
($(b.north east -| lower left corner)+(2,0)$) -- 
($(a.south east -| lower left corner)+(2,0)$)(a.south east -| upper left corner) -- (a.south east) -- (a.north east) -- (a.north east -| upper left corner);
\draw[rounded corners=\roundedsmall] (b.south) -- (b.south |- lower right corner of middle line) -- (lower right corner of middle line) -- 
(lower right corner of middle line |- upper left corner);
	}
\otimes
\tikzmath[scale=\squarescale]{
\coordinate	(lower left corner) at (-10,-16);
\coordinate	(upper right corner) at (10,16);
\coordinate	(upper left corner) at (upper right corner -| lower left corner); \coordinate(lower right corner) at (upper right corner |- lower left corner);
\useasboundingbox ($(lower left corner)-(0,5.3)$) rectangle (upper right corner);
\def\v{5}
\fill[\colA, rounded corners=10] (lower left corner) {[sharp corners] --  ($(upper left corner) - (0,\v)$) -- ($(upper right corner) - (0,\v)$)} -- (lower right corner) -- cycle;
\draw ($(upper left corner) - (0,\v)$) --node[above]{$\xi_1$} ($(upper right corner) - (0,\v)$);
\filldraw[fill=\colB, rounded corners=8] (-4,-16) -- (-4,0) -- (4,0) -- (4,-16);
	}
\otimes
\tikzmath[scale=\squarescale]{
\def \roundedsmall {8}
\coordinate	(lower left corner) at (-6,-16);
\coordinate	(upper right corner) at (15,16);
\coordinate	(lower left corner of middle line) at (0,-12);
\coordinate	(upper left corner) at (upper right corner -| lower left corner); \coordinate(lower right corner) at ($(upper right corner |- lower left corner)-(3,0)$);
\draw 		(9,6) node (a) {$\phantom{O.}$};
\draw 		(a) node {$\psi_1$};
\draw 		(9,-4) node (b) {$\phantom{O.}$};
\draw 		(10,-4) node {$k$};
\fill[\colB] (lower right corner) -- 
(b.south west -| lower right corner) -- (b.south)[rounded corners=\roundedsmall] -- (b.south |- lower left corner of middle line) -- (lower left corner of middle line)[sharp corners] -- 
(lower left corner of middle line |- upper right corner)[rounded corners=10] --
(upper left corner) -- (lower left corner)[sharp corners] -- cycle;
\fill[\colA](b.south) -- (b.south west) -- (b.north west) -- ($(b.north west -| lower right corner)-(2,0)$) --
($(a.south west -| lower right corner)-(2,0)$) -- (a.south west) -- (a.north west) -- (a.north west -| upper right corner)
[rounded corners=10] -- 
(upper right corner)[sharp corners] -- (lower left corner of middle line |- upper right corner)[rounded corners=\roundedsmall]
-- (lower left corner of middle line) -- (b.south |- lower left corner of middle line)[sharp corners] -- cycle;
\draw (lower right corner)node[below,yshift=2,xshift=1]{$\scriptstyle K_2$} --
(b.south west -| lower right corner)(b.south west -| upper right corner) -- (b.south west) -- (b.north west) -- (b.north west -| upper right corner)($(b.north west -| lower right corner)-(2,0)$) -- 
($(a.south west -| lower right corner)-(2,0)$)(a.south west -| upper right corner) -- (a.south west) -- (a.north west) -- (a.north west -| upper right corner);
\draw[rounded corners=\roundedsmall] (b.south) -- (b.south |- lower left corner of middle line) -- (lower left corner of middle line) -- 
(lower left corner of middle line |- upper right corner);
	}\,,
\]
and is therefore given by
\begin{equation}\label{eq: h boxtimes k in pictures}
\def\colA{black!10}
\def\colB{black!30}
h\boxtimes_\alpha k\,\,:\,\,\tikzmath[scale=\squarescale]{
\coordinate	(lower left corner) at (-9,-16);
\coordinate	(upper right corner) at (14,18);
\coordinate	(upper left corner) at ($(upper right corner -| lower left corner) + (-5,0)$); 
\coordinate(lower right corner) at ($(upper right corner |- lower left corner)-(5,0)$);
\draw 		(-8,0) node (a) {$\phantom{O.}$};
\draw 		(a) node {$\phi_1$};
\draw 		(8,0) node (b) {$\phantom{O.}$};
\draw 		(b) node {$\psi_1$};
\def\v{5}
\fill[\colA] (lower left corner) -- (a.south east -| lower left corner)
-- (a.south east) -- (a.north east) -- (a.north east -| upper left corner) 
--  ($(upper left corner) - (0,\v)$) -- ($(upper right corner) - (0,\v)$) 
-- (b.north west -| upper right corner) 
-- (b.north west) -- (b.south west) -- (b.south west -| lower right corner) -- (lower right corner) -- cycle;
\draw (lower left corner) node[below,yshift=2,xshift=1]{$\scriptstyle H_1$} -- (a.south east -| lower left corner)
(a.south east -| upper left corner) -- (a.south east) -- (a.north east) -- (a.north east -| upper left corner);
\draw (lower right corner)node[below,yshift=2,xshift=1]{$\scriptstyle K_1$} -- (b.south west -| lower right corner)(b.south west -| upper right corner) -- (b.south west) -- (b.north west) -- (b.north west -| upper right corner);
\draw ($(upper left corner) - (0,\v)$) --node[above]{$\xi_1$} ($(upper right corner) - (0,\v)$);
	}
\,\,\,\,\mapsto\,\,\,\,
\tikzmath[scale=\squarescale]{
\def \roundedsmall {8}
\coordinate	(lower left corner) at (-15,-16);
\coordinate	(upper right corner) at (18,18);
\coordinate	(center of middle line) at (0,-13);
\coordinate	(upper left corner) at ($(upper right corner -| lower left corner)-(3,0)$);
\coordinate	(lower right corner) at ($(upper right corner |- lower left corner)-(3,0)$);
\draw 		(-12,5) node (a) {$\phantom{O.}$};
\draw 		(a) node {$\phi_1$};
\draw 		(-12,-5) node (b) {$\phantom{O.}$};
\draw 		(-13,-5) node {$h$};
\draw 		(12,5) node (c) {$\phantom{O.}$};
\draw 		(12,-5) node (d) {$\phantom{O.}$};
\draw 		(13,-5) node {$k$};
\draw 		(c) node {$\psi_1$};
\def\v{5}

\fill[\colB] (lower left corner) -- (b.south east -| lower left corner) -- (b.south)[rounded corners=\roundedsmall] -- (b.south |- center of middle line) --
($(center of middle line) + (-4,0)$) --
($(center of middle line) + (-4,24)$) --
($(center of middle line) + (4,24)$) --
($(center of middle line) + (4,0)$) --
(d.south |- center of middle line) [sharp corners] -- (d.south) -- 
(d.south west -| lower right corner) --
(lower right corner) -- cycle; 
\fill[\colA](b.south) -- (b.south east) -- (b.north east) -- ($(b.north east -| lower left corner)+(2,0)$) --
($(a.south east -| lower left corner)+(2,0)$) -- (a.south east) -- (a.north east) -- (a.north east -| upper left corner)
--  ($(upper left corner) - (0,\v)$) -- ($(upper right corner) - (0,\v)$)
-- (c.north west -| upper right corner) -- (c.north west)  -- (c.south west) --
($(c.south west -| lower right corner)-(2,0)$) -- ($(d.north west -| lower right corner)-(2,0)$)
-- (d.north west) -- (d.south west) -- (d.south)
[rounded corners=\roundedsmall] -- (d.south |- center of middle line) -- 
($(center of middle line) + (4,0)$) --
($(center of middle line) + (4,24)$) --
($(center of middle line) + (-4,24)$) --
($(center of middle line) + (-4,0)$) --
(b.south |- center of middle line)[sharp corners] -- cycle;
\draw[rounded corners=\roundedsmall] (b.south) -- (b.south |- center of middle line) --
($(center of middle line) + (-4,0)$) --
($(center of middle line) + (-4,24)$) --
($(center of middle line) + (4,24)$) --
($(center of middle line) + (4,0)$)
-- (d.south |- center of middle line) [sharp corners] -- (d.south);
\draw ($(upper left corner) - (0,\v)$) --node[above]{$\xi_1$} ($(upper right corner) - (0,\v)$);
\draw (lower left corner)node[below,yshift=2,xshift=1]{$\scriptstyle H_2$} -- 
(b.south east -| lower left corner)
(b.south east -| upper left corner) -- (b.south east) -- (b.north east) -- (b.north east -| upper left corner)
($(b.north east -| lower left corner)+(2,0)$) --
($(a.south east -| lower left corner)+(2,0)$)(a.south east -| upper left corner) -- (a.south east) -- (a.north east) -- (a.north east -| upper left corner);
\draw (lower right corner)node[below,yshift=2,xshift=1]{$\scriptstyle K_2$} --
(d.south west -| lower right corner)(d.south west -| upper right corner) -- (d.south west) -- (d.north west) -- (d.north west -| upper right corner)($(d.north west -| lower right corner)-(2,0)$) -- 
($(c.south west -| lower right corner)-(2,0)$)(c.south west -| upper right corner) -- (c.south west) -- (c.north west) -- (c.north west -| upper right corner);
	}
\,=\,
\tikzmath[scale=\squarescale]{
\def \roundedsmall {11}
\coordinate	(lower left corner) at (-11,-16);
\coordinate	(upper right corner) at (14,18);
\coordinate	(center of middle line) at (0,-13);
\coordinate	(upper left corner) at ($(upper right corner -| lower left corner)-(3,0)$);
\coordinate	(lower right corner) at ($(upper right corner |- lower left corner)-(3,0)$);
\draw 		(-8,6) node (a) {$\phantom{O.}$};
\draw 		(a) node {$\phi_1$};
\draw 		(-8,-4) node (b) {$\phantom{O.}$};
\draw 		(-9,-4) node {$h$};
\draw 		(8,6) node (c) {$\phantom{O.}$};
\draw 		(c) node {$\psi_1$};
\draw 		(8,-4) node (d) {$\phantom{O.}$};
\draw 		(9,-4) node {$k$};
\def\v{5}

\fill[\colB] (lower left corner) -- (b.south east -| lower left corner) -- (b.south)[rounded corners=\roundedsmall] -- (b.south |- center of middle line)
 -- (d.south |- center of middle line) [sharp corners] -- (d.south) -- 
(d.south west -| lower right corner) --
(lower right corner) -- cycle; 
\fill[\colA](b.south) -- (b.south east) -- (b.north east) -- ($(b.north east -| lower left corner)+(2,0)$) --
($(a.south east -| lower left corner)+(2,0)$) -- (a.south east) -- (a.north east) -- (a.north east -| upper left corner)
--  ($(upper left corner) - (0,\v)$) -- ($(upper right corner) - (0,\v)$)
-- (c.north west -| upper right corner) -- (c.north west)  -- (c.south west) --
($(c.south west -| lower right corner)-(2,0)$) -- ($(d.north west -| lower right corner)-(2,0)$)
-- (d.north west) -- (d.south west) -- (d.south)
[rounded corners=\roundedsmall] -- (d.south |- center of middle line) -- (b.south |- center of middle line)[sharp corners] -- cycle;
\draw[rounded corners=\roundedsmall] (b.south) -- (b.south |- center of middle line) -- (d.south |- center of middle line) [sharp corners] -- (d.south);
\draw ($(upper left corner) - (0,\v)$) --node[above]{$\xi_1$} ($(upper right corner) - (0,\v)$);
\draw (lower left corner)node[below,yshift=2,xshift=1]{$\scriptstyle H_2$} -- 
(b.south east -| lower left corner)
(b.south east -| upper left corner) -- (b.south east) -- (b.north east) -- (b.north east -| upper left corner)
($(b.north east -| lower left corner)+(2,0)$) -- 
($(a.south east -| lower left corner)+(2,0)$)(a.south east -| upper left corner) -- (a.south east) -- (a.north east) -- (a.north east -| upper left corner);
\draw (lower right corner)node[below,yshift=2,xshift=1]{$\scriptstyle K_2$} --
(d.south west -| lower right corner)(d.south west -| upper right corner) -- (d.south west) -- (d.north west) -- (d.north west -| upper right corner)($(d.north west -| lower right corner)-(2,0)$) -- 
($(c.south west -| lower right corner)-(2,0)$)(c.south west -| upper right corner) -- (c.south west) -- (c.north west) -- (c.north west -| upper right corner);
	}\,.
\end{equation}
Here, the two shades correspond to the algebras $A_1$ and $A_2$, the unlabeled line between those shades corresponds to the bimodule 
${}_{A_1}L^2(A_2)_{A_2}$ and its dual bimodule ${}_{A_2}L^2(A_2)_{A_1}$, and the isomorphism \eqref{eq: It's L2(f)} has been suppressed from the notation.
Abstracting out $\xi_1$, $\phi_1$, $\psi_1$ from \eqref{eq: h boxtimes k in pictures}, we can rewrite $h\boxtimes_\alpha k$ in a more concise form, as
\[
\def\colA{black!10}
\def\colB{black!30}
h\boxtimes_\alpha k\quad=\quad
\tikzmath[scale=\squarescale]{
\coordinate	(lower left corner) at (-8,-12);
\coordinate	(upper right corner) at (8,10);
\coordinate	(upper left corner) at (upper right corner -| lower left corner); 
\coordinate(lower right corner) at (upper right corner |- lower left corner);
\draw 		(-8,0) node[inner xsep=6] (a) {$h$};
\draw 		(8,0) node[inner xsep=6] (b) {$k$};
\def\v{5}
\fill[\colA] (lower left corner) -- (a.south east -| lower left corner)
-- (a.south east) -- (a.north east) -- (a.north east -| upper left corner) 
--  (upper left corner) -- (upper right corner) 
-- (b.north west -| upper right corner) 
-- (b.north west) -- (b.south west) -- (b.south west -| lower right corner) -- (lower right corner) -- cycle;
\fill[\colB] (lower left corner) -- (a.south east -| lower left corner)
-- ($(a.south east)-(2,0)$) [rounded corners = 10]-- ($(a.south east)-(2,5)$) -- ($(b.south west)+(2,-5)$) [sharp corners]-- ($(b.south west)+(2,0)$) 
-- (b.south west -| lower right corner) -- (lower right corner) -- cycle;
\draw ($(a.south east)-(2,0)$) [rounded corners = 10]-- ($(a.south east)-(2,5)$) -- ($(b.south west)+(2,-5)$) [sharp corners]-- ($(b.south west)+(2,0)$);
\draw (lower left corner) node[below,yshift=2,xshift=1]{$\scriptstyle H_2$} -- (a.south east -| lower left corner)
($(a.south east -| lower left corner)-(5,0)$) -- (a.south east) -- (a.north east) -- ($(a.north east -| upper left corner)-(5,0)$)
(a.north east -| upper left corner) -- (upper left corner) node[above,yshift=-2,xshift=1]{$\scriptstyle H_1$};
\draw (lower right corner)node[below,yshift=2,xshift=1]{$\scriptstyle K_2$} -- (b.south west -| lower right corner)
($(b.south west -| upper right corner)+(5,0)$) -- (b.south west) -- (b.north west) -- ($(b.north west -| upper right corner)+(5,0)$)
(b.north west -| upper right corner) -- (upper right corner) node[above,yshift=-2,xshift=1]{$\scriptstyle K_1$};
	}\,.
\]
The latter description also makes it clear that $h\boxtimes_\alpha k$ is bounded.
Compatibility with composition follows from Lemma \ref{lem: product of dualizables}.
\end{proof}

We record the following lemma for future use.
Once again, we make implicit use of the identification \eqref{eq: It's L2(f)} and of its basic property \eqref{eq: iso between dual of A--L2B--B and B--L2B--A}.

\begin{lemma}
Let $f:A\to B$ be a finite map between von Neumann algebras with finite-dimensional center.
Then the map $B\to \hom(L^2A_A,L^2B_A)$ given by
\[
\def\colA{black!10}
\def\colB{black!30}
b\,\,\mapsto\,\, (b\otimes 1)L^2(f) \,=\, 
\tikzmath[scale=.06]{
\def \h {4.8}
\def \midcurv {6}
\coordinate	(lower left corner) at (-8,-8);
\coordinate	(upper right corner) at (9,11);
\coordinate	(lower right corner) at (upper right corner |- lower left corner);
\coordinate	(upper left corner) at (lower left corner |- upper right corner);
\clip[rounded corners=9] (upper left corner) rectangle (lower right corner);
\draw 		(-3.1,0) node[inner xsep=3.5, inner ysep=2.1] (a) {$b$};
\coordinate	(middle right) at ($(a.north east)+(3.8,3.5)$);
\fill[\colB] ($(a.north east -| upper left corner)+(\h,0)$) -- (a.north east) -- (a.south east) -- (a.south east -| lower left corner) --  (lower left corner) -- (lower left corner -| middle right) {[rounded corners=\midcurv] -- (middle right) -- ($(middle right -| upper left corner)+(\h,0)$)} -- cycle;
\fill[\colA] (upper left corner) -- (upper right corner) -- (lower right corner) [sharp corners]-- (lower left corner -| middle right) {[rounded corners=\midcurv] -- (middle right) -- ($(middle right -| upper left corner)+(\h,0)$)}
-- ($(a.north east -| upper left corner)+(\h,0)$) -- (a.north east -| upper left corner) -- cycle; 
\draw {[rounded corners=\midcurv](lower left corner -| middle right) -- (middle right) -- ($(middle right -| upper left corner)+(\h,0)$) -- ($(a.north east -| upper left corner)+(\h,0)$)}
(a.north east -| lower left corner) -- (a.north east) -- (a.south east) -- (a.south east -| lower left corner);
} 
\]
is an isomorphism.
\end{lemma}
\begin{proof}
The inverse map is\,
$
\def\colA{black!10}
\def\colB{black!30}
\def \h {7}
\tikzmath[scale=.06]{
\coordinate	(upper left corner) at (-8,9);
\coordinate	(lower right corner) at (8,-9);
\coordinate	(upper right corner) at (lower right corner |- upper left corner);
\coordinate	(lower left corner) at (upper left corner |- lower right corner);
\clip[rounded corners=9] (upper left corner) rectangle (lower right corner);
\draw 		(-1,0) node (a) {$x$};
\fill[\colA] ($(lower left corner)+(\h,0)$) -- ($(a.south east -| lower left corner)+(\h,0)$)
-- (a.south east) -- (a.north east) -- (a.north east -| upper left corner) --  (upper left corner) -- (upper right corner) -- (lower right corner) -- cycle;
\fill[\colB] (lower left corner) -- ($(lower left corner)+(\h,0)$) -- ($(a.south east -| lower left corner)+(\h,0)$) -- (a.south east -| lower left corner) -- cycle; 
\draw ($(lower left corner)+(\h,0)$) -- ($(a.south east -| lower left corner)+(\h,0)$)
(a.south east -| upper left corner) -- (a.south east) -- (a.north east) -- (a.north east -| upper left corner);
} 
\,\mapsto\,
\tikzmath[scale=.06]{
\def \h {4.4}
\def \midcurv {6}
\coordinate	(upper left corner) at (-8,8);
\coordinate	(lower right corner) at (8,-10);
\coordinate	(upper right corner) at (lower right corner |- upper left corner);
\coordinate	(lower left corner) at (upper left corner |- lower right corner);
\clip[rounded corners=9] (upper left corner) rectangle (lower right corner);
\draw 		(-3.2,0) node (a) {$x$};
\coordinate	(middle right) at ($(a.south east)+(3.2,-3.5)$);
\fill[\colA] ($(a.south east -| lower left corner)+(\h,0)$) -- (a.south east) -- (a.north east) -- (a.north east -| upper left corner) --  (upper left corner) -- (upper left corner -| middle right) {[rounded corners=\midcurv] -- (middle right)
-- ($(middle right -| lower left corner)+(\h,0)$)} -- cycle;
\fill[\colB] (lower left corner) -- (lower right corner) -- (upper right corner) [sharp corners]-- (upper left corner -| middle right) {[rounded corners=\midcurv] -- (middle right) -- ($(middle right -| lower left corner)+(\h,0)$)}
-- ($(a.south east -| lower left corner)+(\h,0)$) -- (a.south east -| lower left corner) [rounded corners=9] -- cycle; 
\draw {[rounded corners=\midcurv](upper left corner -| middle right) -- (middle right) -- ($(middle right -| lower left corner)+(\h,0)$) -- ($(a.south east -| lower left corner)+(\h,0)$)}
(a.south east -| upper left corner) -- (a.south east) -- (a.north east) -- (a.north east -| upper left corner);
} 
\,\in\, \hom(L^2B_B,L^2B_B) \cong B$.
\end{proof}

\section{Index via conditional expectations}\label{sec: Pimsner-Popa inequality}

In this section, we recall the work of Pimsner and Popa on conditional expectations, and use it to establish the equivalence 
between the definition of index via statistical dimension (Definitions \ref{def: statistical dimension} and \ref{def: min index}) and other
notions of index that exist in the literature \cite{Hiai(Minimizing-indices), Kosaki(Extension-of-Jones-index-to-arbitrary-factors), Kosaki(Type-III-factors-and-index-theory), Longo(Index-of-subfactors-and-statistics-of-quantum-fields-I), Pimsner-Popa}.
The basic inequality \eqref{eq:Pimsner-Popa inequality} was introduced in 
\cite{Pimsner-Popa} for type \emph{II} von Neumann algebras,
and later in \cite{Kosaki-Longo(A-remark-on-the-minimal-index), Longo(Index-of-subfactors-and-statistics-of-quantum-fields-I), Longo(Index-of-subfactors-and-statistics-of-quantum-fields-II)} 
for arbitrary von Neumann algebras.
Further references include \cite[section 1.1]{Popa(Classification-of-subfactors-and-their-endomorphisms)} and \cite[section 3.4]{Kosaki(Type-III-factors-and-index-theory)}.

Given a subfactor $N\subset M$ \footnote{Note that in this section, we usually (but not always) use the letters $N$ and $M$ to refer to factors, as is traditional, and use the letters $A$ and $B$ to refer to more general von Neumann algebras.}, a completely positive normal map $E:M\to N$ is called a \emph{conditional expectation} if $E(1)=1$ and $E(axb)=aE(x)b$ for all $x\in M$ and $a,b\in N$.  It may happen that, for some $\lambda$, the conditional expectation satisfies the \emph{Pimsner--Popa inequality}: 
\[ E(x) \ge \lambda^{-1}x,\quad\forall x\in M_+\, . \]
Following \cite{Longo(Index-of-subfactors-and-statistics-of-quantum-fields-I)}, the index of the conditional expectation is the smallest possible such~$\lambda$:
\begin{equation}\label{Ind(E)}
\Ind(E):=\inf\big\{\lambda\,\big|\,E(x)\ge\lambda^{-1}x,\,\,\forall x\in M_+\big\}.
\end{equation}
We call a conditional expectation finite if its index is finite.
For subfactors admitting finite conditional expectations, 
Longo proves \cite[Theorem 5.5]{Longo(Index-of-subfactors-and-statistics-of-quantum-fields-I)} that there exists a unique conditional expectation minimizing $\Ind(E)$ --- see also \cite{Hiai(Minimizing-indices), Kosaki-Longo(A-remark-on-the-minimal-index)}.
For a general subfactor, he defines the \emph{minimal index} to be
\begin{equation}\label{eq: [M:N] = inf_E IndE}
\Ind(N,M) := \inf_E\, \Ind(E) = \inf_E\,\, \inf\big\{\lambda\,\big|\,E(x)\ge\lambda^{-1}x,\,\,\forall x\in M_+\big\},
\end{equation}
where the infimum runs over all conditional expectations $E:M\to N$.
We will show later, in Corollary~\ref{cor: justification of the name minimal}, that the index (Definition \ref{def: min index})
agrees with the minimal index if $N$ and $M$ are infinite-dimensional --- see Warning~\ref{warn: Longo index is bad}.

If the subfactor has finite  index, then an example of a conditional expectation is given by $[M: N]^{-\frac12}$ times the map \eqref{eq: minimal conditional expectation}:
\[
\def\colA{black!10}
\def\colB{black!30}
E_0(b)\;\;:=\;\;[M: N]^{-\frac12}\cdot\,
\tikzmath[scale=\squarescale]
	{\fill[rounded corners=10, fill=\colA] (-14,-15) rectangle (13,15);
	\draw[rounded corners=10, fill=\colB] (-5,-9) rectangle (5,9);
	\coordinate (x) at (-14,0); \coordinate (x') at (-14.1,0);
	\draw 	(-5,0) node(a)[fill=white]{$b$};
	\fill[fill=white] (a.north east -| x') -- (a.north east) -- (a.south east) -- (a.south east -| x');
	\draw (a.north east -| x) -- (a.north east) -- (a.south east) -- (a.south east -| x);
	\node(a) at (-8,0) [inner sep=0]{$b$};
	}\,\,
=\,\,
\big(\tikzmath[scale=.085]{ \fill[fill=\colA, rounded corners=3.2] (-3,-2.5) rectangle (3,2.5); \filldraw[fill=\colB] (0,0) circle (1.4);} 
\big)^{-1}\!\cdot\,
\tikzmath[scale=\squarescale]
	{\fill[rounded corners=10, fill=\colA] (-14,-15) rectangle (13,15);
	\draw[rounded corners=10, fill=\colB] (-5,-9) rectangle (5,9);
	\coordinate (x) at (-14,0); \coordinate (x') at (-14.1,0);
	\draw 	(-5,0) node(a)[fill=white]{$b$};
	\fill[fill=white] (a.north east -| x') -- (a.north east) -- (a.south east) -- (a.south east -| x');
	\draw (a.north east -| x) -- (a.north east) -- (a.south east) -- (a.south east -| x);
	\node(a) at (-8,0) [inner sep=0]{$b$};
	}\,.
\] 
We call $E_0$ the \emph{minimal conditional expectation}.
We will show later, in Proposition~\ref{lem: Extremality properties of E_0}, that the minimal conditional expectation minimizes $\Ind(E)$, thus justifying its name.

We begin by observing that the index of a subfactor provides an upper bound on the 
index of the minimal conditional expectation:
\begin{proposition}
The minimal conditional expectation $E_0$ satisfies the inequality
\begin{equation}\label{eq:Pimsner-Popa inequality}
\qquad E_0(x) \ge [M: N]^{-1} x\quad\quad\forall x\in M_+.
\end{equation}
In other words, $\Ind(E_0)\le[M:N]$.
\end{proposition}
\begin{proof}
Let $x$ be a positive element of $M$, and let us write $d:=[M: N]^{\frac12}$ for the statistical dimension of ${}_NL^2M_M$.
Because the map
\(
\def\colA{black!10}
\def\colB{black!30}
d^{-1}\,\tikzmath[scale=\squarescale]{
\clip[rounded corners=2] (-4,-3) rectangle (4,3);
\fill[fill=\colB] 
(-4,-3) [rounded corners=2]-- (-4,3) [sharp corners]-- (-2,3) [rounded corners = 4]-- (-2,1) -- (2,1) [sharp corners]-- (2,3) [rounded corners=2]--
(4,3) -- (4,-3) [sharp corners]-- (2,-3) [rounded corners = 4]-- (2,-1) -- (-2,-1) [sharp corners]-- (-2,-3) [rounded corners=2]-- cycle;
\fill[fill=\colA] (-2,-3) -- (2,-3) [rounded corners = 4]-- (2,-1) -- (-2,-1) [sharp corners]-- cycle;
\fill[fill=\colA] (-2,3) -- (2,3) [rounded corners = 4]-- (2,1) -- (-2,1) [sharp corners]-- cycle;
\draw (2,3) [rounded corners = 4]-- (2,1) -- (-2,1) -- (-2,3);
\draw (2,-3) [rounded corners = 4]-- (2,-1) -- (-2,-1) -- (-2,-3);
}
\)
is a projection, we have
\(
\def\colA{black!10}
\def\colB{black!30}
d^{-1}\,\tikzmath[scale=\squarescale]{
\clip[rounded corners=2] (-4,-3) rectangle (4,3);
\fill[fill=\colB] 
(-4,-3) [rounded corners=2]-- (-4,3) [sharp corners]-- (-2,3) [rounded corners = 4]-- (-2,1) -- (2,1) [sharp corners]-- (2,3) [rounded corners=2]--
(4,3) -- (4,-3) [sharp corners]-- (2,-3) [rounded corners = 4]-- (2,-1) -- (-2,-1) [sharp corners]-- (-2,-3) [rounded corners=2]-- cycle;
\fill[fill=\colA] (-2,-3) -- (2,-3) [rounded corners = 4]-- (2,-1) -- (-2,-1) [sharp corners]-- cycle;
\fill[fill=\colA] (-2,3) -- (2,3) [rounded corners = 4]-- (2,1) -- (-2,1) [sharp corners]-- cycle;
\draw (2,3) [rounded corners = 4]-- (2,1) -- (-2,1) -- (-2,3);
\draw (2,-3) [rounded corners = 4]-- (2,-1) -- (-2,-1) -- (-2,-3);
}
\,\leq\,
\tikzmath[scale=\squarescale]{
\clip[rounded corners=2] (-4,-3) rectangle (4,3);
\fill[fill=\colB] (-4,-3) [rounded corners=2]-- (-4,3) [sharp corners]-- (-1.8,3) -- (-1.8,-3) [rounded corners=2]-- cycle;
\fill[fill=\colB] (4,-3) [rounded corners=2]-- (4,3) [sharp corners]-- (1.8,3) -- (1.8,-3) [rounded corners=2]-- cycle;
\fill[fill=\colA] (-1.8,-3) rectangle (1.8,3);
\draw (1.8,3) -- (1.8,-3)(-1.8,3) -- (-1.8,-3);
}
\)\,. 
As a consequence of the general fact $(a\le b) \Rightarrow (yay^*\le yby^*)$, it follows that
\[
\def\colA{black!10}
\def\colB{black!30}
d^{-1}\,\tikzmath[scale=\squarescale]{
\coordinate (x) at (-9,0);
\fill[rounded corners=10, fill=\colA] (3,-20) -- (-9,-20) -- (-9,20) [sharp corners]-- (3,20) -- cycle;
\fill[rounded corners=10, fill=\colB] (3,-20) -- (10,-20) -- (10,20) [sharp corners]-- (3,20) -- cycle;
                                             \draw (3,-20) -- (3,20);
\draw (-2,0) node(a) [fill=white, inner sep =5]{\hspace{.3cm} $x$\hspace{.1cm}};
\draw (a.south west -| x) -- (a.south east) -- (a.north east) -- (a.north west -| x);
} 
\,\,=\,\,
d^{-1}\,\tikzmath[scale=\squarescale]{
\coordinate (x) at (-12,0);
\fill[rounded corners=10, fill=\colA] 
(7,20) -- (-12,20) -- (-12,-20) [sharp corners]-- (7,-20) [rounded corners=7]-- (7,-3) -- (0,-3) -- (0,-15) -- (-7,-15) -- (-7,15) -- (0,15) -- (0,3) -- (7,3) [sharp corners]-- cycle;
\fill[rounded corners=10, fill=\colB] 
(7,20) -- (12,20) -- (12,-20) [sharp corners]--    (7,-20) [rounded corners=7]-- (7,-3) -- (0,-3) -- (0,-15) -- (-7,-15) -- (-7,15) -- (0,15) -- (0,3) -- (7,3) [sharp corners]-- cycle;
                                                                \draw[rounded corners=7] (7,-20) -- (7,-3) -- (0,-3) -- (0,-15) -- (-7,-15) -- (-7,15) -- (0,15) -- (0,3) -- (7,3) -- (7,20);
\draw (-8.5,8) node(a)[fill=white, inner sep =2]{$\sqrt x$};
\draw (-8.5,-8) node(b)[fill=white, inner sep =2]{$\sqrt x$};
\draw (a.south west -| x) -- (a.south east) -- (a.north east) -- (a.north west -| x);
\draw (b.south west -| x) -- (b.south east) -- (b.north east) -- (b.north west -| x);
} 
\,\,=\,\,
\tikzmath[scale=\squarescale]{
\coordinate (x) at (-12,0);
\fill[rounded corners=10, fill=\colA] 
(7,20) -- (-12,20) -- (-12,-20) [sharp corners]-- (7,-20) [rounded corners=7]-- (7,-2) -- (0,-2) -- (0,-17) -- (-7,-17) -- (-7,17) -- (0,17) -- (0,2) -- (7,2) [sharp corners]-- cycle;
\fill[rounded corners=10, fill=\colB] 
(7,20) -- (12,20) -- (12,-20) [sharp corners]--    (7,-20) [rounded corners=7]-- (7,-2) -- (0,-2) -- (0,-17) -- (-7,-17) -- (-7,17) -- (0,17) -- (0,2) -- (7,2) [sharp corners]-- cycle;
                                                                \draw[rounded corners=7] (7,-20) -- (7,-2) -- (0,-2) -- (0,-17) -- (-7,-17) -- (-7,17) -- (0,17) -- (0,2) -- (7,2) -- (7,20);
\draw (-9,10.5) node(a)[fill=white, inner sep =2]{$\sqrt x$};
\draw (-9,-10.5) node(b)[fill=white, inner sep =2]{$\sqrt x$};
\draw (a.south west -| x) -- (a.south east) -- (a.north east) -- (a.north west -| x);
\draw (b.south west -| x) -- (b.south east) -- (b.north east) -- (b.north west -| x);
\draw node at (-2,0) {$\scriptstyle d^{-1}$};
\draw[dashed] (-6,-5) rectangle (9,5);
} 
\,\,\le\,\,
\tikzmath[scale=\squarescale]{
\coordinate (x) at (-12,0);
\fill[rounded corners=10, fill=\colA] (6,-20) -- (-12,-20) -- (-12,20) [sharp corners]-- (6,20) -- cycle;
\fill[rounded corners=10, fill=\colB] (6,-20) -- (12,-20) -- (12,20) [sharp corners]-- (6,20) -- cycle;
                                             \draw (6,-20) -- (6,20);
\filldraw[rounded corners=7, fill=\colB] (-7,-17) rectangle (0,17);
\draw (-9,10) node(a)[fill=white, inner sep =2]{$\sqrt x$};
\draw (-9,-10) node(b)[fill=white, inner sep =2]{$\sqrt x$};
\draw (a.south west -| x) -- (a.south east) -- (a.north east) -- (a.north west -| x);
\draw (b.south west -| x) -- (b.south east) -- (b.north east) -- (b.north west -| x);
\draw[dashed] (-3,-4.5) rectangle (9,4.5);
} 
\,\,=\,\,
\tikzmath[scale=\squarescale]{
\coordinate (x) at (-12,0);
\fill[rounded corners=10, fill=\colA] (6,-20) -- (-12,-20) -- (-12,20) [sharp corners]-- (6,20) -- cycle;
\fill[rounded corners=10, fill=\colB] (6,-20) -- (12,-20) -- (12,20) [sharp corners]-- (6,20) -- cycle;
                                             \draw (6,-20) -- (6,20);
\filldraw[rounded corners=7, fill=\colB] (-7,-12) rectangle (0,12);
\draw (-9,0) node(a) [fill=white, inner sep =5]{\,$x$};
\draw (a.south west -| x) -- (a.south east) -- (a.north east) -- (a.north west -| x);
}\,. 
\]
Now multiply both sides by $d^{-1}$ to get the desired inequality.
\end{proof}

The following proposition establishes the connection between 
the Pimsner--Popa inequality and dualizability.

\begin{proposition}\label{Prop. 8.3}
Let $A\subset B$ be von Neumann algebras with finite-dimensional centers, and let $E:B\to A$ be a conditional expectation.
If there exists a constant $\mu>0$ such that $E(x) \ge \mu x$ for all $x\in B_+$, then ${}_A L^2B_B$ is a dualizable bimodule.
\end{proposition}

\begin{proof}
We show that ${}_B L^2B_A$ is the dual of ${}_A L^2B_B$.
To do so, we construct maps
\begin{equation}\label{eq: construction of R and S}
\def\colA{black!10}
\def\colB{black!30}
\begin{split}
\qquad R\,=\,\tikzmath[scale=.1]{
\fill[fill=\colA] 
(-4,-2.5) [rounded corners=4]-- (-4,3) --
(4,3) -- (4,-2.5) [sharp corners]-- (2,-2.5) [rounded corners = 5.5]-- (2,.5) -- (-2,.5) [sharp corners]-- (-2,-2.5) [rounded corners=4]-- cycle;
\fill[fill=\colB] (-2,-2.5) -- (2,-2.5) [rounded corners = 5.5]-- (2,.5) -- (-2,.5) [sharp corners]-- cycle;
\draw (2,-2.5) [rounded corners = 5.5]-- (2,.5) -- (-2,.5) -- (-2,-2.5);}
\,\,&:\,\,\,{}_AL^2(A)_A\,\to\, {}_AL^2B_A\,\cong\,{}_AL^2(B)\boxtimes_B L^2(B)_A \\
S\,=\,\tikzmath[scale=.1]{
\fill[fill=\colB] 
(-4,-2.5) [rounded corners=4]-- (-4,3) --
(4,3) -- (4,-2.5) [sharp corners]-- (2,-2.5) [rounded corners = 5.5]-- (2,.5) -- (-2,.5) [sharp corners]-- (-2,-2.5) [rounded corners=4]-- cycle;
\fill[fill=\colA] (-2,-2.5) -- (2,-2.5) [rounded corners = 5.5]-- (2,.5) -- (-2,.5) [sharp corners]-- cycle;
\draw (2,-2.5) [rounded corners = 5.5]-- (2,.5) -- (-2,.5) -- (-2,-2.5);}
\,\,&:\,\,{}_BL^2(B)_B\,\to\, {}_BL^2(B)\boxtimes_A L^2(B)_B
\end{split}
\end{equation}
satisfying the duality equations \eqref{eq: duality equations}, and 
appeal to Theorem \ref{thm: normalizing duality data}
in order to achieve the normalization \eqref{eq:duality normalization}.

Using equation \eqref{eq:Radon-Nikodym-E} we see that the map $R$ defined by 
$\sqrt{\phi}\mapsto \sqrt{\phi\circ E}$ is an isometry.
Let $e:=RR^*
\def\colA{black!10}
\def\colB{black!30}
=\,\tikzmath[scale=\squarescale]{
\clip[rounded corners=2] (-4,-3) rectangle (4,3);
\fill[fill=\colA] 
(-4,-3) [rounded corners=2]-- (-4,3) [sharp corners]-- (-2,3) [rounded corners = 4]-- (-2,1) -- (2,1) [sharp corners]-- (2,3) [rounded corners=2]--
(4,3) -- (4,-3) [sharp corners]-- (2,-3) [rounded corners = 4]-- (2,-1) -- (-2,-1) [sharp corners]-- (-2,-3) [rounded corners=2]-- cycle;
\fill[fill=\colB] (-2,-3) -- (2,-3) [rounded corners = 4]-- (2,-1) -- (-2,-1) [sharp corners]-- cycle;
\fill[fill=\colB] (-2,3) -- (2,3) [rounded corners = 4]-- (2,1) -- (-2,1) [sharp corners]-- cycle;
\draw (2,3) [rounded corners = 4]-- (2,1) -- (-2,1) -- (-2,3);
\draw (2,-3) [rounded corners = 4]-- (2,-1) -- (-2,-1) -- (-2,-3);
}
$ be the corresponding Jones projection.  By \cite[Theorem 1.1.6]{Popa(Classification-of-subfactors-and-their-endomorphisms)}, there exists a set of elements $b_j\in B$ such that
$\{b_j e b_j^*\}$ are mutually orthogonal projections forming a partition of unity, and such that $\sum b_j b_j^*\in B$ is a bounded operator.
Here, both $b_j$ and $b_j^*$ refer to left multiplication operators on $L^2B$.
It follows that the map $\sum b_j:\bigoplus_j L^2(B) \to L^2(B)$ is also bounded.
Let $K$ be the right $A$-module $\bigoplus_j L^2A$, and let $m$ and $\bar m$ be the following two maps:
\[
\textstyle m: K \underset A \boxtimes L^2B \cong \underset j \bigoplus L^2(B) \xrightarrow{\sum (b_j\cdot)} L^2B,\quad
\bar m: L^2B \underset A \boxtimes \bar K \cong \underset j \bigoplus L^2(B) \xrightarrow{\sum (\cdot b_j)} L^2B.
\]
Graphically, the equation $\sum b_j e b_j^*=1$ means that the map
\[
\tikzmath[scale=\squarescale]{
\def \roundedsmall {8}
\def\colA{black!10}
\def\colB{black!30}
\coordinate	(Dupper left corner) at (-13,-11);
\coordinate	(Dlower right corner) at (4,-29);
\coordinate	(Dupper right corner of middle line) at (-2,-14);
\coordinate	(Dlower left corner) at ($(Dlower right corner -| Dupper left corner)-(2,0)$); \coordinate(Dupper right corner) at (Dlower right corner |- Dupper left corner);
\coordinate	(lower left corner) at (-13,-12);
\coordinate	(upper right corner) at (4,6.5);
\coordinate	(lower right corner of middle line) at (-2,-9);
\coordinate	(upper left corner) at ($(upper right corner -| lower left corner)-(2,0)$); \coordinate(lower right corner) at (upper right corner |- lower left corner);
\draw 		(-10,-2) node[inner xsep = 1.5, inner ysep = 2.5] (b) {$m^*$};
\draw 		(-10,-21) node (a) {$m$};
\fill[\colA] (Dupper left corner) -- 
(a.north east -| Dupper left corner) -- (a.north)[rounded corners=\roundedsmall] -- (a.north |- Dupper right corner of middle line) -- (Dupper right corner of middle line)[sharp corners] -- 
(Dupper right corner of middle line |- Dlower left corner)[rounded corners=10] --
(Dlower right corner) -- (upper right corner) [sharp corners]-- (lower right corner of middle line |- upper left corner)
[rounded corners=\roundedsmall]-- (lower right corner of middle line) -- (b.south |- lower right corner of middle line)
[sharp corners]-- (b.south) -- (b.south east -| lower left corner) -- cycle;
\fill[\colB](a.north) -- (a.north east) -- (a.south east) -- (a.south east -| Dlower left corner) [rounded corners=10] -- 
(Dlower left corner)[sharp corners] -- (Dupper right corner of middle line |- Dlower left corner)[rounded corners=\roundedsmall]
-- (Dupper right corner of middle line) -- (a.north |- Dupper right corner of middle line)[sharp corners] -- cycle;
\draw (a.north east -| Dlower left corner) -- (a.north east) -- (a.south east) -- (a.south east -| Dlower left corner)
($(a.south east -| Dupper left corner)+(2,0)$);
\draw[rounded corners=\roundedsmall] (a.north) -- (a.north |- Dupper right corner of middle line) -- (Dupper right corner of middle line) -- 
(Dupper right corner of middle line |- Dlower left corner);
\fill[\colB](b.south) -- (b.south east) -- (b.north east) -- (b.north east -| upper left corner) [rounded corners=10] -- 
(upper left corner)[sharp corners] -- (lower right corner of middle line |- upper left corner)[rounded corners=\roundedsmall]
-- (lower right corner of middle line) -- (b.south |- lower right corner of middle line)[sharp corners] -- cycle;
\draw[densely dotted] (a.north east -| Dupper left corner)
-- (b.south east -| lower left corner);
\draw (b.south east -| upper left corner) -- (b.south east) -- (b.north east) -- (b.north east -| upper left corner)
($(b.north east -| lower left corner)+(2,0)$)
;
\draw[rounded corners=\roundedsmall] (b.south) -- (b.south |- lower right corner of middle line) -- (lower right corner of middle line) -- 
(lower right corner of middle line |- upper left corner);
	}
\,\,:\,\,L^2(B)\,\to\, L^2(B)
\]
is the identity, where the dotted line stands for $K$.
It is then easy to check that, along with $R$, the map
\[
\def\colA{black!10}
\def\colB{black!30}
\def \roundedsmall {8}
S \,\,:=\,\, 
\tikzmath[scale=\squarescale]{
\coordinate	(lower left middle line) at (-4,-9);
\coordinate	(upper left corner) at (-16,6);
\draw 		(-12,-2) node[inner xsep = 1.5, inner ysep = 2.5] (a) {$m^*$};
\fill[\colB](a.south) -- (a.south east) -- (a.north east) -- (a.north east -| upper left corner) [rounded corners=10] -- 
(upper left corner)[sharp corners] -- (lower left middle line |- upper left corner)[rounded corners=\roundedsmall]
-- (lower left middle line) -- (a.south |- lower left middle line)[sharp corners] -- cycle;
\coordinate	(lower right middle line) at (4,-9);
\coordinate	(upper right corner) at (16,6);
\draw 		(12,-2) node[inner xsep = 1.5, inner ysep = 2.5] (b) {$\bar m^*$};
\fill[\colB](b.south) -- (b.south west) -- (b.north west) -- (b.north west -| upper right corner) [rounded corners=10] -- 
(upper right corner)[sharp corners] -- (lower right middle line |- upper right corner)[rounded corners=\roundedsmall]
-- (lower right middle line) -- (b.south |- lower right middle line)[sharp corners] -- cycle;
\coordinate (L) at ($(a.south) - (2.5,0)$);
\coordinate (R) at ($(b.south) + (2.5,0)$);
\draw		(-8,-17) node[inner xsep=8] (c) {$m$};
\draw		(8,-17) node[inner ysep = 2.7, inner xsep=8] (d) {$\bar m$};
\fill[\colA] (c.north -| L) -- (L) -- (a.south) [rounded corners=\roundedsmall]-- (a.south |- lower left middle line) -- (lower left middle line) [sharp corners]-- (lower left middle line |- upper right corner)
-- (lower right middle line |- upper right corner)[rounded corners=\roundedsmall]
-- (lower right middle line) -- (b.south |- lower right middle line)[sharp corners] -- (b.south)
-- (R) -- (d.north -| R);
\draw (a.south west -| upper left corner) -- (a.south east) -- (a.north east) -- (a.north east -| upper left corner) 
(lower left middle line |- upper left corner)[rounded corners=\roundedsmall]
-- (lower left middle line) -- (a.south |- lower left middle line)[sharp corners] -- (a.south);
\draw (b.south east -| upper right corner) -- (b.south west) -- (b.north west) -- (b.north west -| upper right corner) 
(lower right middle line |- upper right corner)[rounded corners=\roundedsmall]
-- (lower right middle line) -- (b.south |- lower right middle line)[sharp corners] -- (b.south);
\fill[\colB] (c.south east) -- (c.north east) -- ($(c.north)+(3,0)$) [rounded corners=10]-- ($(c.north)+(3,5)$) -- ($(d.north)+(-3,5)$) [sharp corners]-- ($(d.north)+(-3,0)$) -- (d.north west) -- (d.south west)
-- (d.south -| upper right corner)
[rounded corners=10]-- ($(d.south -| upper right corner) - (0,5)$)
-- ($(c.south -| upper left corner) - (0,5)$)
[sharp corners]-- (c.south -| upper left corner);
\draw[rounded corners=10] ($(c.north)+(3,0)$) -- ($(c.north)+(3,5)$) -- ($(d.north)+(-3,5)$) -- ($(d.north)+(-3,0)$);
\draw (c.south west -| upper left corner) -- (c.south east) -- (c.north east) -- (c.north east -| upper left corner);
\draw (d.south east -| upper right corner) -- (d.south west) -- (d.north west) -- (d.north west -| upper right corner);
\draw[densely dotted] (L) -- (c.north -| L);
\draw[densely dotted] (R) -- (d.north -| R);
} 
\,\,\,=\,\,\,
\tikzmath[scale=\squarescale]{
\coordinate	(lower left middle line) at (-4,-9);
\coordinate	(upper left corner) at (-16,6);
\draw 		(-12,-2) node[inner xsep = 1.5, inner ysep = 2.5] (a) {$m^*$};
\fill[\colB](a.south) -- (a.south east) -- (a.north east) -- (a.north east -| upper left corner) [rounded corners=10] -- 
(upper left corner)[sharp corners] -- (lower left middle line |- upper left corner)[rounded corners=\roundedsmall]
-- (lower left middle line) -- (a.south |- lower left middle line)[sharp corners] -- cycle;
\coordinate	(lower right middle line) at (4,-9);
\coordinate	(upper right corner) at (7.5,6);
\coordinate (L) at ($(a.south) - (2.5,0)$);
\draw		(-4,-17) node[inner xsep=10] (c) {$m$};
\coordinate (N) at ($(c.north) + (5.5,0)$);
\fill[\colA] (c.north -| L) -- (L) -- (a.south) [rounded corners=\roundedsmall]-- (a.south |- lower left middle line) -- (lower left middle line) [sharp corners]-- (lower left middle line |- upper right corner)
-- (N |- upper left corner) -- (N);
\fill[\colB] (c.south east) -- (c.north east) -- (N) -- (N |- upper left corner) [rounded corners=10]-- (upper right corner)
-- ($(c.south -| upper right corner) - (0,5)$)
-- ($(c.south -| upper left corner) - (0,5)$)
[sharp corners]-- (c.south -| upper left corner);
\draw (a.south west -| upper left corner) -- (a.south east) -- (a.north east) -- (a.north east -| upper left corner) 
(lower left middle line |- upper left corner)[rounded corners=\roundedsmall]
-- (lower left middle line) -- (a.south |- lower left middle line)[sharp corners] -- (a.south)
(N |- upper left corner) -- (N);
\draw (c.south west -| upper left corner) -- (c.south east) -- (c.north east) -- (c.north east -| upper left corner);
\draw[densely dotted] (L) -- (c.north -| L);
} 
\,\,\,=\,\,\,
\tikzmath[scale=\squarescale]{
\coordinate	(lower right middle line) at (4,-9);
\coordinate	(upper right corner) at (16,6);
\draw 		(12,-2) node[inner xsep = 1.5, inner ysep = 2.5] (a) {$\bar m^*$};
\fill[\colB](a.south) -- (a.south west) -- (a.north west) -- (a.north west -| upper right corner) [rounded corners=10] -- 
(upper right corner)[sharp corners] -- (lower right middle line |- upper right corner)[rounded corners=\roundedsmall]
-- (lower right middle line) -- (a.south |- lower right middle line)[sharp corners] -- cycle;
\coordinate	(lower left middle line) at (4,-9);
\coordinate	(upper left corner) at (-7.5,6);
\coordinate (R) at ($(a.south) + (2.5,0)$);
\draw		(4,-17) node[inner ysep = 2.7, inner xsep=10] (c) {$\bar m$};
\coordinate (N) at ($(c.north) - (5.5,0)$);
\fill[\colA] (c.north -| R) -- (R) -- (a.south) [rounded corners=\roundedsmall]-- (a.south |- lower right middle line) -- (lower right middle line) [sharp corners]-- (lower right middle line |- upper left corner)
-- (N |- upper right corner) -- (N);
\fill[\colB] (c.south west) -- (c.north west) -- (N) -- (N |- upper right corner) [rounded corners=10]-- (upper left corner)
-- ($(c.south -| upper left corner) - (0,5)$)
-- ($(c.south -| upper right corner) - (0,5)$)
[sharp corners]-- (c.south -| upper right corner);
\draw (a.south east -| upper right corner) -- (a.south west) -- (a.north west) -- (a.north west -| upper right corner) 
(lower right middle line |- upper right corner)[rounded corners=\roundedsmall]
-- (lower right middle line) -- (a.south |- lower right middle line)[sharp corners] -- (a.south)
(N |- upper right corner) -- (N);
\draw (c.south east -| upper right corner) -- (c.south west) -- (c.north west) -- (c.north west -| upper right corner);
\draw[densely dotted] (R) -- (c.north -| R);
} 
\]
satisfies the duality equations \eqref{eq: duality equations}.
\end{proof}

The above proof also shows that the following 
variant of Proposition~\ref{Prop. 8.3} holds.

\begin{proposition}
Let $f:A\to B$ be a map between arbitrary von Neumann algebras, and let $E:B\to A$ be a conditional expectation such that $E(x) \ge \mu x$ for all $x\in B_+$.
Then $f$ is a finite homomorphism in the sense (see the discussion before Conjecture \ref{conj: L2 functoriality}) that ${}_AL^2B_B$ admits a not-necessarily normalized dual bimodule.~\hfill~$\square$
\end{proposition}

As a first application of the Pimsner--Popa inequality, we have:

\begin{lemma}\label{lem: N subset P subset M}
Let $N\subset P\subset M$ be factors. Then $[M:N]<\infty\Rightarrow [P:N]<\infty$.
\end{lemma}
\begin{proof}
Let $E:M\to N$ be the minimal conditional expectation.
Then $E|_P$ is a conditional expectation subject to the same bound: $E|_P(x) \ge [M:N]^{-1} x, {\forall x\in P_+}$.
The subfactor $N \subset P$ satisfies the condition of Proposition \ref{Prop. 8.3}, and so ${}_NL^2P_P$ is dualizable.
\end{proof}

\begin{corollary}\label{cor: HK finite ==> H finite}
Let $N$, $P$, $M$ be factors, and let ${}_NH_P$ and ${}_PK_M$ be non-zero bimodules.
If their fusion ${}_NH\boxtimes_PK_M$ is a dualizable $N$-$M$-bimodule, then ${}_NH_P$ and ${}_PK_M$ are dualizable.
\end{corollary}
\begin{proof}
We show that ${}_NH_P$ is dualizable.
Let $P'$ be the commutant of $P$ on $H$, and let $M'$ be the commutant of $M$ on $H\boxtimes_PK$.
We have $N\subset P'\subset M'$.
By Lemma \ref{lem: ind = dim}, we have $[M':N]<\infty$, which implies $[P':N]<\infty$ by the above lemma.
By a second application of Lemma~\ref{lem: ind = dim}, we deduce that ${}_NH_P$ is dualizable.

This argument might looks circular at first glance, as Lemma \ref{lem: ind = dim} depends on \eqref{eq:properties of index 4}.
However, Lemma \ref{lem: ind = dim} only depends on the special case of \eqref{eq:properties of index 4} mentioned in footnote ${\text{\footnotesize\ref{fn: tensor with invertible}}}$,
and is thus independent of the result of this corollary.
\end{proof}

Unless the factors are finite-dimensional, 
the Pimsner--Popa inequality also provides a characterization of the minimal conditional expectation and of the index.
For a subfactor $N \subset M$ of finite index, let $E_0(m):=[M:N]^{-\frac12}R^*(m\otimes 1)R$, as before.

\begin{proposition}\label{lem: Extremality properties of E_0}
Assume the factors $N$ and $M$ are infinite-dimensional, and $N \subset M$ is of finite index.  In this case,
\begin{list}{}{\leftmargin = 0pt \labelsep = 4pt \labelwidth = -4pt}
\item[\it a.] if $0<\lambda < [M: N]$, there exists $x\in M_+$ such that 
\begin{equation}\label{eq: characterization Pimsner-Popa 2}
E_0(x) \not \ge \lambda^{-1} x.
\end{equation}
In other words, $\Ind(E_0)\ge[M:N]$, and therefore, by equation \eqref{eq:Pimsner-Popa inequality}, $\Ind(E_0) = [M:N]$.
\item[\it b.] if $E\!:\!M\to N$ is a conditional expectation and $E\not = E_0$, then $\exists\, x\in M_+$ such that 
\begin{equation}\label{eq: characterization Pimsner-Popa 1}
E(x) \not \ge [M: N]^{-1} x.
\end{equation}
In other words, $\Ind(E)>[M:N]$.
\end{list}
\end{proposition}

\begin{proof}
\def\colA{black!10}
\def\colB{black!30}
{\it a.}
We let
\[
\tikzmath[scale=.085]{ \fill[fill=\colA] (-4,-2.5) [rounded corners=3.5]-- (-4,2.5) --(4,2.5) -- (4,-2.5) [sharp corners]-- (2,-2.5) [rounded corners = 4.8]-- (2,.5) -- (-2,.5) [sharp corners]-- (-2,-2.5) [rounded corners=3.5]-- cycle; \fill[fill=\colB] (-2,-2.5) -- (2,-2.5) [rounded corners = 4.8]-- (2,.5) -- (-2,.5) [sharp corners]-- cycle; \draw (2,-2.5) [rounded corners = 4.8]-- (2,.5) -- (-2,.5) -- (-2,-2.5);}
:L^2N\to L^2M\qquad
\text{and}\qquad 
\tikzmath[scale=.085]{ \fill[fill=\colB] (-4,-2.5) [rounded corners=3.5]-- (-4,2.5) --(4,2.5) -- (4,-2.5) [sharp corners]-- (2,-2.5) [rounded corners = 4.8]-- (2,.5) -- (-2,.5) [sharp corners]-- (-2,-2.5) [rounded corners=3.5]-- cycle; \fill[fill=\colA] (-2,-2.5) -- (2,-2.5) [rounded corners = 4.8]-- (2,.5) -- (-2,.5) [sharp corners]-- cycle;\draw (2,-2.5) [rounded corners = 4.8]-- (2,.5) -- (-2,.5) -- (-2,-2.5);}
:L^2M\to L^2M\boxtimes_N L^2M
\] 
be normalized duality maps for the bimodule ${}_NH{}_M:={}_NL^2M_M$.
Let $d={[M: N]}^{\frac12}$ be the statistical dimension of $H$, and let $e=
\def\colA{black!10}
\def\colB{black!30}
d^{-1}\,\tikzmath[scale=\squarescale]{
\clip[rounded corners=2] (-4,-3) rectangle (4,3);
\fill[fill=\colB] 
(-4,-3) [rounded corners=2]-- (-4,3) [sharp corners]-- (-2,3) [rounded corners = 4]-- (-2,1) -- (2,1) [sharp corners]-- (2,3) [rounded corners=2]--
(4,3) -- (4,-3) [sharp corners]-- (2,-3) [rounded corners = 4]-- (2,-1) -- (-2,-1) [sharp corners]-- (-2,-3) [rounded corners=2]-- cycle;
\fill[fill=\colA] (-2,-3) -- (2,-3) [rounded corners = 4]-- (2,-1) -- (-2,-1) [sharp corners]-- cycle;
\fill[fill=\colA] (-2,3) -- (2,3) [rounded corners = 4]-- (2,1) -- (-2,1) [sharp corners]-- cycle;
\draw (2,3) [rounded corners = 4]-- (2,1) -- (-2,1) -- (-2,3);
\draw (2,-3) [rounded corners = 4]-- (2,-1) -- (-2,-1) -- (-2,-3);
}
$\, be the Jones projection.
Since $\dim(N)=\infty$, one can find a right $M$-module $K_M$ such that 
$K\boxtimes_M L^2(M)_N$ is isomorphic to $L^2N_N$ --- use the classification of modules
over factors of type different from {\it I}${}_n$.
Pick a unitary isomorphism $u:K\boxtimes_M L^2(M)_N\to L^2N_N$ and set $x:=(u\otimes 1)(1\otimes e)(u^*\otimes 1)$.
We then have
\[
E_0\big(x\big) = E_0\bigg(d^{-1}\,
\tikzmath[scale=\squarescale]{
\clip[rounded corners=10] (-12,-17) rectangle (5,17); 
\coordinate (x) at (-12,0);
\fill[rounded corners=10, fill=\colA] 
(-12,10.5) -- (-12,20) -- (12,20) -- (12,-20) -- (-12,-20) [sharp corners] -- (-12,-10.5)  -- (-7,-10.5) [rounded corners=7]-- (-7,-2) -- (0,-2) -- (0,-21) -- (7,-21) -- (7,21) -- (0,21) -- (0,2) -- (-7,2) [sharp corners]-- (-7,10.5) -- cycle;
\fill[fill=\colB] 
(-7,10.5) -- (-10,10.5) -- (-10,-10.5) [sharp corners]--    (-7,-10.5) [rounded corners=7]-- (-7,-2) -- (0,-2) -- (0,-21) -- (7,-21) -- (7,21) -- (0,21) -- (0,2) -- (-7,2) [sharp corners]-- cycle;
                                                                \draw[rounded corners=7] (-7,-10.5) -- (-7,-2) -- (0,-2) -- (0,-21) -- (7,-21) -- (7,21) -- (0,21) -- (0,2) -- (-7,2) -- (-7,10.5);
                                                                \draw[densely dotted] (-10,-10) -- (-10,10);
\draw (-9,8.5) node(a)[fill=white, inner xsep =3, inner ysep =3]{$u^*$};
\draw (-9,-8.5) node(b)[fill=white, inner xsep =5, inner ysep =4]{$u$};
\draw (a.south west -| x) -- (a.south east) -- (a.north east) -- (a.north west -| x);
\draw (b.south west -| x) -- (b.south east) -- (b.north east) -- (b.north west -| x);
} 
\,\,\,\bigg) = d^{-2}
\tikzmath[scale=\squarescale]{
\coordinate (x) at (-12,0);
\fill[rounded corners=10, fill=\colA] 
(-12,10.5) -- (-12,17) -- (12,17) -- (12,-17) -- (-12,-17) [sharp corners] -- (-12,-10.5)  -- (-7,-10.5) [rounded corners=7]-- (-7,-2) -- (0,-2) -- (0,-14) -- (7,-14) -- (7,14) -- (0,14) -- (0,2) -- (-7,2) [sharp corners]-- (-7,10.5) -- cycle;
\fill[fill=\colB] 
(-7,10.5) -- (-10,10.5) -- (-10,-10.5) [sharp corners]--    (-7,-10.5) [rounded corners=7]-- (-7,-2) -- (0,-2) -- (0,-14) -- (7,-14) -- (7,14) -- (0,14) -- (0,2) -- (-7,2) [sharp corners]-- cycle;
                                                                \draw[rounded corners=7] (-7,-10.5) -- (-7,-2) -- (0,-2) -- (0,-14) -- (7,-14) -- (7,14) -- (0,14) -- (0,2) -- (-7,2) -- (-7,10.5);
                                                                \draw[densely dotted] (-10,-10) -- (-10,10);
\draw (-9,8) node(a)[fill=white, inner xsep =3, inner ysep =3]{$u^*$};
\draw (-9,-8) node(b)[fill=white, inner xsep =5, inner ysep =4]{$u$};
\draw (a.south west -| x) -- (a.south east) -- (a.north east) -- (a.north west -| x);
\draw (b.south west -| x) -- (b.south east) -- (b.north east) -- (b.north west -| x);
} 
= [M:N]^{-1}
\tikzmath[scale=\squarescale]{
\coordinate (x) at (-12,0);
\fill[rounded corners=10, fill=\colA] 
(-12,10.5) -- (-12,17) -- (0,17) -- (0,-17) -- (-12,-17) [sharp corners] -- (-12,-10.5)  -- (-7,-10.5) -- (-7,10.5) -- cycle;
\fill[fill=\colB] 
(-7,10.5) -- (-10,10.5) -- (-10,-10.5) [sharp corners]--    (-7,-10.5) -- cycle;
                                                                \draw[rounded corners=7] (-7,-10.5) -- (-7,10.5);
                                                                \draw[densely dotted] (-10,-10) -- (-10,10);
\draw (-9,7.5) node(a)[fill=white, inner xsep =3, inner ysep =3]{$u^*$};
\draw (-9,-7.5) node(b)[fill=white, inner xsep =5, inner ysep =4]{$u$};
\draw (a.south west -| x) -- (a.south east) -- (a.north east) -- (a.north west -| x);
\draw (b.south west -| x) -- (b.south east) -- (b.north east) -- (b.north west -| x);
} 
=[M:N]^{-1},
\]
where the dotted line stands for $K$.
Since $x$ is a non-zero projection and $[M:N]^{-1}$ is a scalar, it follows that $E_0(x) \not \ge \mu x$ for any $\mu > [M:N]^{-1}$.

\noindent
{\it b.}
We need to check that $E_0$ minimizes $\Ind(E)$.
Let $p_i$ be the minimal central projections of $N'\cap M=\mathrm{End}({}_NL^2M_M)$,
let $d={[M:N]}^{\frac12}$, and let $d_i={[p_iMp_i:p_iN]}^{\frac12}$.
Note that $p_iN \subset p_i M p_i$ is an irreducible subfactor,
that is $(p_i N)' \cap p_i M p_i = \IC$.
Thus by \cite[Proposition 5.3]
   {Longo(Index-of-subfactors-and-statistics-of-quantum-fields-I)},
there exits only one  conditional expectation $p_i M p_i \to p_i N$.
Using part ({\emph a}) and Proposition~\ref{Prop. 8.3} we conclude
that for $p_iN \subset p_i M p_i$, the index
coincides with the minimal index.  
Thus $d_i = (\Ind(p_i N, p_i M p_i))^{\frac{1}{2}}$.
According to \cite[Theorem 5.5]{Longo(Index-of-subfactors-and-statistics-of-quantum-fields-I)}, it suffices to check that $E_0|_{N'\cap M}$ is a trace and that
\begin{equation} \label{eq: E_0(p_i) = d_i/d}
E_0(p_i) = \frac{d_i}{\sum d_i}.
\end{equation}
The first condition was proven in Lemma \ref{lem: phi is a trace}.
To check the latter, let $H_i:=p_i(L^2(M))$.
We then have
\[\vspace{.3cm}
\begin{split}
d_i&=\dim({}_{p_iN}\,L^2(p_iMp_i)\,{}_{p_iMp_i})=
\dim({}_{p_iN}\,p_i(L^2M)p_i\,{}_{p_iMp_i})\\&=
\dim({}_{p_iN}\,p_iL^2M\underset M\boxtimes L^2Mp_i\,{}_{p_iMp_i})=
\dim({}_{p_iN}\,p_iL^2M\,{}_M)=\dim(H_i{})=\tikzmath[scale=\squarescale]
	{\useasboundingbox (-11,-4) rectangle (8,4);
	\fill[rounded corners=10, fill=\colA] (-11,-10) rectangle (8,10);
	\draw[rounded corners=8, fill=\colB] (-5,-6) rectangle (3,6);
	\draw 	(-5,0) node[fill=white, draw, inner ysep=3, inner xsep=2]{$p_i$};
	}
,
\end{split}
\]
where the second equality is Lemma \ref{lem: L^2(pAp) = pL^2(A)p}, the fourth one holds by equations \eqref{eq:properties of index 2} and \eqref{eq:properties of index 4},
and the last one is given by Lemma \ref{lem: direct summand of dualizable bimodule}.
Note that $\sum d_i=d$ now follows by equation \eqref{eq:properties of index 3}.
By the definition of $E_0$, we therefore have
\[
\big(\textstyle \sum d_i\big)\cdot E_0(p_i) = d\,E_0(p_i) =
\tikzmath[scale=\squarescale]
	{\fill[rounded corners=10, fill=\colA] (-11,-10) rectangle (8,10);
	\draw[rounded corners=8, fill=\colB] (-5,-6) rectangle (3,6);
	\draw 	(-5,0) node[fill=white, draw, inner ysep=3, inner xsep=2]{$p_i$};
	}
\,=d_i. \qedhere
\]
\end{proof}

\begin{corollary}\label{cor: justification of the name minimal}
Let $N\subset M$ be infinite-dimensional factors, let $[M:N]$ be the  index, as in Definition \ref{def: min index}, and let $\Ind(N,M)$ be the minimal index, as in equation \eqref{eq: [M:N] = inf_E IndE}.
Then
\begin{equation}\label{eq: [M:N] = Ind(N,M)}
[M:N] = \Ind(N,M).
\end{equation}
\end{corollary}

\begin{warning}\label{warn: Longo index is bad}
As noted in \cite{Longo(Index-of-subfactors-and-statistics-of-quantum-fields-I)}, the equality \eqref{eq: [M:N] = Ind(N,M)} fails 
to be true, for example, for the subfactors $\IC\hookrightarrow M_n(\IC)$.
The minimal index $\Ind(N,M)$ is not a good notion of index in the case of finite dimensional factors.
\end{warning}

Now is an appropriate moment to pay our debt to Remark \ref{remark: non-zero maps are enough}, by giving a particularly mild condition that ensures that a bimodule is dualizable --- compare \cite[Theorem
4.1]{Longo(Index-of-subfactors-and-statistics-of-quantum-fields-II)}.

\begin{proposition}\label{Prop: existence of non-zero maps is enough}
Let ${}_AH_B$ and ${}_B K_A$ be irreducible bimodules between von Neumann algebras with finite-dimensional centers.
If there exist non-zero maps $\tilde R:{}_AL^2(A)_A \rightarrow
{}_AH\boxtimes_B K_A$ and
$\tilde S:{}_BL^2(B)_B \rightarrow  {}_BK\boxtimes_A H_B$, then
${}_AH_B$ and ${}_B K_A$ are dualizable.
\end{proposition}

\begin{proof}
\def\colA{black!10} \def\colB{black!30}
We denote $\tilde R$ by \,$\tikzmath[scale=.085]{
\fill[fill=\colA] (-4,-2.5) [rounded corners=3.5]-- (-4,2.5) --(4,2.5)
-- (4,-2.5) [sharp corners]-- (2,-2.5) [rounded corners = 4.8]--
(2,.5) -- (-2,.5) [sharp corners]-- (-2,-2.5) [rounded corners=3.5]--
cycle; \fill[fill=\colB] (-2,-2.5) -- (2,-2.5) [rounded corners =
4.8]-- (2,.5) -- (-2,.5) [sharp corners]-- cycle; \draw (2,-2.5)
[rounded corners = 4.8]-- (2,.5) -- (-2,.5) -- (-2,-2.5);}$\,
and $\tilde S$ by \,$\tikzmath[scale=.085]{ \fill[fill=\colB]
(-4,-2.5) [rounded corners=3.5]-- (-4,2.5) --(4,2.5) -- (4,-2.5)
[sharp corners]-- (2,-2.5) [rounded corners = 4.8]-- (2,.5) -- (-2,.5)
[sharp corners]-- (-2,-2.5) [rounded corners=3.5]-- cycle;
\fill[fill=\colA] (-2,-2.5) -- (2,-2.5) [rounded corners = 4.8]--
(2,.5) -- (-2,.5) [sharp corners]-- cycle;\draw (2,-2.5) [rounded
corners = 4.8]-- (2,.5) -- (-2,.5) -- (-2,-2.5);}$\,. 
We may assume without loss of generality that $A$ and $B$ are factors,
and that $\tilde R$ and $\tilde S$ are isometries.
Define conditional expectations $E:B'\to A$ and $F:A'\to B$ by
\[
E:x\,\mapsto\,\tikzmath[scale=\squarescale]
       {\fill[rounded corners=10, fill=\colA] (-13,-13) rectangle (12,13);
       \draw[rounded corners=10, fill=\colB] (-5,-8) rectangle (5,8);
       \coordinate (x) at (-13,0); \coordinate (x') at (-13.1,0);
       \draw   (-5,0) node(a)[fill=white]{$b$};
       \fill[fill=white] (a.north east -| x') -- (a.north east) -- (a.south
east) -- (a.south east -| x');
       \draw (a.north east -| x) -- (a.north east) -- (a.south east) --
(a.south east -| x);
       \node(a) at (-8,0) [inner sep=0]{$x$};
       }\,\,,
\qquad
F:y\,\mapsto\,\tikzmath[scale=\squarescale]
       {\fill[rounded corners=10, fill=\colB] (-12,-13) rectangle (13,13);
       \draw[rounded corners=10, fill=\colA] (-5,-8) rectangle (5,8);
       \coordinate (x) at (13,0); \coordinate (x') at (13.1,0);
       \draw   (5,0) node(a)[fill=white]{$b$};
       \fill[fill=white] (a.north west -| x') -- (a.north west) -- (a.south
west) -- (a.south west -| x');
       \draw (a.north west -| x) -- (a.north west) -- (a.south west) --
(a.south west -| x);
       \node(a) at (8,0) [inner sep=0]{$y$};
       }\,\,,
\]
where the commutants are taken on $H$.

Denote by $U(A)$ the group of unitary elements of $A$.  For any non-zero projection $p\in B'$, the least upper bound
$\bigvee_{u\in U(A)} upu^*$
belongs to $A'\cap B'=\mathrm{End}({}_AH_B)=\IC$ and is therefore equal to $1$.
If $E(p)$ were zero, we would have
\[
\textstyle1=E(1)=E\big( \bigvee upu^*\big)=\bigvee E(upu^*)=\bigvee uE(p)u^*=0 \, .
\]
Thus the conditional expectation $E$ is faithful, and similarly $F$ is faithful.
It follows from \cite[Proposition
4.4]{Longo(Index-of-subfactors-and-statistics-of-quantum-fields-II)}
that the inclusion $A\subset B'$ has finite index.
By Lemma \ref{lem: ind = dim}, we then have $\dim({}_AH_B)=\llbracket
B':A\rrbracket<\infty$, and so ${}_AH_B$ is dualizable.
The bimodule ${}_BK_A$ is dualizable for similar reasons.
\end{proof}

We finish this section by establishing some useful inequalities for the matrix of statistical dimensions $\llbracket B : A \rrbracket$ --- recall Definition \ref{def: [[ B : A ]]} --- associated to a finite homomorphism $A \to B$ of von Neumann algebras with finite-dimensional centers.  Our proofs are all based on the Pimsner--Popa inequality.

Let $A_1,B_1\subset \bfB(H_1)$ and $A_2,B_2\subset \bfB(H_2)$ be von Neumann algebras such that $A_i$ commutes with $B_i$.
The algebras $A_1\vee B_1\subset \bfB(H_1)$ and $A_2\vee B_2\subset \bfB(H_2)$ are therefore completions of the corresponding algebraic tensor products
$A_1\otimes_\alg B_1$ and $A_2\otimes_\alg B_2$.
Given homomorphisms $\alpha:A_1\to A_2$ and $\beta:B_1\to B_2$, 
the induced map $\alpha\otimes \beta:A_1\otimes_\alg B_1\to A_2\otimes_\alg B_2$ 
does not always extend to a map $A_1\vee B_1\to A_2\vee B_2$.
This will however be the case in the presence of 
an $\alpha\otimes \beta$-equivariant surjective  homomorphisms
$h:H_1\to H_2$.

\begin{lemma}\label{lem: A1v B1 --> A2 v B2 finite}
Let $A_i$, $B_i$, $H_i$, and $h$ be as above.
If the algebras $A_i$, $B_i$, and $A_i\vee B_i$ have finite-dimensional centers and the homomorphisms $\alpha:A_1\to A_2$ and $\beta:B_1\to B_2$ are finite,
then the induced map
\begin{equation}\label{eq: A1 v B_1 --> A2 v B2}
\alpha\otimes \beta\,:\, A_1\vee B_1\to A_2\vee B_2
\end{equation} 
is a finite homomorphism.
\end{lemma}
\begin{proof}
Let us write $\vee_{H_1}$ and $\vee_{H_2}$ for the completions inside $\bfB(H_1)$ and $\bfB(H_2)$, respectively.
We can then factor the map \eqref{eq: A1 v B_1 --> A2 v B2} as
\[
A_1\vee_{H_1} B_1\longrightarrow A_1\vee_{H_2} B_1\longrightarrow A_2\vee_{H_2} B_1\longrightarrow A_2\vee_{H_2} B_2.
\]
The first map is a projection, and therefore finite.
We analyze the second map --- the third one is similar.
From now on let $\vee$ mean $\vee_{H_2}$.
The restriction to $A_1'\cap B_1'=(A_1\vee B_1)'$ of the minimal conditional expectation $E_0:A_1'\to A_2'$
satisfies the same Pimsner--Popa bound as $E_0$.
The homomorphism $(A_1\vee B_1)'\to (A_2\vee B_1)'$ is therefore finite by Proposition \ref{Prop. 8.3}.
Corollary \ref{cor: [B:A]=[A':B']} finishes the argument.
\end{proof}

\begin{proposition}\label{prop: Lemma C5}
Let $A$ be an infinite-dimensional factor sitting in a von Neumann algebra $B$.
If there exists a conditional expectation $E : B \to A$ satisfying the Pimsner--Popa bound
\begin{equation}\label{eq: bound for mu}
E(x) \ge \mu^{-1} x\qquad   \forall x \in  B_+,
\end{equation}
then $B$ has finite-dimensional center.
Furthermore, letting $p_i$ be the minimal central projections of $B$, we then have $\sum [p_iB:A]\le\mu$.
In other words, we have the inequality
\[
\|\llbracket B : A \rrbracket\| \le \sqrt{\mu}\,,
\]
where $\|\,\,\|$ stands for the $\ell^2$ norm of a vector.
\end{proposition}

\begin{proof}
Let $q_i\in B$ be non-zero central projections adding up to $1$.
Since
\[
aE(q_i) = E(aq_i) = E(q_ia) = E(q_i)a
\]
for all $a\in A$, the element $E(q_i)$ is central in $A$, and hence a scalar.
From the bound \eqref{eq: bound for mu}, we conclude that $E(q_i) \ge \mu^{-1}$.
Summing up over all indices $i$, we deduce
\[
\textstyle 1=E(1)=E\big(\sum q_i\big)=\sum_i E(q_i) \ge \sum_i \mu^{-1},
\]
from which it follows that the number of $q_i$'s is at most $\mu$.
The center of $B$ is therefore finite-dimensional.

Now let $p_i$ be the minimal central projections of $B$, and let $B_i:=p_i B$.
The restriction $F_i:=E|_{B_i}:B_i\to A$ satisfies all the properties for being a conditional expectation,
except that it does not send the unit $p_i$ of $B_i$ to $1$.
The map $E_i:=F_i(p_i)^{-1}F_i$ is therefore a conditional expectation.
It satisfies the bound
\[
E_i(x) \ge F_i(p_i)^{-1}\mu^{-1} x\qquad   \forall x \in  B_{i+},
\]
from which it follows that $[B_i:A]\le F_i(p_i)\mu$.
Adding up over indices, we get that $\sum[B_i:A]\le \sum F_i(p_i)\mu =  E(\sum p_i)\mu = E(1)\mu = \mu$.
\end{proof}

The following lemma is, in some sense, dual to Proposition \ref{prop: Lemma C5}:

\begin{proposition}\label{prop: mu ge sum_i [e_i E e_i : e_iA]}
Let $A=\oplus A_i$ be a sum of finitely many infinite-dimensional factors $A_i$,
and suppose that $A$ is a subalgebra of some factor $B$.
Let $E : B \to A$ be a conditional expectation satisfying
\begin{equation}\label{eq: bound for mu bis}
E(x) \ge \mu^{-1} x\qquad   \forall x \in  B_+.
\end{equation}
Letting $p_i$ be the minimal central projections of $A$, we have
\(
\sum\, [p_i B p_i \!:\! A_i]\le\mu
\).
In other words, we have the inequality
\[
\|\llbracket B : A \rrbracket\| \le \sqrt{\mu}\,,
\]
where $\|\,\,\|$ stands for the $\ell^2$ norm of a vector.
\end{proposition}

\begin{proof}
Under our assumption on $A$ the optimal
$\mu$ satisfying~\eqref{eq: bound for mu bis} can be identified
with the Kosaki index $\| E^{-1}(1) \|$ of the conditional expectation $E$, 
see \cite[Theorem 1.1.6]
 {Popa(Classification-of-subfactors-and-their-endomorphisms)}.
By its definition \cite{Kosaki(Extension-of-Jones-index-to-arbitrary-factors), Kosaki(Type-III-factors-and-index-theory)}, the Kosaki index does not change under tensor 
product with another factor.
In particular, given a type {\it III} factor $R$, we conclude that the conditional expectation 
$E \otimes R:B\,\bar\otimes\,R\to A\,\bar\otimes\,R$ satisfies
the same bound \eqref{eq: bound for mu bis} as $E$.
The index of $A_i\otimes R$ in $p_i (B\otimes R) p_i$ being equal to
that of $A_i$ in $p_i B p_i$, we may assume without 
loss of generality that $B$ is a type {\it III} factor.

Let us define
\(
B_{ij} := p_i B\, p_j.
\)
If $B$ is a type {\it III} factor, then the projections $p_i$ are all Murray-von Neumann equivalent;
we can therefore identify each matrix block $B_{ij}$ with a given algebra, say $C$, and get an isomorphism
\[
B=\bigoplus_{ij} B_{ij}\cong M_n(C).
\]
Taking the composite  $B_{ii} \hookrightarrow B  \xrightarrow{E}  A  \twoheadrightarrow  A_i$, we get a conditional expectation $E_i : B_{ii} \to A_i$.
Let $\lambda_i$ be the smallest number for which the Pimsner--Popa inequality
\[
E_i(x) \ge \lambda_i^{-1} x\qquad   \forall x \in  B_{ii+}
\]
holds, and note that there exist projections $e_i \in  B_{ii}$ such that $E_i(e_i) = \lambda_i^{-1} p_i$;
for example, we can take $e_i$ to be a Jones projection as in the proof of Proposition \ref{lem: Extremality properties of E_0}{\it a}.

Let $u_{ij} \in C$ be partial isometries with $u_{ij} u_{ij}^* = e_i$, $u_{ij}^* = u_{ji}$, and $u_{ij} u_{jk} = u_{ik}$.
In particular, we have $u_{ii} = e_i$.
Consider now the projection $Q\in M_n(C)$ given by
\[
Q_{ij} := \textstyle\frac{\,\sqrt{\lambda_i \lambda_j}\,}{\underset k \sum \lambda_k}\, u_{ij}.
\]
We then have
\[
E(Q) = \bigoplus E_i(Q_{ii}) =  \bigoplus E_i\big({\textstyle\frac{\lambda_i}{\underset k \sum \lambda_k}e_i}\big)
=\bigoplus {\textstyle\frac{\lambda_i}{\underset k \sum \lambda_k}} E_i(e_i)
=\bigoplus {\textstyle\frac{1}{\underset k \sum \lambda_k}} p_i=\textstyle\frac{\textstyle 1}{\,\underset k \sum\, \textstyle \lambda_k^{\phantom{t}}}.
\]
Combined with the bound \eqref{eq: bound for mu bis}, the above estimate shows that $\mu\ge \sum\lambda_k$.
To finish the proof,
we use the inequality $\lambda_i \ge [p_i B p_i : p_iA]$, which follows from \eqref{eq: characterization Pimsner-Popa 2} and \eqref{eq: characterization Pimsner-Popa 1}.
\end{proof}

\begin{remark}
We expect that, analogously to Proposition \ref{prop: Lemma C5}, when $A \subset B$ with $B$ a factor,
the existence of a conditional expectation $B \to A$ satisfying a Pimsner--Popa bound actually implies that $A$ has finite-dimensional center.
\end{remark}

Given the results of Propositions \ref{prop: Lemma C5} and \ref{prop: mu ge sum_i [e_i E e_i : e_iA]} it is natural to ask the following:

\begin{question}
Let $A\subset B$ be von Neumann algebras with finite-dimensional center, and let $E:B\to A$ be a conditional expectation 
satisfying the Pimsner--Popa bound ${E(x) \ge \mu^{-1} x}$, $\forall x \in  B_+$.
For which norm $\|\,\,\|$ on matrices do we then get the inequality
$\|\llbracket B : A \rrbracket\| \le \sqrt{\mu}\,$ ?
\end{question}

Finally, we use the previous two propositions to explain the relationship between 
index and the operations of relative commutant and of completed tensor product.

\begin{corollary}
Let $N\subset M\subset A\subset \bfB(H)$ be subalgebras with $N$ and $M$ factors and $[M:N] < \infty$.
Suppose that one of the two relative commutants $N'\cap A$ or $M'\cap A$ is a factor, and that the other one has finite-dimensional center.
In this case,
\[
\big\| \llbracket N'\cap A : M'\cap A\rrbracket \big\|\, \,\,\le\,\, \llbracket M:N\rrbracket.
\]
\end{corollary}

\begin{proof}
By Corollary \ref{cor: [B:A]=[A':B']}, we know that $[N' : M'] = [M:N]$.
Let $E':N'\to M'$ be the minimal conditional expectation from $N'$ to $M'$.
If $a\in A'\subset N'$ and $x\in N'\cap A$, then we have $aE'(x)=E'(ax)=E'(xa)=E'(x)a$, showing that $E'(x)\in M'\cap A$.
The restriction $E:=E'|_{N'\cap A}$ is therefore a conditional expectation from $N'\cap A$ to $M'\cap A$.
By the Pimsner--Popa inequality for $E'$, we know that 
\[
E(x)\ge [N' : M']^{-1} x = [M : N]^{-1} x,\quad \forall x\in N'\cap A.
\]
By Proposition \ref{prop: Lemma C5} or Proposition \ref{prop: mu ge sum_i [e_i E e_i : e_iA]},
it follows that $\|\llbracket N'\cap A : M'\cap A\rrbracket\| \le \llbracket M : N \rrbracket$.
\end{proof}

\begin{corollary}
Let $N\subset M\subset \bfB(H)$ be factors with $[M:N]<\infty$, and let $A\subset \bfB(H)$ be an algebra that commutes with $M$.
Suppose that one of the algebras $N\vee A$ and $M\vee A$ is a factor, and that the other one has finite-dimensional center.
In this case,
\[
\big\| \llbracket M\vee A:N\vee A \rrbracket \big\|\, \,\,\le\,\, \llbracket M:N\rrbracket.
\]
\end{corollary}

\begin{proof}
By the previous corollary, we have $\| \llbracket N'\cap A' : M'\cap A'\rrbracket \|\le \llbracket M:N\rrbracket$.
The result now follows from Corollary \ref{cor: [B:A]=[A':B']}, because $(M\vee A)'=M'\cap A'$ and $(N\vee A)'=N'\cap A'$.
\end{proof}


\bibliographystyle{abbrv}

\bibliography{../Files/db-cn3}

\end{document}